\documentclass{elsarticle}

\usepackage[utf8]{inputenc}
\usepackage{amsmath, amssymb, amsthm, amsfonts}
\usepackage{bm}
\usepackage{titlesec}
\usepackage{mathtools}
\usepackage{stmaryrd}
\usepackage{tikz, graphicx}
\usepackage{caption, subcaption}
\usepackage{hyperref, cleveref}
\usepackage[titletoc]{appendix}
\usepackage[normalem]{ulem}
\usepackage{geometry}
\usepackage{algorithm}
\usepackage{algpseudocode}

\usepackage{multirow}

\usetikzlibrary{positioning}

\theoremstyle{definition}

\theoremstyle{lemma}
\newtheorem{lemma}{Lemma}
\theoremstyle{theorem}

\theoremstyle{assumption}

\renewcommand{\hat}[1]{\widehat{#1}}

\newcommand{\td}[2]{\frac{{\rm d}#1}{{\rm d}{ {#2}}}} 
\newcommand{\pd}[2]{\frac{\partial#1}{\partial#2}}
\newcommand{\pdn}[3]{\frac{\partial^{#3}#1}{\partial#2^{#3}}}
\newcommand{\nor}[1]{\left\| #1 \right\|} 
\newcommand{\LRp}[1]{\left( #1 \right)} 
\newcommand{\LRs}[1]{\left[ #1 \right]} 
\newcommand{\LRb}[1]{\left| #1 \right|} 
\newcommand{\LRc}[1]{\left\{ #1 \right\}} 
\newcommand{\LRl}[1]{\left. #1 \right|} 
\newcommand{\jump}[1] {\ensuremath{\left\llbracket#1\right\rrbracket}} 
\newcommand{\avg}[1] {\ensuremath{\LRc{\!\!\LRc{#1}\!\!}}}
\newcommand{\prodmean}[1] {\ensuremath{\LRp{\!\LRp{#1}\!}}}
\newcommand{\fnt}[1]{\bm{\mathsf{ #1}}}

\newtheorem{remark}{Remark}

\date{}

\begin{document}


\begin{frontmatter}
\title{Nodal discontinuous Galerkin methods for non-ideal equations of state: pressure equilibrium preservation and entropy correction\tnoteref{tnote1}}

\tnotetext[tnote1]{Distribution Statement A: Approved for Public Release; Distribution is Unlimited. AFRL-2026-3663.}
\author[oden,ASE]{Jesse Chan}
\ead{jesse.chan@oden.utexas.edu}
\author[JGU]{Hendrik Ranocha}
\author[oden]{Raymond Park}
\author[UHH]{Joshua Lampert}
\author[NRL]{Eric Ching}
\author[AFRL]{Ayaboe Edoh}
\address[oden]{Oden Institute for Computational Engineering and Sciences, The University of Texas-Austin}
\address[ASE]{Department of Aerospace Engineering \& Engineering Mechanics, The University of Texas-Austin}
\address[JGU]{Institute of Mathematics, Johannes Gutenberg University Mainz}
\address[UHH]{Department of Mathematics, University of Hamburg}
\address[NRL]{Laboratories for Computational Physics and Fluid Dynamics, U.S. Naval Research Laboratory}
\address[AFRL]{Amentum --- U.S. Air Force Research Laboratory (Edwards)}

\begin{abstract}
Structure-preserving discontinuous Galerkin (DG) methods typically improve the robustness of high order simulations of real fluids. In addition to conservation, key structures include the preservation of pressure equilibrium and satisfaction of at least one entropy inequality. In this work, we investigate conservative discretizations using exactly pressure equilibrium conserving (EPEC) and approximately pressure equilibrium conserving (APEC) flux differencing DG formulations, as well as entropy stable formulations through the use of minimally dissipative corrections for non-ideal equations of state (EOS). 

We introduce an analysis of EPEC schemes and a new procedure for designing such fluxes based on a generalization of Tadmor's shuffle condition. We also analyze APEC DG schemes and show that the incorporation of dissipative interface penalization terms does not significantly increase pressure equilibrium errors, especially at higher orders of approximation. Finally, we observe that when combined with APEC flux differencing formulations, entropy correction improves robustness for under-resolved solutions and long-time simulations. 
 

\end{abstract}
\end{frontmatter}


\section{Introduction}

High order methods are of interest for unsteady fluid calculations due to their low dissipation and dispersion errors, but suffer from instabilities in the presence of shock waves and other under-resolved solution features. This lack of robustness can be treated using stabilization and shock capturing methods; however, such methods typically introduce heuristic parameters which must be tuned to balance robustness and accuracy. Moreover, optimal values of such parameters may not be generalizable across different problem settings. 

In this work, we focus on two mechanisms for improving robustness: pressure equilibrium conservation and entropy stability. For multi-component flows and non-ideal equations of state (EOS), spurious pressure oscillations can arise due to the ratio of specific heats varying spatially \cite{abgrall2001computations}. These spurious pressure artifacts can grow indefinitely and crash a simulation \cite{fujiwara2023fully}, and are not effectively suppressed by common artificial viscosity or shock capturing methods. While these pressure oscillations have historically been addressed using non-conservative formulations \cite{abgrall1996prevent, liu1998quasi, boyd2021diffuse}, recent work has developed conservative schemes that possess exact pressure equilibrium conserving (EPEC) or approximate pressure equilibrium conserving (APEC) properties \cite{coppola2026pep, terashima2025approximately, degrendele2025construction, ching2025conservative}. This work will introduce extensions of such techniques to nodal DG formulations.

Entropy stable discontinuous Galerkin (DG) methods also aim to address losses of robustness in discretizations of time-dependent nonlinear conservation laws, focusing instead on enforcing an entropy inequality. For example, entropy stable DG methods \cite{chen2020review, gassner2021novel} enforce a semi-discrete cell entropy inequality by constructing an entropy conservative algebraic ``flux differencing'' formulation using summation-by-parts operators and entropy conservative two-point finite volume fluxes \cite{tadmor1987numerical,tadmor2003entropy}. Adding entropy dissipation through interface fluxes then results in an entropy stable scheme. 

While entropy conservative fluxes have been well-studied for multi-component and thermally perfect gases \cite{hansen2019entropy, gouasmi2020formulation, peyvan2023high, ching2024positivity, ching2025positivity, oblapenko2025entropy}, they have only recently been extended to non-ideal equations of state (EOS) \cite{aiello2025entropy, klein2026generalized}. A real-fluid EOS is necessary to accurately model a variety of engineering applications, including propulsion systems, carbon sequestration, explosives, and emerging power generation technologies. In these environments, the fluid experiences high pressures and temperatures, often reaching supercritical/transcritical conditions at which the ideal-gas assumption breaks down. In this work, we provide an alternative means of enforcing a semi-discrete entropy inequality for a single-component non-ideal EOS fluid using entropy correction, including a variant of \cite{chan2025artificial} based on a flux-corrected transport-like scheme. Finally, we note that entropy correction approaches provide additional flexibility and portability, as they do not require the derivation of entropy conservative two-point fluxes in order to enforce a cell entropy inequality. 

Entropy corrections based on ``companion'' schemes have been introduced based on entropy conservative flux differencing formulations \cite{carpenter2014entropy, doehring2026volume}. An alternative approach is to utilize an entropy dissipative correction \cite{khalfallah1990conditions, richter1995entropy, abgrall2018general, berthon2020easy, abgrall2022reinterpretation} such as entropy correction artificial viscosity \cite{chan2025artificial, christner2025entropy}. Rather than first constructing an entropy conservative method and adding entropy dissipation, these approaches add a dissipative correction designed to enforce a cell entropy inequality. The key observation for entropy correction artificial viscosity is that the magnitude of artificial viscosity necessary to enforce a cell entropy inequality is extremely small at high orders, resulting in a minimally dissipative high order method suitable for scale-resolving and under-resolved flows.


The paper proceeds as follows. We present the governing equations and entropy stability theory in Section~\ref{sec:1}. Section~\ref{sec:ndg} reviews flux differencing nodal DG formulations, through which pressure equilibrium errors can be treated using exactly pressure equilibrium conserving (EPEC) or approximately pressure equilibrium conserving (APEC) finite volume fluxes \cite{terashima2025approximately, degrendele2025construction} discussed in Section~\ref{sec:pec}. We also adapt conditions from \cite{artiano2026affordable} to characterize pressure equilibrium conservation and derive a new EPEC flux. Section~\ref{sec:ec} discusses the enforcement of a cell entropy inequality using a version of entropy correction based on flux-corrected transport. Finally, Section~\ref{sec:4} presents representative numerical experiments for the van der Waals and Peng-Robinson EOS.

\section{Governing equations}
\label{sec:1}

We are interested in the single-component compressible Euler equations
\begin{equation}
\pd{\bm{u}}{t} + \sum_{i=1}^d \pd{\bm{f}_i(\bm{u})}{x_i}  = \bm{0},
\label{eq:ncl}
\end{equation}
where $\bm{u}$ denotes the conservative variables and $\bm{f}_i(\bm{u})$ denote the convective fluxes along the $i$th coordinate direction.
This work focuses on $d=1, 2$; for $d=2$, $\bm{u} = [\rho, \rho v_1, \rho v_2, \rho E]$ and the inviscid fluxes $\bm{f}_i$ are given by
\[
\bm{f}_1(\bm{u}) = \begin{bmatrix}
\rho v_1\\
\rho v_1^2 + p\\
\rho v_1 v_2 \\
v_1 (\rho E + p)
\end{bmatrix},
\qquad
\bm{f}_2(\bm{u}) = \begin{bmatrix}
\rho v_2\\
\rho v_1 v_2\\
\rho v_2^2 + p\\
v_2 (\rho E + p)
\end{bmatrix}.
\]
The one-dimensional compressible Euler system is obtained by setting $v_2 = 0$ and dropping the $y$-momentum equation. 

The compressible Euler system is  closed by specifying an equation of state (EOS) for the pressure. We will assume a general pressure EOS
\[
p \coloneqq p(V, T)
\]
where $V = \rho^{-1}$ denotes the specific volume and $T$ denotes the temperature. The temperature is given by a temperature-dependent internal energy model. The formulation in this work will consider more general internal energy and temperature relations $e \coloneqq e(V, T)$, where $e =E - \frac{1}{2} (v_1^2 + v_2^2)$. In the case of an ideal gas, this yields the linear relation $e = c_v T$. 

\subsection{Entropy inequality}

Vanishing viscosity solutions to \eqref{eq:ncl} can be shown \cite{dafermos2005hyperbolic, godlewski2013numerical, chen2017entropy} to satisfy a mathematical entropy inequality
\begin{equation}
\pd{S(\bm{u})}{t} + \sum_{m=1}^d \pd{F_m(\bm{u})}{x_m} \leq 0, \qquad F_m(\bm{u}) = \bm{v}(\bm{u})^T\bm{f}_m(\bm{u}) - \psi_m(\bm{u}),
\label{eq:entropy_ineq}
\end{equation}
where $S(\bm{u})$ is a convex scalar entropy function, $F_m(\bm{u})$ are the associated entropy fluxes, $\psi_m$ are the entropy potentials, and $\bm{v}(\bm{u})$ are the entropy variables
\[
\bm{v}(\bm{u}) = \pd{S}{\bm{u}}.
\]
The proof of \eqref{eq:entropy_ineq} relies on the relation $\bm{v}(\bm{u})^T \pd{\bm{f}}{\bm{u}} = \pd{\bm{F}}{\bm{u}}$, which implies through the chain rule that
\begin{equation}
\label{eq:entropychainrule}
\bm{v}(\bm{u})^T \pd{\bm{f}(\bm{u})}{\bm{x}} = \pd{\bm{F}(\bm{u})}{\bm{x}}.
\end{equation}
We assume that $S(\bm{u})$ is strictly convex, which guarantees that the mapping between conservative and entropy variables is invertible \cite{lukavcova2025note}. In this work, we utilize an integrated cell version of \eqref{eq:entropy_ineq} \cite{jiang1994cell}. Consider a closed domain $D \subset \mathbb{R}^d$ with boundary $\partial D$. Then, integrating \eqref{eq:entropy_ineq} over $D$ and applying the divergence theorem yields
\begin{equation}
\int_{D}\pd{S(\bm{u})}{t} + \int_{\partial D} \sum_{m=1}^d\LRp{\bm{v}^T\bm{f}_m(\bm{u}) - \psi_m(\bm{u})} n_m\leq 0.
\label{eq:cell_entropy_ineq}
\end{equation}

For a single fluid with a general equation of state, the compressible Euler equations are symmetrized under the entropy flux pair 
\[
S(\bm{u}) = -\rho s, \qquad F_i(\bm{u}) = -\rho s v_i,
\]
where $s$ is the specific entropy. 
The proof can be found, for example, in \cite{lukavcova2025note}, or in Section 2.1 of \cite{godlewski2013numerical}. Moreover, one can show using the thermodynamic relation $T ds = de + pdV$ that the entropy variables are 
\[
\bm{v} \coloneqq \pd{S}{\bm{u}} = T^{-1} \LRs{g - \frac{1}{2}(v_1^2 + v_2^2), v_1, v_2, -1},
\]
where $g = e + pV - Ts$ is the Gibbs free energy. We also introduce the entropy potential $\psi_i(\bm{u})$, which for general EOS can be computed via the relation \cite{klein2026generalized}
\begin{equation}
\psi_i(\bm{u}) = \bm{v}(\bm{u})^T\bm{f}_i(\bm{u}) - F_i(\bm{u}) = \frac{p v_i}{T}.
\label{eq:psi}
\end{equation}

\subsection{Pressure equilibrium}
The compressible Euler system gives rise to transport equations for velocity $v_k$,
\begin{equation}
\frac{\partial v_{k}}{\partial t}+\sum_{i=1}^{d}v_{i}\frac{\partial v_{k}}{\partial x_{i}}+\frac{1}{\rho}\frac{\partial p}{\partial x_{k}}=0,\;k=1,\ldots,d,
\end{equation}
and pressure $p$~\cite{Fed02},
\begin{equation}
\frac{\partial p}{\partial t}+\sum_{i=1}^{d}\frac{\partial\left(pv_{i}\right)}{\partial x_{i}}+\left(\rho c^{2}-p\right)\sum_{i=1}^{d}\frac{\partial v_{i}}{\partial x_{i}}=0,
\end{equation}
where $c$ is the speed of sound. In the case of uniform pressure and velocity, we have
\[
\frac{\partial v_{k}}{\partial t}=0,\qquad 
\frac{\partial p}{\partial t}=0,
\]
i.e., pressure and velocity remain constant in time. This property, known as pressure equilibrium preservation/conservation, is what we aim to achieve either exactly or approximately in the discrete setting. Failure to preserve pressure equilibrium in numerical simulations can lead to spurious pressure oscillations that pollute the solution and may cause divergence~\cite{fujiwara2023fully}. Note that these equations imply that velocity equilibrium preservation is a prerequisite for pressure equilibrium preservation.

\section{Flux differencing nodal discontinuous Galerkin formulation}
\label{sec:ndg}

For this work, we will restrict ourselves to nodal (collocation) DG methods \cite{kopriva2010quadrature}, as well as DG-like methods based on multi-dimensional summation by parts (SBP) operators \cite{hicken2016multidimensional, chen2017entropy, fernandez2019staggered}. Both discretizations expose a structure which is amenable to conservative ``flux differencing'' formulations, which we will use to address spurious pressure oscillations present in multi-component flows and non-ideal thermodynamics.

\subsection{A conservative flux differencing nodal DG formulation}

First, we introduce a flux differencing nodal DG formulation. For simplicity, we first present the formulation in 1D and refer the reader to relevant references for the multi-dimensional extension. Let $\hat{\bm{x}}, \bm{x}$ denote coordinates on the reference element $\hat{D} = [-1,1]$ and a physical element $D^k$, respectively. We assume that the solution is approximated by a degree $N$ nodal (Lagrange) basis on the reference element 
\[
\bm{u}_h(x(\hat{x}), t) = \sum_{j=1}^{N+1} \fnt{u}_j(t) \ell_j(\hat{x}).
\]
In one dimension, the Lagrange basis functions are defined on $(N+1)$-point Legendre-Gauss-Lobatto (LGL) quadrature nodes $\hat{x}_i$ with associated quadrature weights $w_i > 0$.

The flux differencing formulation on each element $D^k$ can be written in terms of a summation-by-parts (SBP) differentiation matrix $\hat{\fnt{Q}}$, a boundary interpolation matrix $\fnt{E}$, and a diagonal reference mass matrix $\hat{\fnt{M}}$ on the reference element $\hat{D}$ \cite{gassner2013skew}. These matrices are defined as follows:
\[
\hat{\fnt{M}}_{ij} = w_i \delta_{ij},
\qquad
\hat{\fnt{Q}}_{ij} = \int_{-1}^1 \ell_i \pd{\ell_j}{\hat{x}},
\qquad
\fnt{E} = \begin{bmatrix}
1 & 0 & \ldots & 0\\
0 & \ldots & 0 & 1
\end{bmatrix}.
\]
Note that the integral in $\hat{\fnt{Q}}_{ij}$ is exact when calculated using LGL quadrature.
These operators also satisfy the following summation by parts property
\[
\hat{\fnt{Q}}+\hat{\fnt{Q}}^T = \fnt{E}^T\fnt{B}\fnt{E}, \qquad \fnt{B} = \begin{bmatrix}-1 & 0\\ 0 & 1\end{bmatrix}, \qquad \hat{\fnt{Q}}\fnt{1} = \fnt{0}.
\]
Let $\bm{f}^*_n(\bm{u}_L, \bm{u}_R) \approx \sum_{i=1}^d \bm{f}_i(\bm{u}) \bm{n}_i$ denote an interface numerical flux along the normal direction (in 1D, the normal vector reduces to $\bm{n} = \pm 1$), and let $\bm{f}_{\rm vol}(\bm{u}_L, \bm{u}_R, \bm{n}_{LR})$ denote a volume flux which is consistent and skew-symmetric in the sense that
\[
\bm{f}_{\rm vol}(\bm{u}, \bm{u}, \bm{n}) = \sum_{i=1}^d\bm{f}_i(\bm{u}) \bm{n}_i,
\qquad
\bm{f}_{\rm vol}(\bm{u}_L, \bm{u}_R, \bm{n}_{LR}) = -\bm{f}_{\rm vol}(\bm{u}_R, \bm{u}_L, -\bm{n}_{LR}).
\]
The flux differencing formulation over each element $D^k$ is then given by the following:
\begin{gather}
\label{eq:dgform}
J\hat{\fnt{M}}_{ii}\td{\fnt{u}_i}{t} + \sum_{j=1}^{N+1} \nor{\fnt{n}_{ij}}\bm{f}_{\rm vol}\LRp{\fnt{u}_i, \fnt{u}_j, \frac{\fnt{n}_{ij}}{\nor{\fnt{n}_{ij}}}} + \sum_{j=1}^2\fnt{E}_{ji}\bm{f}^*_n(\fnt{u}_{f, j}, \fnt{u}_{f, j}^+)  = 0\\
\fnt{n}_{ij} = \begin{cases}
(\hat{\fnt{Q}}_{ij} - \hat{\fnt{Q}}_{ji}) & i \neq j \\
0 & \text{otherwise}
\end{cases}
\qquad \fnt{u}_f = \fnt{E}\fnt{u}, \nonumber
\end{gather}
where $\fnt{u}_f$ denotes the solution values at face nodes, $\fnt{u}_f^+$ denotes the exterior value of the solution (either on a neighboring element or defined through the imposition of boundary conditions), $J$ is the Jacobian of the mapping from the reference to physical element, and $\fnt{n}_{ij}$ is an algebraic term intended to behave like a normal vector in finite volume formulations. Note that the formulation is conservative; summing \eqref{eq:dgform} over all nodes in an element yields that the scaled solution average is balanced by interface flux contributions, with the volume term vanishing due to the skew-symmetry of $\fnt{n}_{ij}$ and symmetry of the volume flux. Moreover, one can show that if $\bm{f}_{\rm vol}$ is of the form
\[
\bm{f}_{\rm vol}(\bm{u}_i, \bm{u}_j, \bm{n}) = \sum_{i=1}^d\bm{f}_{i, {\rm vol}}(\bm{u}_i, \bm{u}_j) \bm{n}_i
\]
where $\bm{f}_{i, {\rm vol}}(\bm{u}_i, \bm{u}_j)$ is symmetric and consistent, then \eqref{eq:dgform} is conservative and high order accurate as well \cite{chen2017entropy, crean2018entropy}. Note that $\bm{f}_{i, {\rm vol}}(\bm{u}_i, \bm{u}_j)$ does not contain the normal $\bm{n}$ as an argument, and is used to refer to a non-dissipative two-point flux. 

The extension of \eqref{eq:dgform} to multiple dimensions is straightforward through either a tensor product construction or through multi-dimensional SBP operators. We refer the reader to \cite{chen2017entropy, crean2018entropy, wu2021positivity, lin2023high} for more details on multi-dimensional flux differencing formulations.

\section{Conservative pressure-equilibrium-conserving fluxes}
\label{sec:pec}


The flux differencing formulation \eqref{eq:dgform} can be combined with entropy conservative volume fluxes to construct entropy conservative or entropy stable schemes. However, they can also be combined with pressure equilibrium conserving fluxes \cite{ranocha2020entropy, ranocha2021preventing}. In this section, we review conditions under which a two-point flux is exactly pressure-equilibrium conserving (EPEC) and derive a new EPEC flux based on these conditions. 

A condition for exact pressure equilibrium conservation is given in \cite{fujiwara2023fully, terashima2025approximately} for the multi-species case. For a single fluid, this condition reduces to
\begin{equation}
    \LRl{\pd{\rho e}{\rho}}_{p, \bm{u}_i}
    \LRp{ \rho_{i + 1/2} - \rho_{i - 1/2} }
    =
    \rho e_{i + 1/2} - \rho e_{i - 1/2},
\label{eq:pe_condition}
\end{equation}
where $\rho_{i \pm 1/2}$ and $\rho e_{i \pm 1/2}$ are consistent approximations of the density and the internal energy density at the $i \pm 1/2$ interface, respectively.
Adapted to the single-species case and our notation, Terashima et al.~\cite{terashima2025approximately} write
\begin{quote}
    However, as recognized in Eq.~\eqref{eq:pe_condition}, the
    asymmetry associated with the partial derivative term
    $\LRl{\pd{\rho e}{\rho}}_{p, \bm{u}_i}$
    defined at $i$ means that no locally-defined half-point values
    satisfy Eq.~\eqref{eq:pe_condition} for all cells.
\end{quote}
This conclusion does not hold in general. 
Indeed, the condition \eqref{eq:pe_condition} is similar to conditions for entropy conservative numerical fluxes developed by Tadmor \cite{tadmor1987numerical, tadmor2003entropy}; the partial derivative $\LRl{\pd{\rho e}{\rho}}_{p, \bm{u}_i}$ takes the role of the entropy variables, the density flux $\rho_{i \pm 1/2}$ takes the role of an entropy-conservative flux, and the internal energy density flux $\rho e_{i \pm 1/2}$ takes the role of the entropy flux.

Define the average and jump as $\avg{a} = \frac{1}{2}(a_L + a_R)$ and $\jump{a} = a_L - a_R$. We first state a generalization of Tadmor's entropy analysis from \cite{artiano2026affordable}. 
\begin{lemma}
\label{lemma:EPEC}
    Let $\bm{v}_i = \bm{v}(\bm{u}_i)$ and $\bm{f}_{i \pm 1/2}$ be consistent and symmetric two-point approximations of $\bm{f}$, i.e., $\bm{f}_{i + 1/2} = \bm{f}_{\rm avg}(\bm{u}_{i}, \bm{u}_{i+1})$ satisfying $\bm{f}_{\rm avg}(\bm{u}, \bm{u}) = \bm{f}(\bm{u})$ and $\bm{f}_{\rm avg}(\bm{u}_L, \bm{u}_R) = \bm{f}_{\rm avg}(\bm{u}_R, \bm{u}_L)$. Additionally, define $\bm{F}_{\rm avg}$
    \begin{equation}
        \bm{F}_{\rm avg} = \avg{\bm{F}} + \avg{\bm{v}} \cdot \bm{f}_{\rm avg} - \avg{\bm{v} \cdot \bm{f}}.
        \label{eq:generalized_tadmor_entropy_flux}
    \end{equation}
    Then, 
    \begin{equation}
        \bm{v}_i \cdot \LRp{\bm{f}_{i + 1/2} - \bm{f}_{i - 1/2}}
        = \LRp{\bm{F}_{i + 1/2} - \bm{F}_{i - 1/2}}
        \label{eq:generalized_shuffle_condition}
    \end{equation}
    is satisfied if and only if
    \begin{equation}
        \jump{\bm{v}} \cdot \bm{f}_{\rm avg} = \jump{\bm{v} \cdot \bm{f} - \bm{F}}.
        \label{eq:generalized_tadmor_condition}
    \end{equation}
\end{lemma}
\begin{proof}
    This proof follows essentially the same steps as Tadmor's original proof \cite{tadmor1987numerical}, but without the assumption that $\bm{v}$ and $\bm{F}$ are derived from an entropy.
    First, we show that \eqref{eq:generalized_tadmor_condition} is necessary.
    Setting $\bm{u}_{i} = \bm{u}_{i + 1} = \bm{u}_R$ and $\bm{u}_{i - 1} = \bm{u}_L$ in \eqref{eq:generalized_shuffle_condition} yields due to consistency
    \begin{equation*}
        \bm{v}_R \cdot \LRp{\bm{f}(\bm{u}_R) - \bm{f}_{\rm avg}(\bm{u}_L, \bm{u}_R)}
        = \LRp{\bm{F}(\bm{u}_R) - \bm{F}_{\rm avg}(\bm{u}_L, \bm{u}_R)}.
    \end{equation*}
    Similarly, setting $\bm{u}_{i - 1} = \bm{u}_{i} = \bm{u}_L$ and $\bm{u}_{i + 1} = \bm{u}_R$ in \eqref{eq:generalized_shuffle_condition} yields
    \begin{equation*}
        \bm{v}_L \cdot \LRp{\bm{f}_{\rm avg}(\bm{u}_L, \bm{u}_R) - \bm{f}(\bm{u}_L)}
        = \LRp{\bm{F}_{\rm avg}(\bm{u}_L, \bm{u}_R) - \bm{F}(\bm{u}_L)}.
    \end{equation*}
    Multiplying both sides by $-1$ and adding the two equations above yields \eqref{eq:generalized_tadmor_condition}
    \begin{equation*}
        \LRp{\bm{v}_L - \bm{v}_R} \cdot \bm{f}_{\rm avg}(\bm{u}_L, \bm{u}_R)
        =
        \LRp{\bm{v}_L \cdot \bm{f}(\bm{u}_L) - \bm{F}(\bm{u}_L)}
        - \LRp{\bm{v}_R \cdot \bm{f}(\bm{u}_R) - \bm{F}(\bm{u}_R)}.
    \end{equation*}
    Next, we show that \eqref{eq:generalized_tadmor_condition} is sufficient. Expanding $\bm{v}_i \cdot \LRp{\bm{f}_{i + 1/2} - \bm{f}_{i - 1/2}}$ and using the definition of $\bm{F}_{i \pm 1/2}$ in \eqref{eq:generalized_tadmor_entropy_flux} yields
    \begin{equation*}
    \begin{aligned}
        &\quad
        \bm{v}_i \cdot \LRp{\bm{f}_{i + 1/2} - \bm{f}_{i - 1/2}}
        \\
        &=
        \LRp{\frac{\bm{v}_{i} + \bm{v}_{i + 1}}{2} - \frac{\bm{v}_{i + 1} - \bm{v}_{i}}{2}} \cdot \bm{f}_{i + 1/2}
        -
        \LRp{\frac{\bm{v}_{i - 1} + \bm{v}_{i}}{2} + \frac{\bm{v}_{i} - \bm{v}_{i - 1}}{2}} \cdot \bm{f}_{i - 1/2}
        \\
        &=
        \frac{\bm{v}_{i} + \bm{v}_{i + 1}}{2} \cdot \bm{f}_{i + 1/2}
        - \frac{\bm{v}_{i + 1} \cdot \bm{f}_{i + 1} - \bm{v}_{i} \cdot \bm{f}_{i}}{2}
        + \frac{\bm{F}_{i + 1} - \bm{F}_{i}}{2}
        \\
        &\quad
        - \frac{\bm{v}_{i - 1} + \bm{v}_{i}}{2} \cdot \bm{f}_{i - 1/2}
        - \frac{\bm{v}_{i} \cdot \bm{f}_{i} - \bm{v}_{i - 1} \cdot \bm{f}_{i - 1}}{2}
        + \frac{\bm{F}_{i} - \bm{F}_{i - 1}}{2}
        \\
        &=
        \frac{\bm{v}_{i} + \bm{v}_{i + 1}}{2} \cdot \bm{f}_{i + 1/2}
        - \frac{\bm{v}_{i + 1} \cdot \bm{f}_{i + 1} + \bm{v}_{i} \cdot \bm{f}_{i}}{2}
        + \frac{\bm{F}_{i + 1} + \bm{F}_{i}}{2}
        \\
        &\quad
        - \frac{\bm{v}_{i - 1} + \bm{v}_{i}}{2} \cdot \bm{f}_{i - 1/2}
        + \frac{\bm{v}_{i} \cdot \bm{f}_{i} + \bm{v}_{i - 1} \cdot \bm{f}_{i - 1}}{2}
        - \frac{\bm{F}_{i} + \bm{F}_{i - 1}}{2}
        \\
        &=
        \bm{F}_{i + 1/2} - \bm{F}_{i - 1/2}.
    \end{aligned}
    \end{equation*}
\end{proof}
Thus, \eqref{eq:pe_condition} is satisfied if and only if
\begin{align}
    \jump{\LRl{\pd{\rho e}{\rho}}_{p}} \rho_{\rm avg}
    &=
    \jump{\LRl{\pd{\rho e}{\rho}}_{p} \rho} - \jump{\rho e}, \label{eq:equiv_pe_condition}
    \\
    (\rho e)_{\rm avg}
    &=
    \avg{\rho e}
    + \avg{\LRl{\pd{\rho e}{\rho}}_{p}} \rho_{\rm avg}
    - \avg{\LRl{\pd{\rho e}{\rho}}_{p} \rho}. \label{eq:internal_energy_avg}
\end{align}
Assuming $f_{\rho} = \rho_{\rm avg} \avg{v_1}$ and a kinetic energy preserving (KEP) structure \cite{kuya2018kinetic, ranocha2021preventing, coppola2026pep}, the full expression for the EPEC two-point flux is given by 
\begin{equation}
    \bm{f}_{\rm EPEC}(\bm{u}_L,\bm{u}_R)
    =
    \begin{bmatrix}
        f_{\rho} \\
        f_{\rho}\,\avg{v_1} + \avg{p} \\
        (\rho e)_{\rm avg}\,\avg{v_1}
        + f_{\rho}\,\frac{1}{2}\prodmean{v_1^2}
        + \prodmean{p, v_1}
    \end{bmatrix},
    \label{eq:ranocha_EPEC}
\end{equation}
where the product means of $v_1^2$ and $pv_1$ are given by
\[
\prodmean{v_1^2} = v_{1,L} v_{1,R},
\qquad
\prodmean{p, v_1} = \frac{p_L v_{1,R} + p_R v_{1,L}}{2}.
\]

\begin{remark}
\label{remark:PEC_condition}
    For an ideal gas EOS, $\LRl{\pd{\rho e}{\rho}}_{p} = 0$. Thus, the equivalent conditions \eqref{eq:pe_condition} and \eqref{eq:equiv_pe_condition}, \eqref{eq:internal_energy_avg} are satisfied trivially. In particular, the simpler analysis of \cite{ranocha2021preventing} applies in this case. Additionally, for a stiffened gas, $\LRl{\pd{\rho e}{\rho}}_{p} = q$ where $q$ is a constant, and $\jump{\rho e} = q\jump{\rho}$. Thus, \eqref{eq:pe_condition} and \eqref{eq:equiv_pe_condition} are also trivially satisfied for a stiffened gas.
\end{remark}

\subsection{A new EPEC flux based on density average correction}

Conditions \eqref{eq:equiv_pe_condition} suggest a way to construct an exactly PEC flux by correcting the density average. Expanding the jump leads to the condition
\begin{equation*}
    \jump{\LRl{\pd{\rho e}{\rho}}_{p}} \rho_{\rm avg}
    =
    \jump{\LRl{\pd{\rho e}{\rho}}_{p}} \avg{\rho}
    + \avg{\LRl{\pd{\rho e}{\rho}}_{p}} \jump{\rho}
    - \jump{\rho e}.
\end{equation*}
Solving for $\rho_{\rm avg}$ above yields
\begin{equation}
    \rho_{\rm avg}
    =
    \avg{\rho}
    - \frac{
        \jump{\rho e}
        - \avg{\LRl{\pd{\rho e}{\rho}}_{p}} \jump{\rho}
    }{
        \jump{\LRl{\pd{\rho e}{\rho}}_{p}}
    }.
    \label{eq:EPEC_no_consistency_term}
\end{equation}
This recovers the exact PEC correction introduced by \cite{degrendele2025construction} for the single-fluid case. Note also that the APEC flux derived in \cite{terashima2025approximately} uses the same $(\rho e)_{\rm avg}$ approximation, but uses only the arithmetic average for $\rho_{\rm avg}$. 

We note that several modifications are possible without compromising PEC. For example, since this expression needs to hold only for constant $p$, adding $\jump{p}$ to the numerator retains the PEC property, and the following flux is also PEC:
\begin{equation}
    \rho_{\rm avg}
    =
    \avg{\rho}
    - \frac{
        \jump{\rho e}
        - \avg{\LRl{\pd{\rho e}{\rho}}_{p}} \jump{\rho}
        - \avg{\LRl{\pd{\rho e}{p}}_{\rho}} \jump{p}
    }{
        \jump{\LRl{\pd{\rho e}{\rho}}_{p}}
    }.
    \label{eq:correction_EPEC}
\end{equation}
This additional term is included to more closely mimic the mean value theorem $\jump{\rho e} = \pd{\rho e}{\rho}_{\rho^*} \jump{\rho} + \pd{\rho e}{p}_{p^*}\jump{p}$, where $\rho^*, p^*$ are intermediate states. 

We note that we do not need to add a correction term to the arithmetic density average if the denominator vanishes (since we are only interested in the condition when multiplied by the denominator). Moreover, if pressure and density are constant, both the numerator and denominator vanish. Thus, the resulting flux can be consistently computed near states which satisfy pressure equilibrium. 

In this work, we treat the singularity which occurs in the denominator of $\rho_{\rm avg}$ using the simple switch suggested in \cite{coppola2026pep}
\begin{equation}
\rho_{\rm avg} = 
\begin{cases}
\avg{\rho} & \LRb{\jump{\LRl{\pd{\rho e}{\rho}}_{p}}} < tol\\
\avg{\rho}
    - \frac{
        \jump{\rho e}
        - \avg{\LRl{\pd{\rho e}{\rho}}_{p}} \jump{\rho}
        - \avg{\LRl{\pd{\rho e}{p}}_{\rho}} \jump{p}
    }{
        \jump{\LRl{\pd{\rho e}{\rho}}_{p}}
    } & \text{otherwise}
\end{cases}.
\label{eq:switch_regularization}
\end{equation}
In this work, $tol$ is arbitrarily taken to be $100\epsilon_{\rm mach}$. 

\begin{remark}
The correction term can also be computed using the following regularized density average
\begin{equation}
\begin{gathered}
    \rho_{\rm avg}
    =
    \avg{\rho} - \frac{a b}{b^2 + \epsilon},
    \\
    a = \jump{\rho e}
        - \avg{\LRl{\pd{\rho e}{\rho}}_{p}} \jump{\rho}
        - \avg{\LRl{\pd{\rho e}{p}}_{\rho}} \jump{p}, \qquad
    b = \jump{\LRl{\pd{\rho e}{\rho}}_{p}},
    \label{eq:correction_EPEC_regularized}
\end{gathered}
\end{equation}
and $\epsilon > 0$ in the order of machine accuracy. Up to regularization, the combination of the corrected density average together with the corresponding internal energy density average \eqref{eq:internal_energy_avg} ensures that pressure equilibria are preserved exactly. We observe similar behavior for both the sharp switch and this regularized ratio in numerical experiments. 
\end{remark}

\begin{remark}
The quantity $\LRl{\pd{\rho e}{\rho}}_{p}$ appears in formulas for both EPEC and APEC fluxes, and can be calculated purely in terms of thermodynamic quantities and pressure derivatives via
\begin{gather*}
\LRl{\pd{\rho e}{\rho}}_{p} = \rho c_v \LRl{\pd{p}{T}}_V^{-1}  \LRl{\pd{p}{V}}_T \frac{V}{\rho} + \LRl{\pd{\rho e}{\rho}}_T\\
\LRl{\pd{\rho e }{\rho}}_T = e - \rho \LRl{\pd{e}{V}}_T \frac{V}{\rho}, \qquad \LRl{\pd{e}{V}}_T = T \LRl{\pd{p}{T}}_V - p(V,T),
\end{gather*}
where $c_v = \LRl{\pd{e}{T}}_V$ is the specific heat at constant volume.
\end{remark}

\subsection{Connection to the EPEC flux of Coppola, De Michele, and Aiello \cite{coppola2026pep}}
\label{sec:EPEC_coppola}

Remarkably, the correction-based EPEC flux \eqref{eq:correction_EPEC} is closely related to an EPEC two-point flux recently introduced by
Coppola, De Michele, and Aiello \cite{coppola2026pep} using a similar generalization of Tadmor's entropy-variable approach. The derivation is based on defining 
\[
\chi \coloneqq \rho^2 \LRl{\pd{e}{\rho}}_{p} = \rho \LRl{\pd{\rho e}{\rho}}_{p} - \rho e
\]
where the latter identity follows from the product rule 
\begin{equation}
\LRl{\pd{\rho e}{\rho}}_{p} = \rho\LRl{\pd{e}{\rho}}_{p} + e. 
\label{eq:product_rule_coppola}
\end{equation}
Recall from Lemma~\ref{lemma:EPEC} that an exactly pressure-equilibrium-conserving (EPEC) two-point flux must satisfy \eqref{eq:equiv_pe_condition}
\begin{equation}
    \jump{\LRl{\pd{\rho e}{\rho}}_{p}} \rho_{\rm avg}
    =
    \jump{\LRl{\pd{\rho e}{\rho}}_{p} \rho} - \jump{\rho e}   =  \jump{\chi}.
    \label{eq:chi_pep_identity_target}
\end{equation}
Thus \eqref{eq:chi_pep_identity_target} is satisfied if $\rho_{\rm avg}$ is given by
\begin{equation}
    \rho_{\rm avg} = \frac{\jump{\chi}}{\jump{\LRl{\pd{\rho e}{\rho}}_{p}}}.
    \label{eq:coppola_rho_avg}
\end{equation}
The authors derive the associated internal energy flux as
\begin{align*}
(\rho e)_{\rm avg}\avg{v_1} &= \avg{\LRl{\pd{\rho e}{\rho}}_{p}} f_{\rho} - \avg{v_1} \avg{\rho^2 \LRl{\pd{e}{\rho}}_{p}} \\
&= \underbrace{\LRp{
\avg{\LRl{\pd{\rho e}{\rho}}_{p}} \rho_{\rm avg} - \avg{\rho^2 \LRl{\pd{e}{\rho}}_p} 
}}_{(\rho e)_{\rm avg}}\avg{v_1}.
\end{align*}
Noting that $\rho^2\LRl{\pd{e}{\rho}}_{p} = \rho\LRl{\pd{\rho e}{\rho}}_{p} - \rho e$ recovers the corrected internal energy average of \eqref{eq:ranocha_EPEC}:
\begin{align*}
(\rho e)_{\rm avg} = \avg{\LRl{\pd{\rho e}{\rho}}_{p}} \rho_{\rm avg} - \avg{\rho^2 \LRl{\pd{e}{\rho}}_p} 
&= \avg{\LRl{\pd{\rho e}{\rho}}_{p}} \rho_{\rm avg} - \avg{\rho\LRl{\pd{\rho e}{\rho}}_{p} - \rho e}\\
&= \avg{\rho e} + \avg{\LRl{\pd{\rho e}{\rho}}_{p}} \rho_{\rm avg} - \avg{\rho\LRl{\pd{\rho e}{\rho}}_{p}}.
\end{align*}
The EPEC flux of Coppola, De Michele, and Aiello is then given by
\begin{align}
    \bm{f}_{\rm EPEC}(\bm{u}_L,\bm{u}_R)
    =
    \begin{bmatrix}
        f_{\rho} \\
        f_{\rho}\,\avg{v_1} + \avg{p} \\
        (\rho e)_{\rm avg}\,\avg{v_1}
        + f_{\rho}\,\frac{1}{2}\prodmean{v_1^2}
        + \prodmean{p, v_1}
    \end{bmatrix}, \qquad f_{\rho} = \rho_{\rm avg}\,\avg{v_1}, \label{eq:coppola_EPEC} \\[.5em]
    \rho_{\rm avg} = \frac{\jump{\chi}}{\jump{\LRl{\pd{\rho e}{\rho}}_{p}}}, \qquad (\rho e)_{\rm avg} = \avg{\rho e} + \avg{\LRl{\pd{\rho e}{\rho}}_{p}}\,\rho_{\rm avg}  - \avg{\LRl{\pd{\rho e}{\rho}}_{p}\rho}. 
    \nonumber
\end{align}
Note that this EPEC flux is identical to the correction-based EPEC flux \eqref{eq:ranocha_EPEC} except for the definition of $\rho_{\rm avg}$. We also utilize the same singularity treatment \eqref{eq:switch_regularization}.

By construction, \eqref{eq:coppola_EPEC} satisfies the exact PEC conditions \eqref{eq:equiv_pe_condition}--\eqref{eq:internal_energy_avg} of Lemma~\ref{lemma:EPEC} whenever the denominator $\jump{\LRl{\pd{\rho e}{\rho}}_{p}}\neq 0$. Moreover, one can show that the intermediate density average introduced previously \eqref{eq:EPEC_no_consistency_term} is exactly identical to $\rho_{\rm avg}$ in  \eqref{eq:coppola_rho_avg} and \eqref{eq:coppola_EPEC} using the identity $\avg{\rho}\jump{\LRl{\pd{\rho e}{\rho}}_{p}} + \avg{\LRl{\pd{\rho e}{\rho}}_{p}}\jump{\rho}
    =
    \jump{\LRl{\pd{\rho e}{\rho}}_{p}\rho}$. 
Thus, differences in behavior between the two EPEC fluxes \eqref{eq:coppola_EPEC} and \eqref{eq:correction_EPEC} can be attributed to the inclusion of the additional consistency term $\avg{\LRl{\pd{\rho e}{p}}_{\rho}} \jump{p}$ in \eqref{eq:correction_EPEC}. Numerical experiments in Section~\ref{sec:comparison_EPEC} suggest that including this additional term improves the maximum stable time-step of the resulting simulation. 

%

\subsection{Approximately pressure-equilibrium conserving (APEC) fluxes}
\label{sec:APEC}

Both EPEC fluxes \eqref{eq:ranocha_EPEC} and \eqref{eq:coppola_EPEC} lead to density averages which depend on additional thermodynamic quantities. This may not be desirable \cite{derigs2017novel, ranocha2018comparison, ranocha2021preventing, ranocha2022note}. Decoupling the dependence of $\rho_{\rm avg}$ from other thermodynamic quantities results in a loss of the exact PEC property, but has been observed to improve robustness for highly non-ideal EOS \cite{terashima2025approximately, coppola2026pep}. 

Taking $\rho_{\rm avg} = \avg{\rho}$ reduces both \eqref{eq:ranocha_EPEC} and \eqref{eq:coppola_EPEC} to the approximately pressure-equilibrium conserving (APEC) flux of \cite{terashima2025approximately}: 
\begin{equation}
\label{eq:APEC_KEEP}
\bm{f}_{\rm APEC}(\bm{u}_L, \bm{u}_R) = \begin{bmatrix}
\avg{\rho}\avg{v_1}\\
\avg{\rho}\avg{v_1}^2 + \avg{p}\\
\avg{\rho} \avg{v_1} \frac{1}{2}\prodmean{v_1^2} + \prodmean{p, v_1} + (\rho e)_{\rm avg}\avg{v_1}
\end{bmatrix},
\end{equation}
where for a single-fluid system $(\rho e)_{\rm avg}$ is 
\[
(\rho e)_{\rm avg} = \avg{\rho e} - \frac{1}{4} \LRp{\LRl{\pd{\rho e}{\rho}}_{p_R, \bm{u}_R} - \LRl{\pd{\rho e}{\rho}}_{p_L, \bm{u}_L}} \LRp{\rho_R - \rho_L},
\]
whose equivalence to \eqref{eq:internal_energy_avg} follows directly from the identity
\[
\avg{ab} - \avg{a}\avg{b} = \frac{1}{4}\jump{a}\jump{b}.
\]
\begin{remark}
While Terashima, Ly, and Ihme combine the APEC modification with the KEEP flux \cite{kuya2018kinetic, tamaki2022comprehensive}, they note that this is not strictly necessary, as only the modification of the internal energy average is necessary for the APEC property.
\end{remark}


\section{Enforcing a cell entropy inequality via entropy correction}
\label{sec:ec}

In this work, we utilize the flux differencing formulation \eqref{eq:dgform} to construct a conservative scheme with enhanced pressure-equilibrium-conservation properties; however, the resulting scheme does not satisfy an entropy inequality. To enforce a discrete form of \eqref{eq:cell_entropy_ineq}, we utilize an entropy correction approach. This work considers a flux-corrected transport (FCT) approach to entropy correction, which is a simplification of the knapsack limiting approach introduced in \cite{christner2025entropy, christner2025entropyfd}. The entropy correction adds a minimal entropy dissipative correction term to \eqref{eq:dgform} in order to enforce a semi-discrete cell entropy inequality \eqref{eq:cell_entropy_ineq}. It was shown in \cite{chan2025artificial} that the entropy residuals used to define these correction terms can be expressed as the product of two $L^2$ best approximation errors. As a result, the magnitude of these correction terms is often significantly smaller than the error in the solution. 

The entropy correction used in this work is minimally diffusive and contact-preserving \cite{van2026choice}, and both theory and numerical experiments indicate that the magnitude of the correction is orders of magnitude smaller than the approximation error for both smooth and discontinuous solutions \cite{chan2025artificial}. As a result, these entropy correction methods are designed to address instabilities related to aliasing and high-wavenumber errors, but do not suppress Gibbs-type oscillations around stronger shocks. 

We note that the enforcement of a semi-discrete entropy inequality can also be performed using artificial viscosity  (AV) \cite{chan2025artificial}; see \ref{sec:SG} for examples. However, since the numerical experiments in this work are purely inviscid, the FCT-based entropy correction is more efficient than the AV-based approach, as it avoids the introduction of a viscous DG discretization. We observe that all one-dimensional results are virtually identical between the FCT-based and AV-based entropy corrections, so we do not include AV-based entropy correction results for brevity. 

\begin{remark}
The FCT-based and AV-based entropy correction approaches differ with respect to their flexibility. The FCT-based entropy correction used in this paper assumes a nodal collocation formulation and summation-by-parts operators; in contrast, the AV-based entropy correction of \cite{chan2025artificial} can be extended to arbitrary choices of basis and quadrature. We defer a careful comparison of these two approaches to a future paper. 
\end{remark}


\subsection{A semi-discrete entropy inequality and volume entropy residual}

To derive an entropy correction method, we first specify the semi-discrete entropy inequality we wish to enforce. We illustrate this by deriving an entropy evolution equation for the 1D formulation; the derivation is more technical but similar in higher dimensions. Testing the formulation \eqref{eq:dgform} with the nodal interpolant of the entropy variables $\bm{v}(\fnt{u})$ and summing over all nodes in $D^k$ yields
\begin{gather}
\sum_i J\hat{\fnt{M}}_{ii}\bm{v}(\fnt{u}_i)^T\td{\fnt{u}_i}{t} + \sum_{i, j=1}^{N+1} \bm{v}(\fnt{u}_i)^T\nor{\fnt{n}_{ij}}\bm{f}_{\rm vol}\LRp{\fnt{u}_i, \fnt{u}_j, \frac{\fnt{n}_{ij}}{\nor{\fnt{n}_{ij}}}} + \nonumber\\
\sum_{i=1}^{N+1}\sum_{j=1}^2\fnt{E}_{ji}\bm{v}(\fnt{u}_i)^T\bm{f}^*_n(\fnt{u}_{f, j}, \fnt{u}_{f, j}^+)= 0.
\end{gather}
Assuming continuity and the chain rule in time, $\bm{v}(\fnt{u}_i)^T\td{\fnt{u}_i}{t} = \pd{S(\fnt{u}_i)}{\bm{u}}\td{\fnt{u}_i}{t} = \td{S(\fnt{u}_i)}{t}$. Noting that the diagonal entries of the mass matrix $\hat{\fnt{M}}_{ii}$ are quadrature weights, we have that
\begin{equation}
\sum_i J\hat{\fnt{M}}_{ii}\bm{v}(\fnt{u}_i)^T\td{\fnt{u}_i}{t} = \sum_i J w_i\td{S(\fnt{u}_i)}{t} \approx \int_{D^k} \pd{S(\bm{u})}{t}.
\label{eq:dSdt_derivation}
\end{equation}
We now introduce the volume entropy residual $\delta_k(\bm{u}_h)$   on $D^k$
\begin{equation}
\delta_k(\bm{u}_h) = \sum_{i, j=1}^{N+1} \bm{v}(\fnt{u}_i)^T\nor{\fnt{n}_{ij}}\bm{f}_{\rm vol}\LRp{\fnt{u}_i, \fnt{u}_j, \frac{\fnt{n}_{ij}}{\nor{\fnt{n}_{ij}}}} + \psi(\fnt{u}_{N+1}) - \psi(\fnt{u}_1).
\end{equation}
Using the sparsity of $\fnt{E}$ yields the cell entropy evolution equation for the formulation \eqref{eq:dgform} in 1D 
\begin{equation}
\label{eq:entropyequation}
\sum_i J\hat{\fnt{M}}_{ii}\td{S(\fnt{u}_i)}{t} + \delta_k(\bm{u}_h) + 
F^*_n(\fnt{u}_{N+1}) + F^*_n(\fnt{u}_1) = 0, 
\end{equation}
where for conciseness, we have introduced a numerical version of the entropy flux $F^*_n(\fnt{u})$ \eqref{eq:psi}
\[
F^*_n(\fnt{u}) = \bm{v}(\fnt{u})^T\bm{f}^*_n(\fnt{u}, \fnt{u}^+) - \psi(\fnt{u})n.
\]
Recall that the 1D cell entropy inequality \eqref{eq:cell_entropy_ineq} we wish to enforce is
\[
\int_{D^k}\pd{S(\bm{u})}{t} + \int_{\partial D^k} F(\bm{u}) n \leq 0,
\]
whose discrete form is 
\[
 \sum_i J w_i\td{S(\fnt{u}_i)}{t} + F_n^*(\bm{u}_{N+1})+ F_n^*(\bm{u}_{1}) \leq 0.
\]
Since $ \sum_i J w_i\td{S(\fnt{u}_i)}{t}\approx \int_{D^k} \pd{S(\bm{u})}{t}$ already appears in \eqref{eq:entropyequation} and since $\int_{\partial D^k} \bm{v}(\bm{u})^T\bm{f}^*_n$ is a consistent approximation to $\bm{v}^T\bm{f}(\bm{u})$ in the entropy inequality, \eqref{eq:entropyequation} is a consistent approximation of \eqref{eq:cell_entropy_ineq} if
\begin{equation}
\delta_k(\bm{u}_h) \geq 0.
\label{eq:vol_entropy_residual}
\end{equation}
Note that $\delta_k(\bm{u}_h)$ can be interpreted as a quadrature approximation of
\[
\int_{D^k} -\pd{\bm{v}(\bm{u})}{x}^T\bm{f}(\bm{u}) + \int_{\partial D^k}\psi(\bm{u}) n.
\]
In general, the sign of $\delta_k(\bm{u}_h)$ is indeterminate, and can lead to violation of the entropy inequality \eqref{eq:cell_entropy_ineq}. The entropy correction approaches in this work add an entropy dissipative correction term to enforce $\delta_k(\bm{u}_h) \geq 0$. We also note that this definition of the volume entropy residual can also be derived directly from the more general modal formulation in \cite{chan2025artificial} involving the $L^2$ projection of the entropy variables.

\subsection{Entropy correction via flux corrected transport}
\label{sec:entropy_correction}
We consider a flux corrected transport (FCT) version of entropy correction in this work, where a high and low order method are blended together using an element-wise blending coefficient designed to satisfy a cell entropy inequality. The low order method is based on the flux differencing formulation \eqref{eq:dgform}, and sets the volume flux and interface flux equal to each other $\bm{f}_{\rm vol} = \bm{f}^*_n$. Taking an entropy dissipative volume flux $\bm{f}^*_n$ then produces a ``low order'' scheme which is entropy stable. This approach forms the foundation for several shock capturing and subcell limiting techniques \cite{rueda2021subcell, hennemann2021provably, zhang2025discontinuous}. This method is restricted to nodal collocation formulations; for a more general entropy correction formulation, we refer the reader to \cite{chan2025artificial} and \ref{sec:SG}. 



For this work, we construct the low order scheme by adding a local Lax-Friedrichs penalization to a non-dissipative flux 
\begin{equation}
\bm{f}_{\rm LxF}(\bm{u}_L, \bm{u}_R, \bm{n}) = \bm{f}^*_n(\bm{u}_L, \bm{u}_R) = \sum_{i=1}^d \bm{f}_{i, {\rm avg}}(\bm{u}_L, \bm{u}_R)\bm{n}_i - \frac{\lambda}{2}(\bm{u}_R - \bm{u}_L),
\label{eq:LxF_type}
\end{equation}
where $\lambda > 0$ is an estimate of the maximum wave speed. Under appropriate estimates of the wavespeed, the local Lax-Friedrichs flux is entropy stable for an ideal gas \cite{guermond2016fast, chen2017entropy, toro2020bounds}. In this work, we utilize the Davis wavespeed estimate $\lambda = \max(\lambda^-, \lambda^+)$. Here, $\lambda^\pm = \LRb{v_n^\pm} + c^\pm$, $c$ is the speed of sound, and $v_n$ is the normal velocity. Assuming that the local Lax-Friedrichs-like flux \eqref{eq:LxF_type} is entropy stable, formulation \eqref{eq:dgform} also satisfies a cell entropy inequality \cite{lin2023positivity}, though it is only first order accurate. 

To improve robustness while retaining accuracy, we blend this low order method with the high order flux differencing formulation described in Section~\ref{sec:APEC} to enforce a cell entropy inequality. We define $\bm{f}_{\rm vol}$ as a blending of a high order EPEC or APEC flux and a low order dissipative flux
\begin{equation}
\bm{f}_{\rm vol}\LRp{\fnt{u}_i, \fnt{u}_j, \frac{\fnt{n}_{ij}}{\nor{\fnt{n}_{ij}}}} = (1-\theta) \bm{f}_{\rm high}\LRp{\fnt{u}_i, \fnt{u}_j, \frac{\fnt{n}_{ij}}{\nor{\fnt{n}_{ij}}}} + \theta \bm{f}_{\rm low}\LRp{\fnt{u}_i, \fnt{u}_j, \frac{\fnt{n}_{ij}}{\nor{\fnt{n}_{ij}}}} 
\label{eq:flux_blend}
\end{equation}
In this work, the low order flux is taken to be an APEC flux with local Lax-Friedrichs jump penalization. This choice of flux is intended to minimize pressure equilibrium errors resulting from the central contribution, although the jump penalty still results in significant pressure equilibrium errors. We describe this in more detail in Section~\ref{sec:apec_ec}. 

If $\theta$ is taken to be constant over an element, the volume contribution to \eqref{eq:dgform} is also a convex combination
\[
\fnt{r}_{i, \theta}(\fnt{u}) = (1 - \theta) \fnt{r}_{i, {\rm high}}(\fnt{u}) + \theta\fnt{r}_{i, {\rm low}}(\fnt{u}) = \fnt{r}_{i, {\rm high}}(\fnt{u}) + \theta \LRp{  \fnt{r}_{i, {\rm low}}(\fnt{u}) - \fnt{r}_{i, {\rm high}}(\fnt{u})},
\]
where $\fnt{r}_{i, {\rm high}}(\fnt{u}), \fnt{r}_{i, {\rm low}}(\fnt{u})$ denote high and low order volume contributions 
\[
\fnt{r}_{i, {\rm high}}(\fnt{u}) = \sum_{j=1}^{N+1}\nor{\fnt{n}_{ij}}\bm{f}_{\rm high}\LRp{\fnt{u}_i, \fnt{u}_j, \frac{\fnt{n}_{ij}}{\nor{\fnt{n}_{ij}}}}, \qquad 
\fnt{r}_{i, {\rm low}}(\fnt{u}) = \sum_{j=1}^{N+1}\nor{\tilde{\fnt{n}}_{ij}}\bm{f}_{\rm low}\LRp{\fnt{u}_i, \fnt{u}_j, \frac{\tilde{\fnt{n}}_{ij}}{\nor{\tilde{\fnt{n}}_{ij}}}}.
\]
Here, $\tilde{\fnt{n}}_{ij}$ denotes algebraic normals defined as in \eqref{eq:dgform} but using sparsified SBP operators described in \cite{pazner2021sparse, lin2023high}. The resulting discretization can be shown to be equivalent to a subcell finite volume scheme at Legendre-Gauss-Lobatto nodes \cite{hennemann2021provably, rueda2022subcell}. We note that the use of sparsified operators is not strictly necessary, and is intended only to reduce algebraic dissipation in the low order contributions at higher orders. 

Then, enforcing $\delta_k(\bm{u}_h) \geq 0$ is equivalent to enforcing 
\[
\sum_{i=1}^{N+1} \bm{v}(\fnt{u}_i)^T\LRp{\fnt{r}_{i, {\rm high}}(\fnt{u}) + \theta \LRp{ \fnt{r}_{i, {\rm low}}(\fnt{u}) -  \fnt{r}_{i, {\rm high}}(\fnt{u})}} + \psi(\fnt{u}_{N+1}) - \psi(\fnt{u}_1)  \geq 0,
\] 
and $\theta$ is given elementwise by
\[
\theta \geq \begin{cases}
\frac{-\min(0, \delta_k(\bm{u}_h))}{\sum_{i = 1}^{N+1} \bm{v}(\fnt{u}_i)^T (\fnt{r}_{i, {\rm low}}(\fnt{u}) - \fnt{r}_{i, {\rm high}}(\fnt{u}))} & \delta_k(\bm{u}_h) < 0\\
0 & \text{otherwise}.
\end{cases}
\]
Note that the denominator is non-zero if the low order flux is entropy stable \cite{lin2023high, christner2025entropy}.


\subsubsection{Dissipation and pressure equilibrium errors}

Utilizing an APEC flux as a volume flux $\bm{f}_{\rm vol} = \bm{f}_{\rm APEC}$ within the flux differencing formulation \eqref{eq:dgform} reduces pressure-equilibrium errors. However, the full DG scheme is only APEC if the interface flux is also APEC. To the authors' knowledge, for conservative formulations common interface fluxes such as local Lax-Friedrichs are only pressure-equilibrium conserving for polytropic ideal gas and stiffened gas EOS \cite{ching2025conservative}. Consider the local Lax-Friedrichs penalization
\[
\sum_{i=1}^d \bm{f}_{i, {\rm avg}} \bm{n}_i - \frac{\lambda}{2}\jump{\bm{u}},
\]
where $\bm{f}_{i, {\rm avg}}$ is a symmetric non-dissipative flux. 
Note that taking the inner product of $\pd{\bm{u}}{t}$ with state
\[
\bm{w} = \begin{bmatrix}
\frac{1}{2}v_1^2 - \LRl{\pd{\rho e}{\rho}}_p\\
- v_1\\
1
\end{bmatrix}
\]
yields $\bm{w}^T\pd{\bm{u}}{t} = \LRl{\pd{\rho e}{p}}_\rho \pd{p}{t}$. 
If $\bm{f}_{i, {\rm avg}}$ is exactly PEC, then the contributions $\bm{w}^T\bm{f}_{i, {\rm avg}}$ also vanish. 
In order to retain PEC in the presence of dissipation, we also require $\bm{w}^T \jump{\bm{u}} = 0$ assuming constant pressure and velocity. This condition is equivalent to 
\[
\bm{w}^T\jump{\bm{u}} = \LRp{\frac{1}{2}v_1^2 - \LRl{\pd{\rho e}{\rho}}_p}\jump{\rho} - v_1^2 \jump{\rho} + \jump{\rho e} + \frac{1}{2}v_1^2 \jump{\rho} =  \jump{\rho e} - \LRl{\pd{\rho e}{\rho}}_p\jump{\rho} = 0,
\]
which is identical to \eqref{eq:pe_condition}. Thus, for ideal and stiffened gas where $\LRl{\pd{\rho e}{\rho}}_p$ is constant with respect to $\rho, p$ and $\jump{\rho e}$ is a constant multiple of $\jump{\rho}$, the local Lax-Friedrichs flux is EPEC. However, for the van der Waals or Peng-Robinson EOS, this is not the case, and the addition of local Lax-Friedrichs dissipation results in pressure equilibrium violations.



\section{Numerical experiments}
\label{sec:4}

In this section, we analyze the behavior of the proposed entropy correction EPEC/APEC schemes for representative problems. We focus on the van der Waals and Peng-Robinson equations of state; results for the stiffened gas equation of state are discussed in \ref{sec:SG}. All results are produced using the open-source Julia library Trixi.jl \cite{schlottkelakemper2021purely, ranocha2022adaptive, ranocha2023efficient} and OrdinaryDiffEq.jl \cite{rackauckas2017differentialequations}.
Unless otherwise specified, the fourth-order, five-stage, low-storage Runge-Kutta method of Carpenter and Kennedy \cite{carpenter1994fourth} is used for time integration.
ForwardDiff.jl \cite{revels2016forward} is used to compute derivatives, and Plots.jl \cite{christ2023plots} is used for the visualizations.
All code to reproduce the results is available online \cite{chan2026nodalRepro}.

\textbf{Van der Waals.} The van der Waals EOS is described by the following relations \cite{coppola2026pep}
\begin{gather*}
p(V, T) = \frac{\rho R T}{1 - \rho b} - a \rho^2, \qquad e(V, T) = c_v T - a \rho,
\\
s(V, T) = c_v \log(T) + R \log(V - b) + s_0,
\end{gather*}
where $s_0$ is some arbitrary reference value. The speed of sound is given by
\[
c^2 = \gamma (\gamma - 1)\frac{e + \rho a}{(1 - \rho b)^2} - 2 a \rho
\]
Unless otherwise stated, the van der Waals parameters used are given by \cite{coppola2026pep}
\begin{gather*}
a = 5.94768233178 \times 10^{-3}, \qquad
b = 1.72768204288 \times 10^{-3}, \\
\gamma = 1.4, \qquad
R = 1, \qquad c_v = \frac{R}{\gamma - 1} = 2.5.
\end{gather*}

\textbf{Peng-Robinson.} The Peng-Robinson EOS is described by the following relations \cite{ma2017entropy}
\begin{gather*}
p(V, T) = \frac{R T}{V - b} - \frac{a(T)}{V^2 + 2 b V - b^2},\\
e(V, T) = c_{v,0} T + K  (a(T) - T a'(T)),\\
s(V, T) = c_{v,0} \log(T) + R \log(V - b) - a'(T) K,
\end{gather*}
where the auxiliary functions are defined as
\begin{gather*}
a(T) = a_0 \LRp{1 + \kappa  \LRp{1 - \sqrt{\frac{T}{T_c}}}}^2,
\qquad
K(V) = \frac{1}{2\sqrt{2} b} \log\left( \frac{V + (1 - \sqrt{2})b}{V + (1 + \sqrt{2})b}\right).
\end{gather*}
%
Finally, the speed of sound is given by
\[
c = \sqrt{-\gamma V^2\LRl{\frac{\partial p}{\partial V}}_{T}}
\]
where $\gamma = \frac{c_p}{c_v}$ and the specific heats are given by
\[
c_v = c_{v,0} - K(V)\,T\,a''(T), 
\qquad
c_p = c_v
      - T\,
        \frac{
          \left(\left.\dfrac{\partial p}{\partial T}\right|_{V}\right)^2
        }{
          \left.\dfrac{\partial p}{\partial V}\right|_{T}
        }.
\]
Unless otherwise stated, all numerical experiments in this work use parameters for nitrogen \cite{ma2017entropy, ma2018modeling}
\begin{gather*}
R = 296.8\,\mathrm{J\,kg^{-1}\,K^{-1}}, \quad
p_c = 3.4\times 10^6\,\mathrm{Pa}, \quad
T_c = 126.2\,\mathrm{K}, \\
c_{v,0} = 743.2\,\mathrm{J\,kg^{-1}\,K^{-1}}, \quad
\omega = 0.0372,\\
b = 0.077796\,\frac{R T_c}{p_c}, \qquad
a_0 = 0.457236\,\frac{\LRp{R T_c}^2}{p_c}, \qquad
\kappa = 0.37464 + 1.54226\,\omega - 0.26992\,\omega^2.
\end{gather*}

\subsection{Comparison of EPEC fluxes}
\label{sec:comparison_EPEC}
We first compare the EPEC fluxes \eqref{eq:ranocha_EPEC} and \eqref{eq:coppola_EPEC}. We utilize a flux differencing DG discretization where both the volume and surface fluxes are non-dissipative (e.g., central, EPEC, or APEC), and consider a smooth supercritical density wave under a van der Waals EOS from \cite{coppola2026pep}. 
%
The domain is periodic over $[-0.5, 0.5]$, and the initial condition is given by
\begin{gather*}
\rho(x,0) = \rho_c \LRp{A + B e^{\sin(2\pi x)}}, \qquad
v_1(x,0) = 1, \qquad p(x,0) = 100,
\end{gather*}
where $\rho_c = (3b)^{-1} \approx 192.9367 \mathrm{kg\,m^{-3}}$ is the critical density, $A = 0.07$, and $B = 0.12$. 

First, we observe that for a flux differencing DG formulation, the EPEC flux of Coppola, de Michele, and Aiello \eqref{eq:coppola_EPEC} requires a small CFL for stability. For $N=3$, we obtain stable computations for a CFL of $0.05$ on 4 elements and $0.002$ on 8 elements, and a similarly small or even smaller CFL was used in \cite{coppola2026pep}. In contrast, the correction-based EPEC flux \eqref{eq:ranocha_EPEC} and APEC flux \eqref{eq:APEC_KEEP} of \cite{terashima2025approximately} both remain stable up to a larger CFL of $1.5$.

This sensitivity can be traced back to the density average \eqref{eq:coppola_rho_avg}. While its value stays close to $\avg{\rho}$, its derivative with respect to the two states scales like $1/\jump{\LRl{\pd{\rho e}{\rho}}_{p}}$, so that the linearization of the semi-discretization is amplified whenever neighboring states have nearly equal values of $\LRl{\pd{\rho e}{\rho}}_{p}$. This is illustrated in Figure~\ref{fig:epec_stiffness}, which shows the evolution of $-\min_i \mathrm{Re}(\lambda_i)$, where $\lambda_i$ denote the eigenvalues of the Jacobian of the flux differencing DG formulation, computed by automatic differentiation along the solution trajectory. For the EPEC flux of Coppola, de Michele, and Aiello \eqref{eq:coppola_EPEC}, this quantity is several orders of magnitude larger than for the central, correction-based EPEC, and APEC fluxes, and it varies strongly in time, which explains the small time step required for stability. Since the amplification grows under mesh refinement, the CFL for which we obtain stable computations also decreases as the mesh is refined, so that the restriction is not of CFL type. Since the amplification is intermittent, stability is in addition not monotone in the CFL: on 8 elements, for instance, computations with a CFL of $0.01$ and $0.002$ complete, whereas a CFL of $0.005$ does not. For the same reason, it is sensitive to differences at the level of round-off: the computation with a CFL of $0.01$ on 8 elements completes on one machine, but diverges on another, which is why we use a CFL of $0.002$ on this mesh in the following. The sensitivity moreover increases sharply with the polynomial degree. For $N=7$, we obtain a stable computation for a CFL of $0.001$ on 4 elements, while neither $0.002$ nor $0.0005$ do, and on 8 elements none of the CFLs we tested between $0.01$ and $0.0005$ complete the simulation. The correction-based EPEC flux \eqref{eq:ranocha_EPEC} and the APEC flux \eqref{eq:APEC_KEEP} remain stable for $N=7$ at a CFL of $0.5$ on both meshes.

\begin{figure}
\centering
\subfloat[$N=3$, 4 elements]{\includegraphics[width=0.45\textwidth]{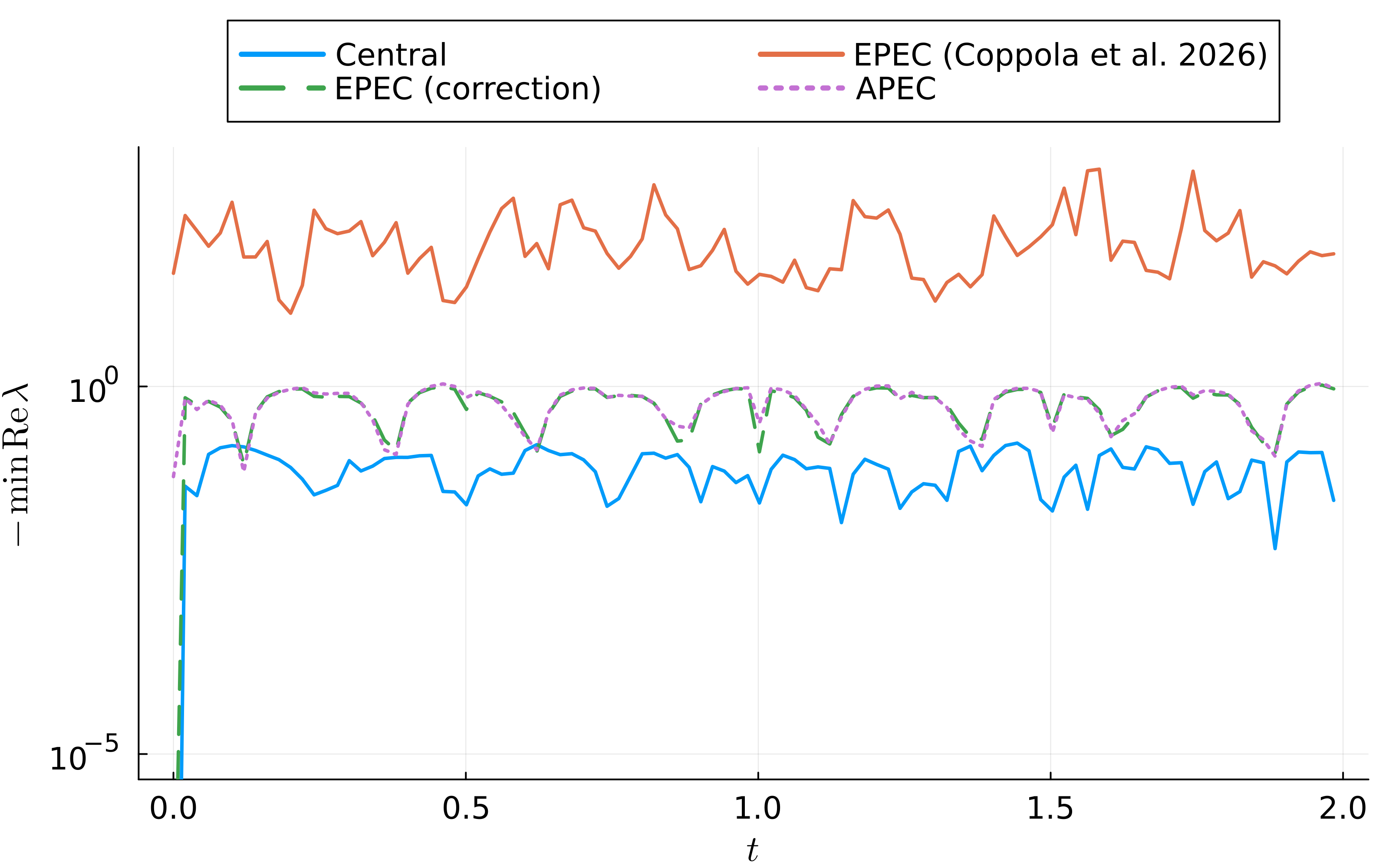}}
\hspace{.2em}
\subfloat[$N=3$, 8 elements]{\includegraphics[width=0.45\textwidth]{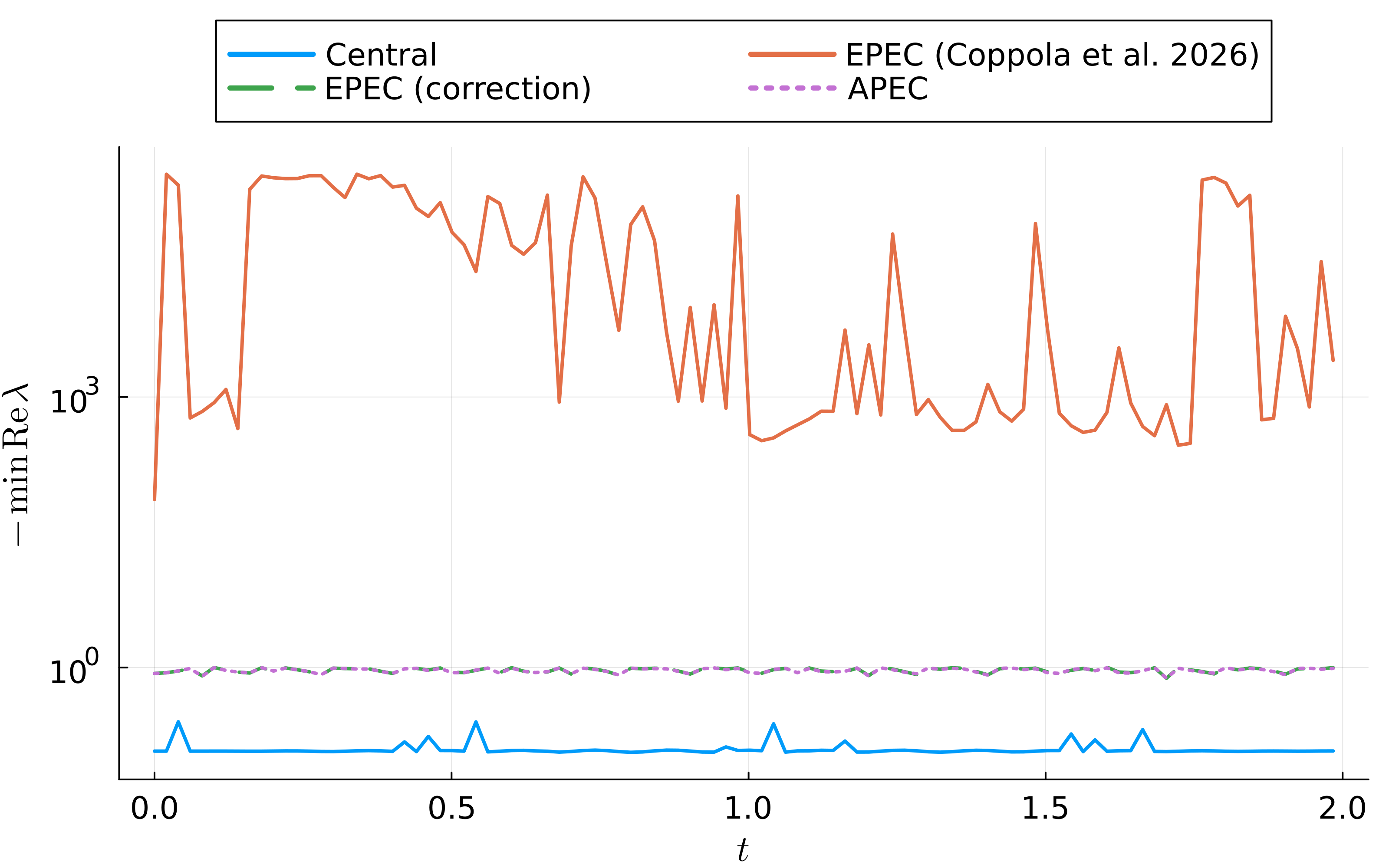}}
\caption{Evolution of $-\min_i \mathrm{Re}(\lambda_i)$, where $\lambda_i$ denote the eigenvalues of the Jacobian of the flux differencing DG formulation with EPEC fluxes \eqref{eq:ranocha_EPEC} and \eqref{eq:coppola_EPEC}, APEC flux \eqref{eq:APEC_KEEP}, and a central flux. CFL constants of $0.01$ and $0.002$ are used for \eqref{eq:coppola_EPEC} on 4 and 8 elements, respectively, and a CFL of $0.5$ for all other fluxes.}
\label{fig:epec_stiffness}
\end{figure}

\begin{figure}
\centering
\subfloat[$N=3$, 4 elements, pressure error]{\includegraphics[width=0.4\textwidth]{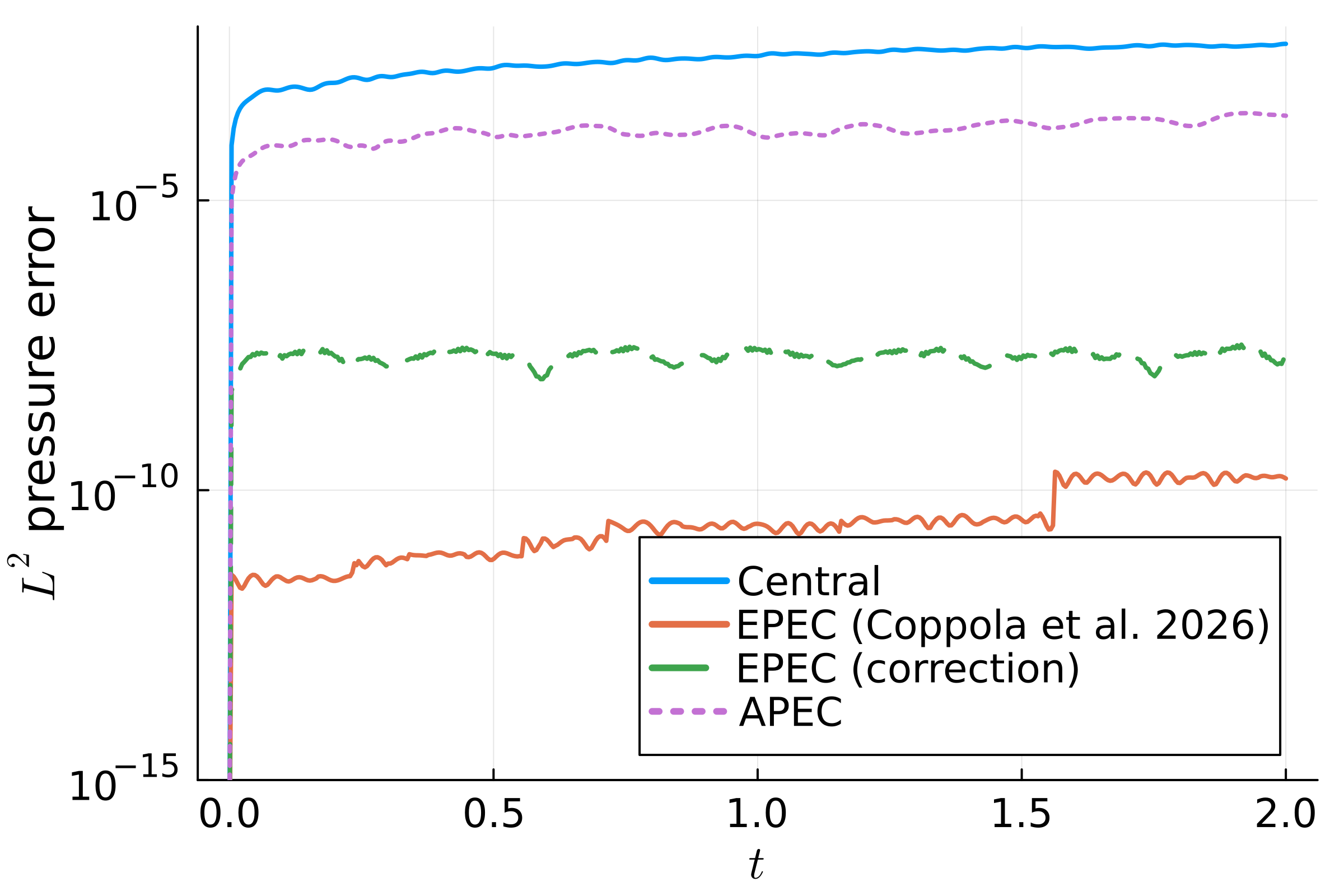}}
\hspace{.2em}
\subfloat[$N=3$, 8 elements, pressure error]{\includegraphics[width=0.4\textwidth]{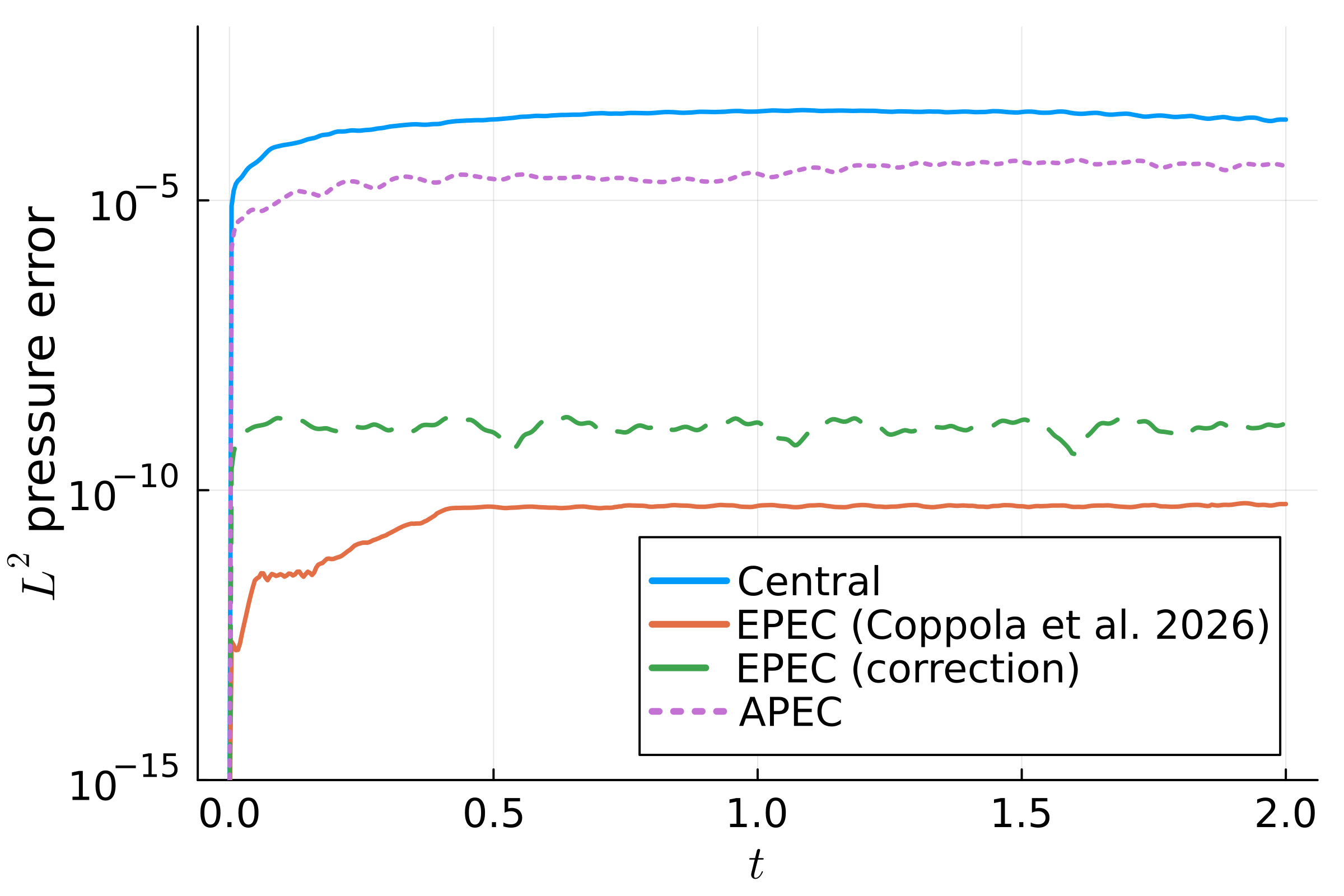}}
\caption{Comparison of flux differencing DG with EPEC fluxes \eqref{eq:ranocha_EPEC} and \eqref{eq:coppola_EPEC} and APEC flux \eqref{eq:APEC_KEEP}.  CFL constants of $0.01$ and $0.002$ are used for \eqref{eq:coppola_EPEC} on 4 and 8 elements, respectively. A CFL of $0.5$ is used for all other fluxes.}
\label{fig:epec_comparison}
\end{figure}

Next, we compare the evolution of the error over time for several different fluxes. In addition to EPEC and APEC fluxes, we compare the performance of central fluxes, for which the flux differencing DG formulation reduces to a standard DG weak form with a central surface flux. CFL values of $0.01$ and $0.002$ are used for the EPEC flux of Coppola, de Michele, and Aiello \eqref{eq:coppola_EPEC} on the coarser and finer mesh, respectively. These lower values are required for stability. A CFL of $0.5$ is used for all other fluxes.


Figure~\ref{fig:epec_comparison} shows the evolution of the pressure error until final time $T = 2$ for degree $N=3$ polynomials on uniform meshes of 4 and 8 elements. Unsurprisingly, the pressure error is lowest for EPEC/APEC fluxes. The EPEC flux of Coppola, de Michele, and Aiello \eqref{eq:coppola_EPEC} maintains the lowest pressure equilibrium error overall by a few orders of magnitude, although this error grows over the course of the simulation on the coarser mesh. The EPEC correction \eqref{eq:ranocha_EPEC} maintains a pressure equilibrium around $1\times 10^{-8}$ to $1\times 10^{-9}$ depending on the resolution.

Reducing the CFL for the EPEC correction \eqref{eq:ranocha_EPEC} decreases the pressure equilibrium error further. The growth of the pressure equilibrium error observed for \eqref{eq:coppola_EPEC} is likewise a time integration effect: reducing the CFL to $0.001$ lowers it to the order of $1\times 10^{-13}$ over the entire time interval on both meshes, which is several orders of magnitude below the error of all other fluxes considered here.

\subsection{APEC fluxes with interface dissipation}

While exact pressure equilibrium conservation is attractive, all EPEC fluxes are observed to be numerically sensitive for highly non-ideal EOS. We illustrate this for the Peng-Robinson EOS under transcritical conditions using the density wave of \cite{ma2017entropy}. The initial conditions are given by 
\begin{equation}
\rho = \frac{\rho_{\min} + \rho_{\max}}{2} +
          \frac{\rho_{\max} - \rho_{\min}}{2}  \sin(2 \pi (x - v_1 t)), 
          \qquad 
v_1 = 100, \qquad p = 5e6 \,\text{Pa},
\label{eq:transcritical_wave}
\end{equation}
where $\rho_{\min} = 56.9 \, \text{kg/m}^3$ and $\rho_{\max} = 739.1 \, \text{kg/m}^3$ and the domain is $[-0.5, 0.5]$. Both EPEC fluxes \eqref{eq:ranocha_EPEC} and \eqref{eq:coppola_EPEC} result in flux differencing schemes which diverge rapidly independently of mesh size or CFL. However, the APEC flux of \cite{terashima2025approximately} remains stable up to a final time of $t = 0.01$ (one period). Both of these results are consistent with observations in \cite{coppola2026pep}. 

Next, we consider the effect of adding local Lax-Friedrichs-like (LxF) interface dissipation on pressure equilibrium errors. Since only APEC fluxes currently appear viable for highly non-ideal EOS, some level of pressure equilibrium error must be introduced by volume terms in the flux differencing \eqref{eq:dgform}. Moreover, the magnitude of the jump, and thus the dissipation added by local Lax-Friedrichs-type jump penalization, is $O(h^{N+1})$ for sufficiently regular solutions \cite{cockburn1998runge, persson2006sub}. In contrast, the analyses of \cite{terashima2025approximately, degrendele2025construction} indicate that APEC schemes yield smaller leading coefficients in expressions for pressure errors but do not affect the order of accuracy. Thus, at higher orders and sufficiently refined grids, the pressure errors introduced by interface dissipation should be small relative to the pressure error introduced by APEC flux differencing in the volume term. 

\begin{figure}
\centering
\subfloat[$N=0$, $256$ elements]{\includegraphics[width=0.45\textwidth]{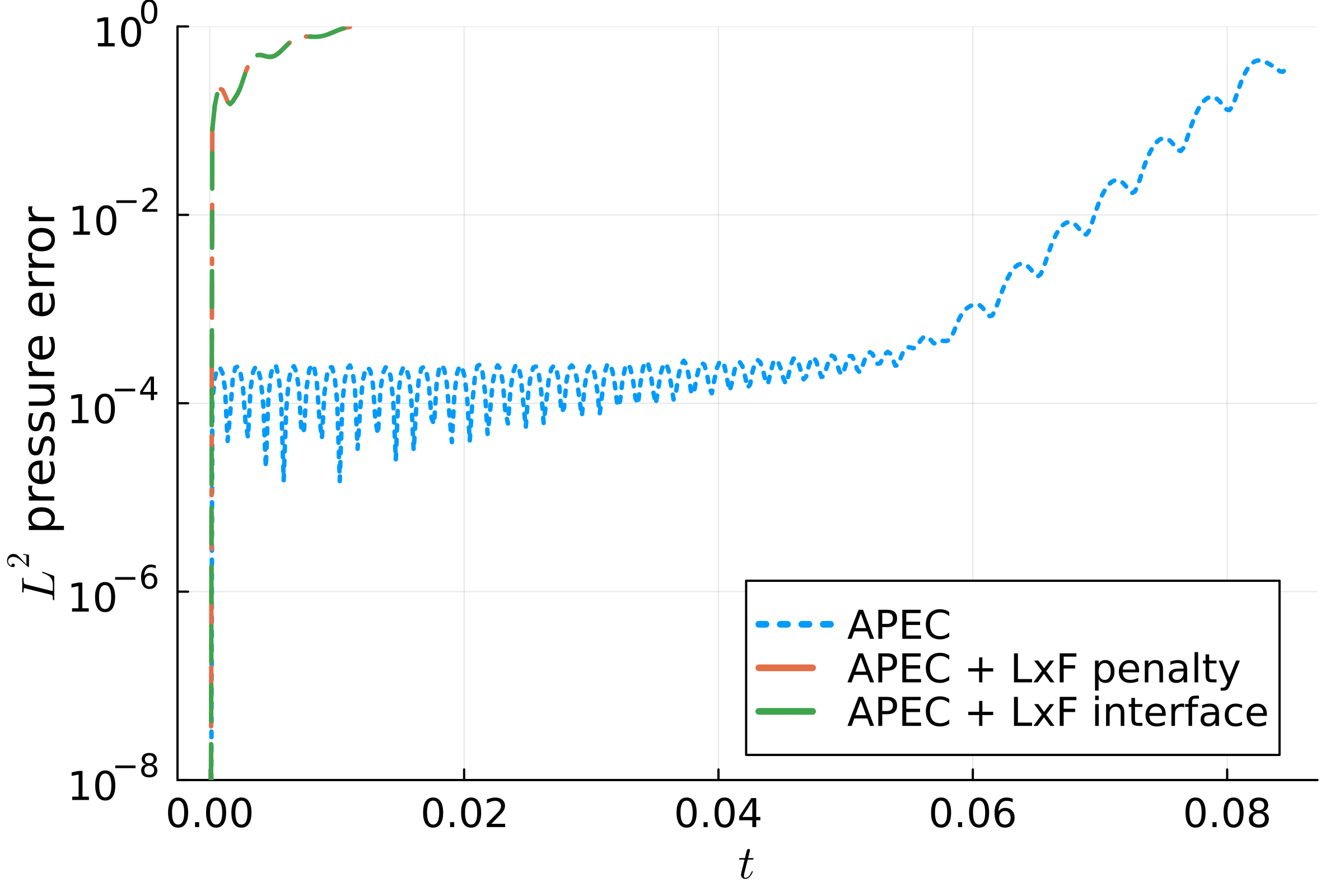}}
\subfloat[$N=1$, $32$ elements]{\includegraphics[width=0.45\textwidth]{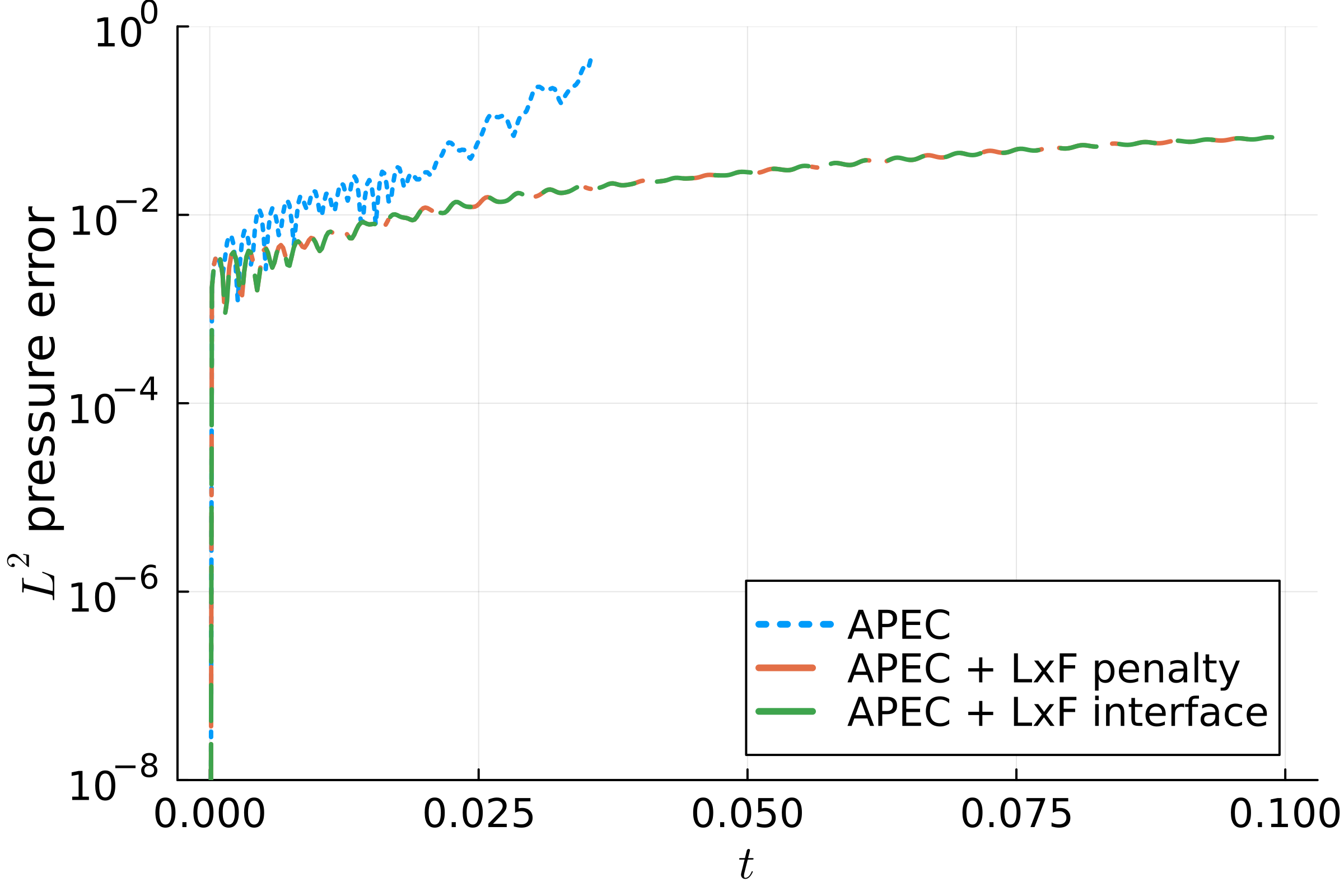}}\\
\subfloat[$N=3$, $16$ elements]{\includegraphics[width=0.45\textwidth]{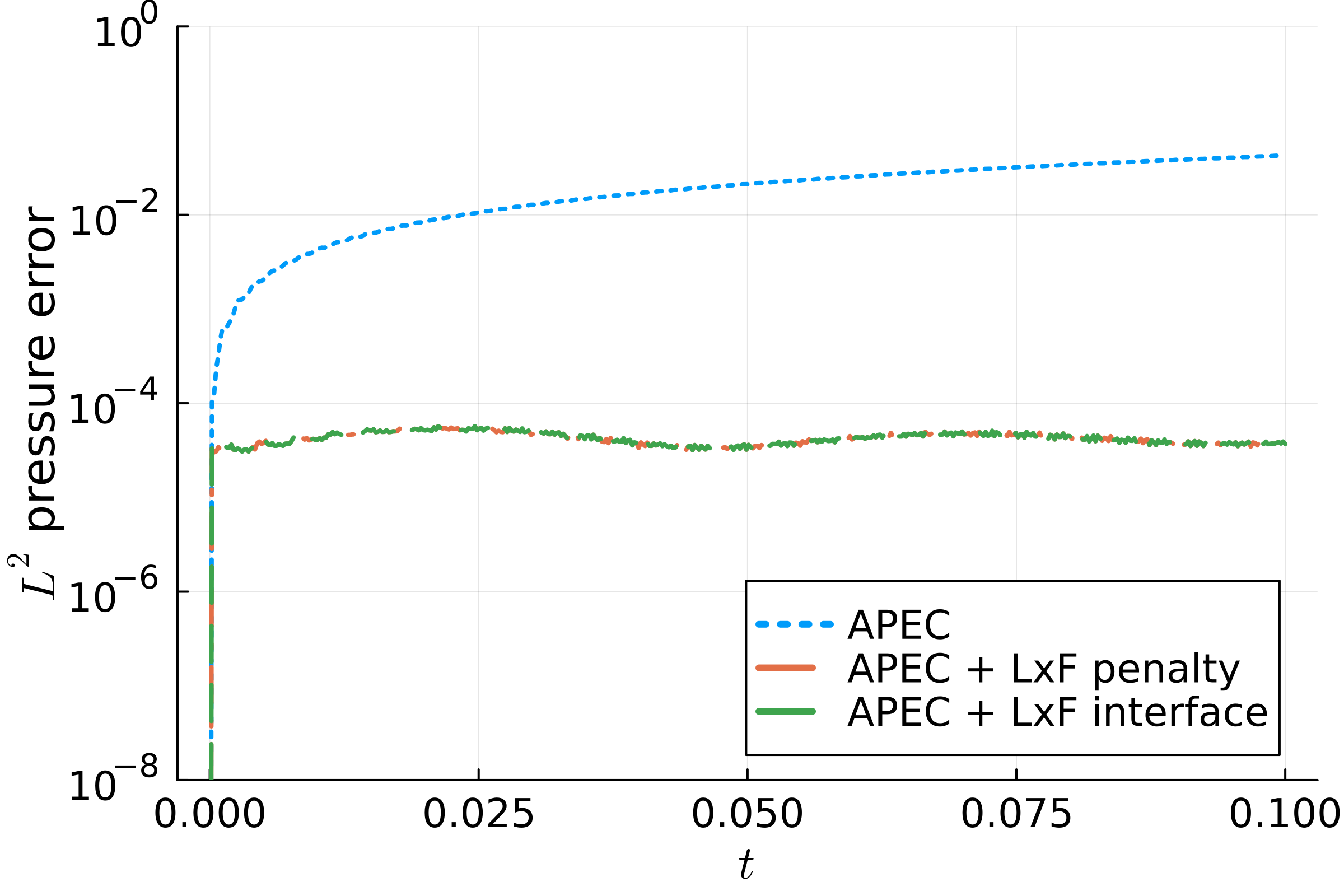}}
\subfloat[$N=7$, $8$ elements]{\includegraphics[width=0.45\textwidth]{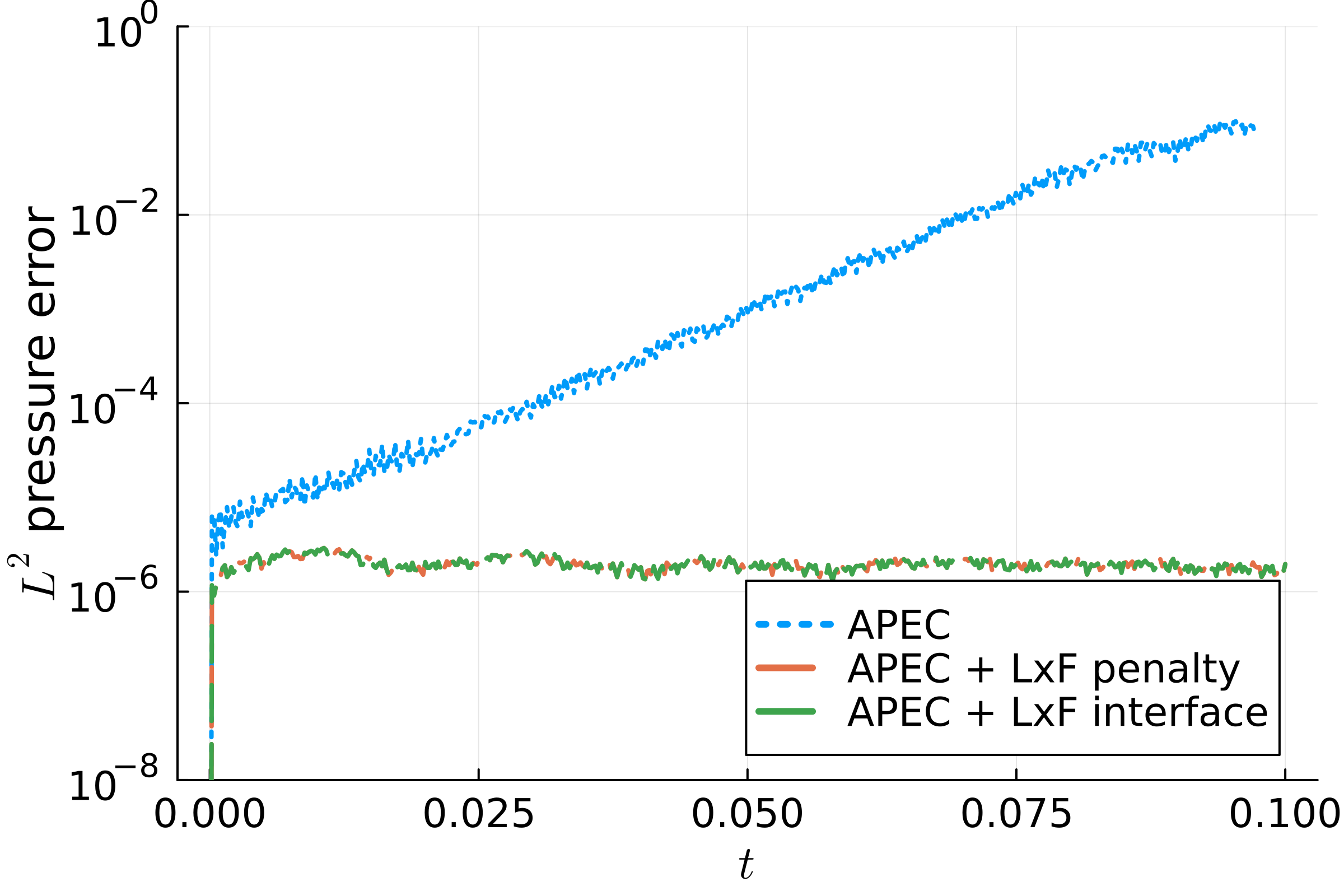}}
\caption{Comparison of $L^2$ pressure errors over time for a pure APEC DG formulation, an APEC + LxF interface penalty formulation, and an APEC + LxF interface flux formulation.  }
\label{fig:APEC_interface_flux}
\end{figure}

Figure~\ref{fig:APEC_interface_flux} illustrates this for the transcritical density wave at a longer final time 0.1, or 10 convective cycles. We assume the volume flux is APEC, and compare three different choices of surface fluxes: a pure APEC surface flux 
\[
\bm{f}_{\rm APEC}(\bm{u}^-, \bm{u}^+, \bm{n}) = 
\sum_{i=1}^d \bm{f}_{i, \rm APEC}(\bm{u}^-, \bm{u}^+) \bm{n}_i,
\]
an APEC surface flux with LxF interface penalty 
\[
\bm{f}_{\rm APEC-LxF}(\bm{u}^-, \bm{u}^+, \bm{n}) = 
\sum_{i=1}^d \bm{f}_{i, \rm APEC}(\bm{u}^-, \bm{u}^+) \bm{n}_i - \frac{\lambda}{2}(\bm{u}^+ - \bm{u}^-),
\]
and a standard local Lax-Friedrichs interface flux
\[
\bm{f}_{\rm LxF}(\bm{u}^-, \bm{u}^+, \bm{n}) = 
\sum_{i=1}^d \avg{\bm{f}_{i}}\bm{n}_i - \frac{\lambda}{2}(\bm{u}^+ - \bm{u}^-).
\]
First, we observe that, with the exception of the $N=0$ case (e.g., first order finite volumes), pressure errors are improved by adding interface dissipation. Moreover, the pure APEC scheme develops large errors or diverges for the degree $N=0$ and $N=1$ cases, while the schemes with interface dissipation all run to the final time. We also observe that there is no significant difference between using a standard local Lax-Friedrichs flux and an APEC interface flux + a LxF interface penalty.  

\subsection{APEC fluxes with interface dissipation and entropy correction}
\label{sec:apec_ec}

Finally, we investigate the effect of enforcing a cell entropy inequality via entropy correction. For this section, we assume that the volume flux in \eqref{eq:dgform} is APEC, and that the surface flux is APEC with a local Lax-Friedrichs penalization. The low order subcell finite volume scheme used to enforce the cell entropy inequality \eqref{eq:flux_blend} is constructed using the same surface flux. For sufficiently accurate quadrature, the entropy residual $\delta_k(\bm{u}_h)$ is $O(h^{2N+2+d})$ in $d$ dimensions \cite{chan2025artificial}, and artificial viscosity versions of entropy correction result in an artificial viscosity coefficient which is $O(h^{2N+2})$ for smooth solutions \cite{van2026choice}. Even for low-regularity solutions, $\delta_k(\bm{u}_h)$ can be shown to be the product of two $L^2$ best approximation errors, and is typically much smaller than the approximation error. 

\begin{figure}[h]
\centering
\subfloat[Smooth wave, $N=3$, 4 elements]{\includegraphics[width=0.42\textwidth, trim={0 0cm 0 0cm}, clip]{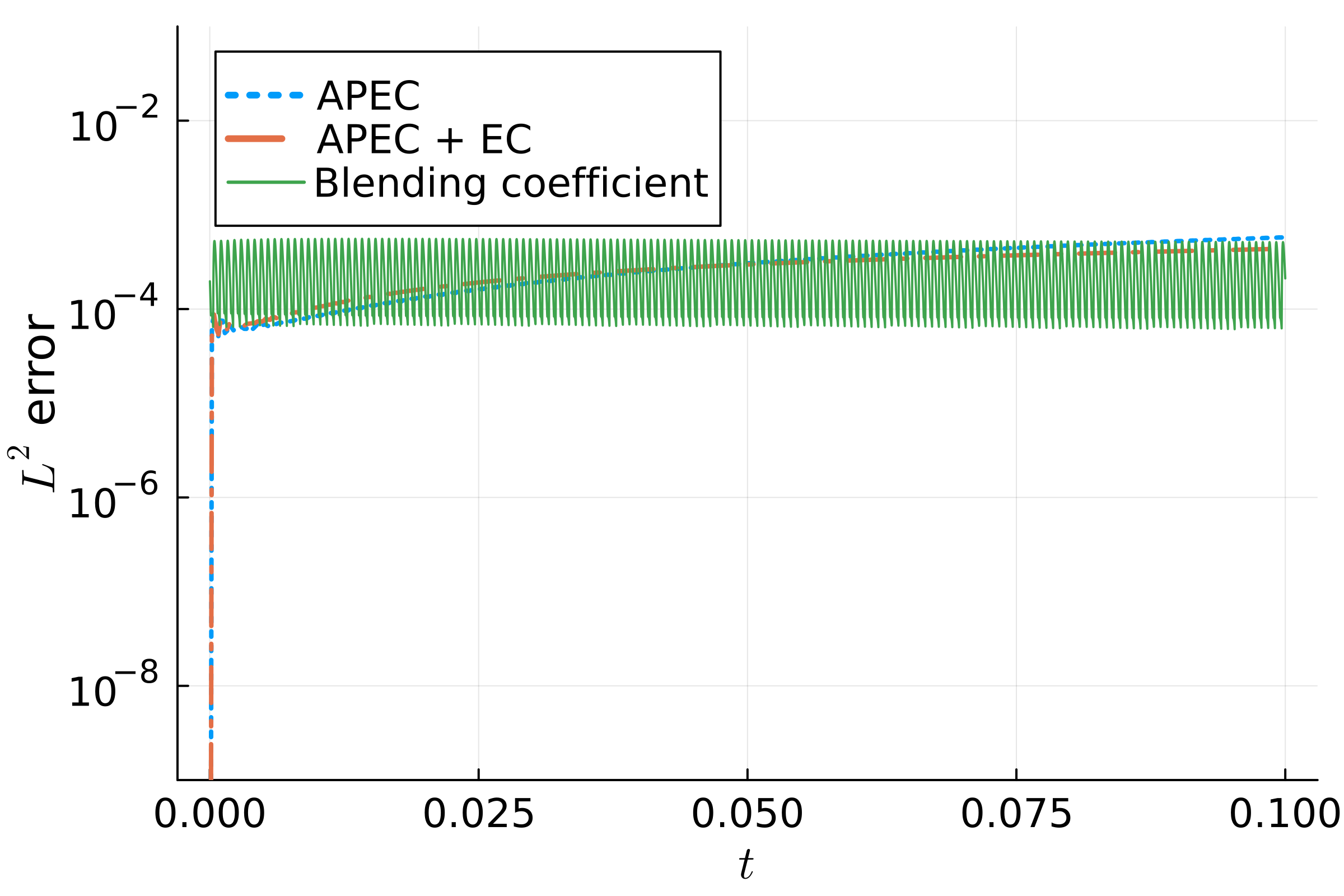}}
\hspace{1em}
\subfloat[Sharp wave, $N=3$, 4 elements]{\includegraphics[width=0.42\textwidth, trim={0 0cm 0 0cm}, clip]{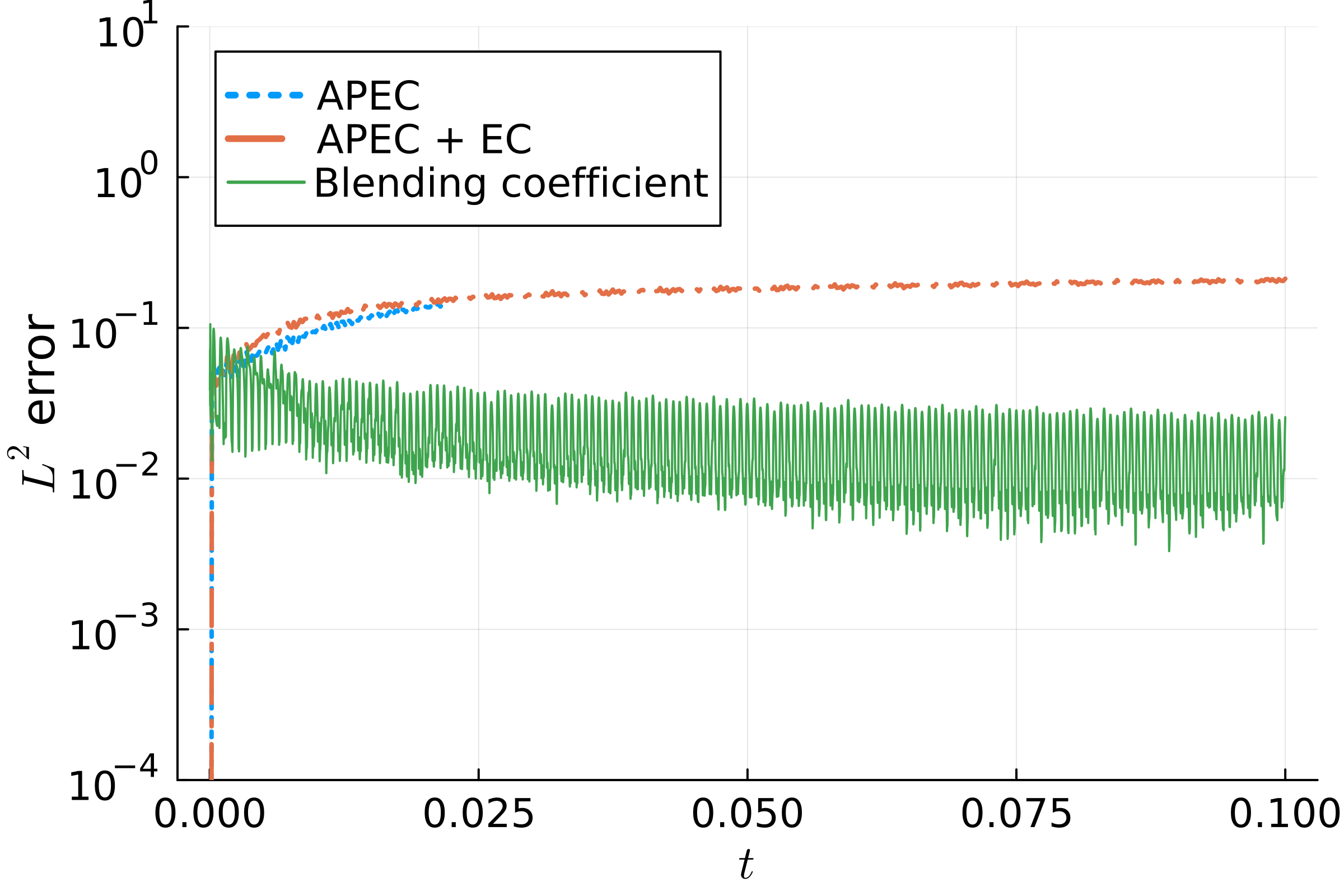}}\\
\subfloat[Smooth wave, $N=7$, 3 elements]{\includegraphics[width=0.42\textwidth, trim={0 0cm 0 0cm}, clip]{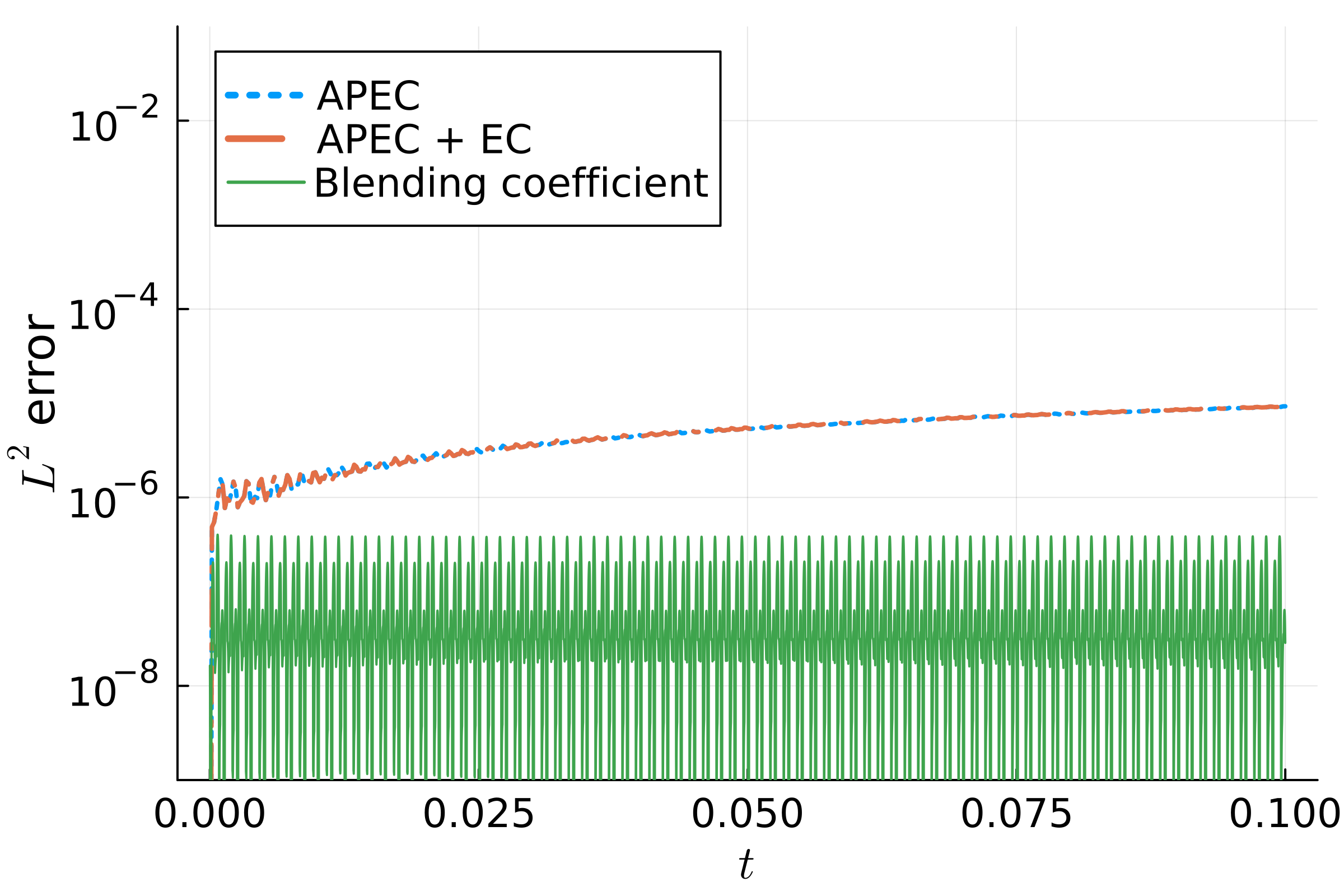}}
\hspace{1em}
\subfloat[Sharp wave, $N=7$, 3 elements]{\includegraphics[width=0.42\textwidth, trim={0 0cm 0 0cm}, clip]{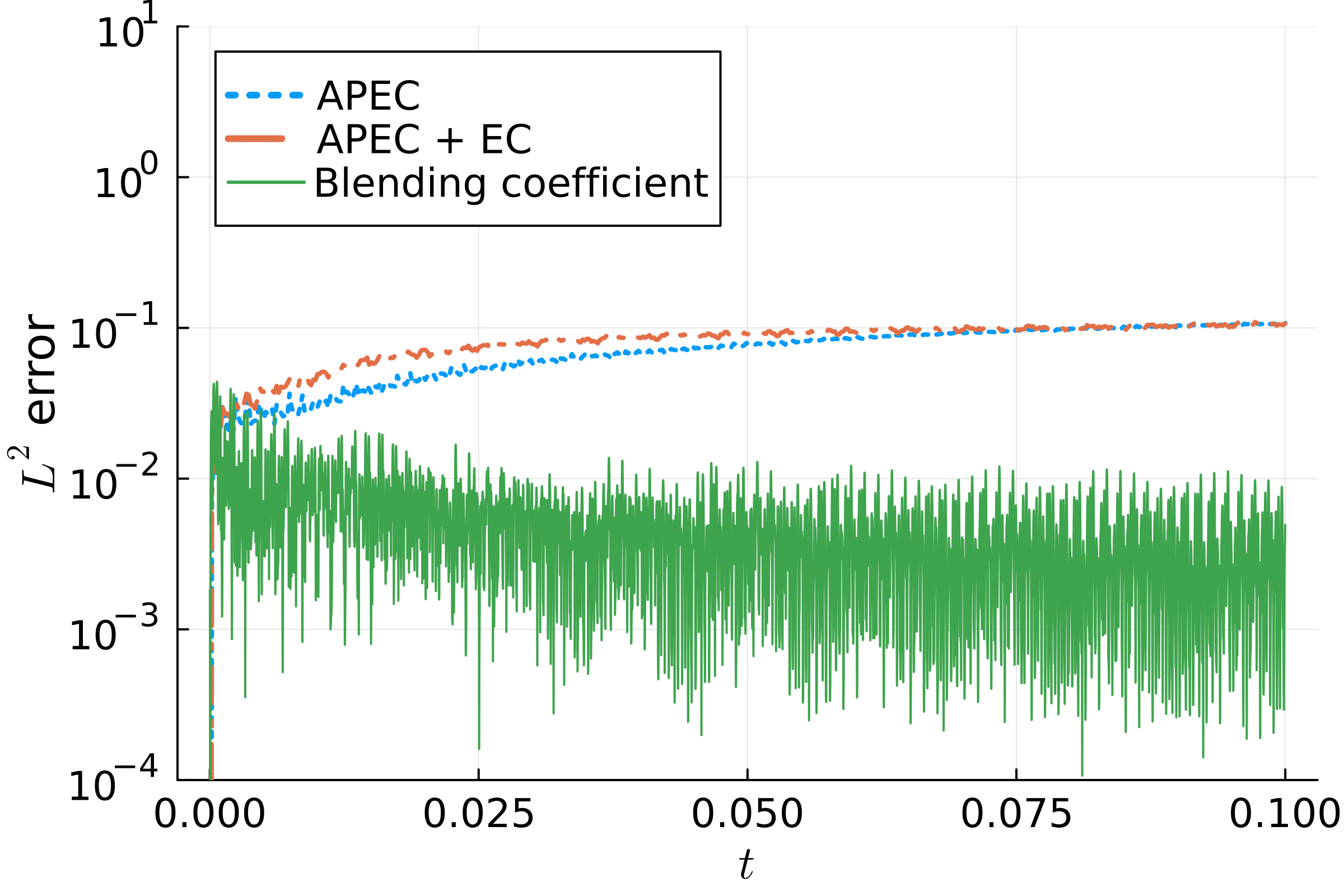}}
\caption{Evolution of the relative $L^2$ error and maximum value of the (unnormalized) entropy correction factor $\theta$ over all elements for both smooth \eqref{eq:transcritical_wave} and sharp versions of the transcritical density wave \eqref{eq:transcritical_wave_sharp}. }
\label{fig:apec_ec_wave}
\end{figure}
Figure~\ref{fig:apec_ec_wave} illustrates this by plotting the evolution of the relative $L^2$ error and maximum value of the (unnormalized) entropy correction blending coefficient $\theta$ over all elements for both the smooth transcritical density wave \eqref{eq:transcritical_wave} and a version of the transcritical density wave with sharper gradients
\begin{equation}
\rho = \frac{\rho_{\min} + \rho_{\max}}{2} +
          \frac{\rho_{\max} - \rho_{\min}}{2} \tanh\LRp{20 \LRb{\sin(\pi (x - v_1 t))} - \frac{3}{4}},
          \qquad 
v_1 = 100, \qquad p = 5e6 \,\text{Pa}.
\label{eq:transcritical_wave_sharp}
\end{equation}

We observe that the errors for APEC flux differencing with and without entropy correction are virtually identical for both $N=3$ and $N=7$, and the magnitude of $\theta$ relative to the error decreases as the polynomial degree $N$ increases independently of whether the solution is smooth or possesses steep gradients.  Moreover, we observe that entropy correction improves robustness; for degree $N=3$ and $4$ elements, APEC flux differencing diverges before $t = 0.025$ for the sharp transcritical wave, while APEC with entropy correction runs stably to the final time. However, we note that (in contrast to the ideal gas case) the improvement in robustness from entropy correction still depends on the resolution; for example, APEC flux differencing with entropy correction diverges for $N=3$ if $3$ elements are used instead of $4$. 



%

\subsubsection{Transcritical shock tube}
\label{sec:transcritical_shock}

We next consider the transcritical shock tube from \cite{ma2017entropy}. The domain is $\LRs{-0.5,\,0.5}$, and the initial condition is given by the discontinuous state
\begin{gather*}
\rho, v_1, p = \begin{cases}
800\,\mathrm{kg\,m^{-3}}, 0, 60\times 10^6\,\mathrm{Pa}, & x < 0 \\
80\,\mathrm{kg\,m^{-3}}, 0, 6\times 10^6\,\mathrm{Pa}, & x > 0.
\end{cases}
\end{gather*}
The simulation is run until final time $T = 5\times 10^{-4}\,\mathrm{s}$ with a CFL number of $0.3$. 

\begin{figure}
\centering
\subfloat[Density, $N=3$, $128$ elements]{\includegraphics[width=0.45\textwidth, trim={0 0cm 0 3cm}, clip]{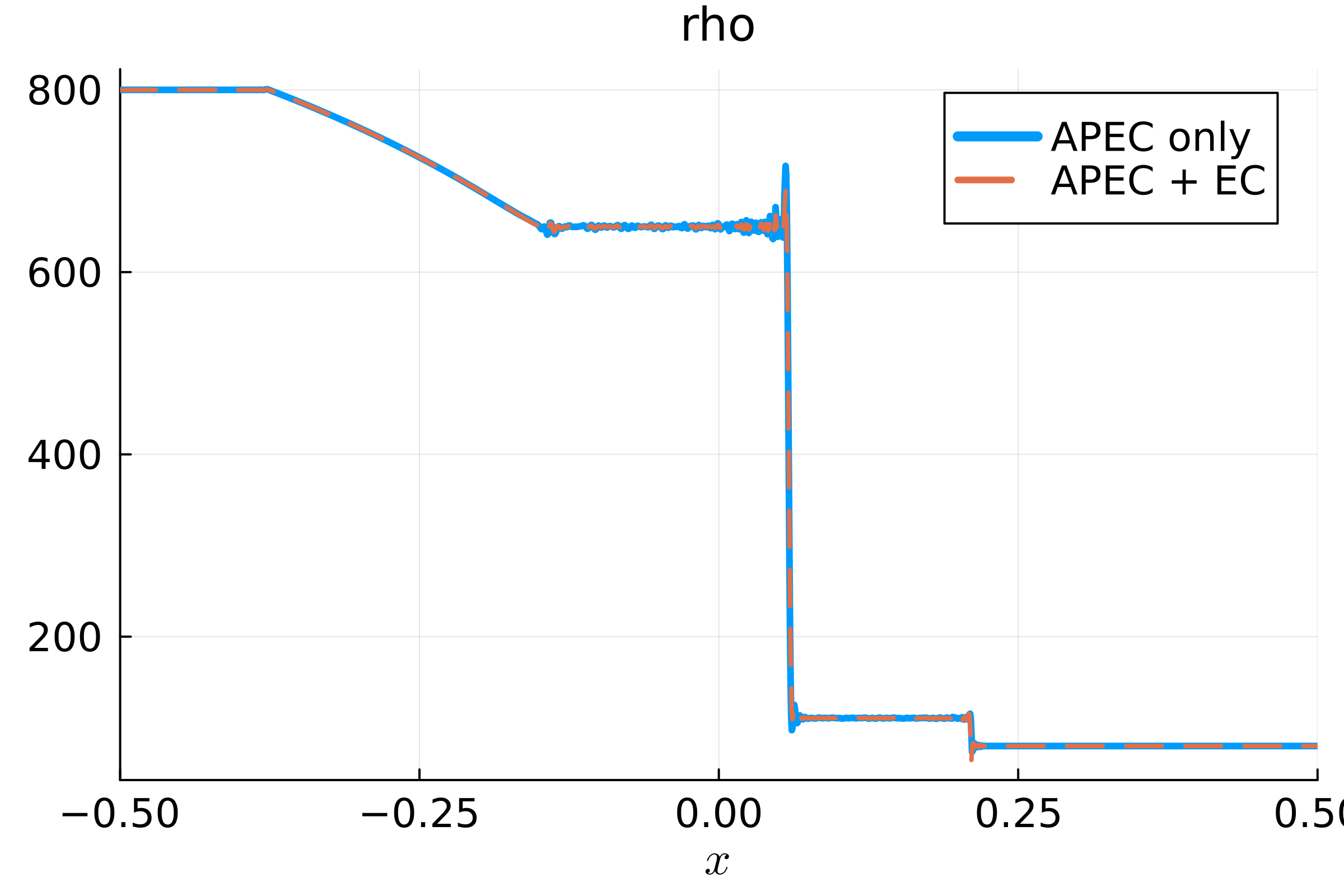}}
\subfloat[Pressure, $N=3$, $128$ elements]{\includegraphics[width=0.45\textwidth, trim={0 0cm 0 3cm}, clip]{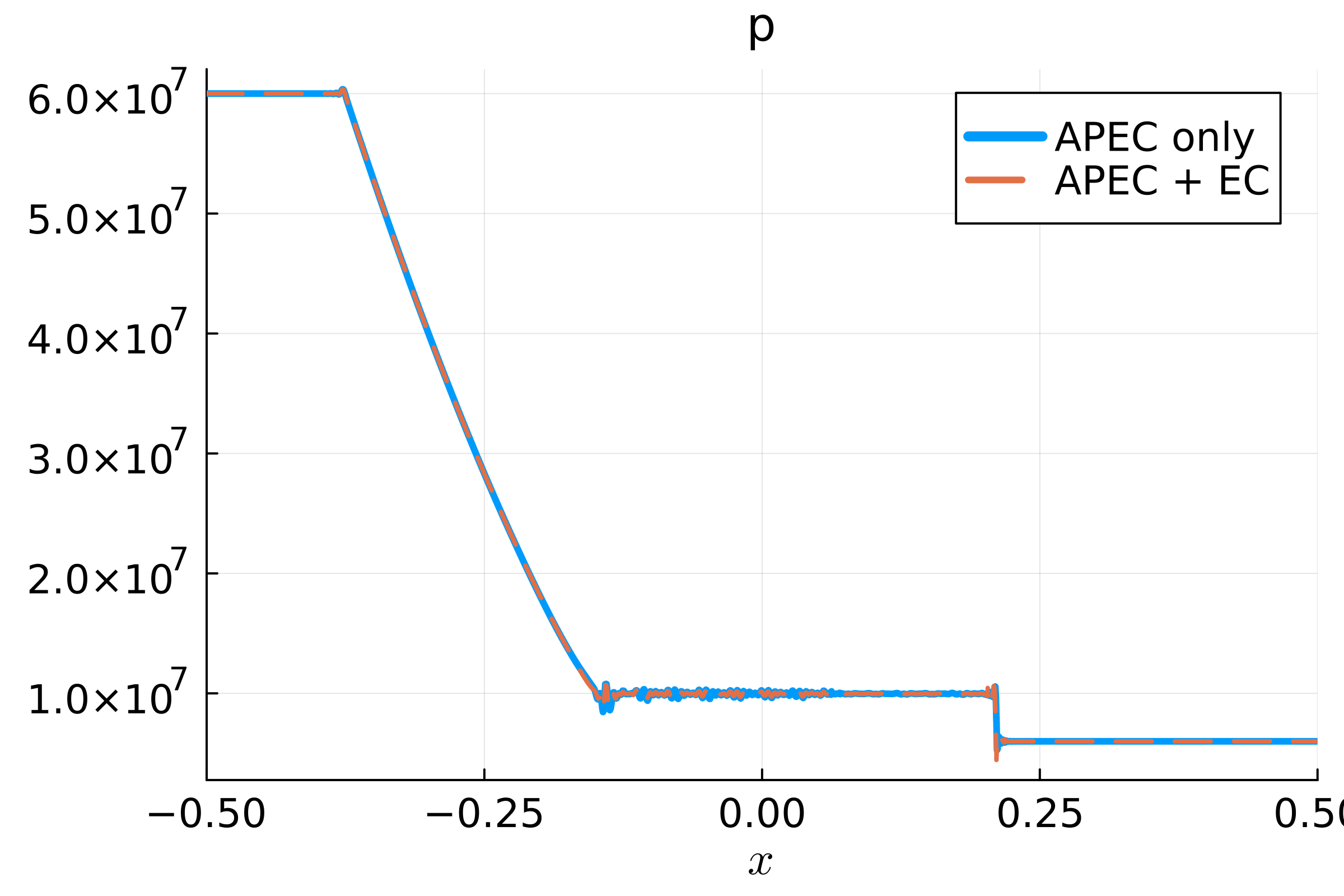}}\\
\subfloat[Density, $N=7$, $64$ elements]{\includegraphics[width=0.45\textwidth, trim={0 0cm 0 3cm}, clip]{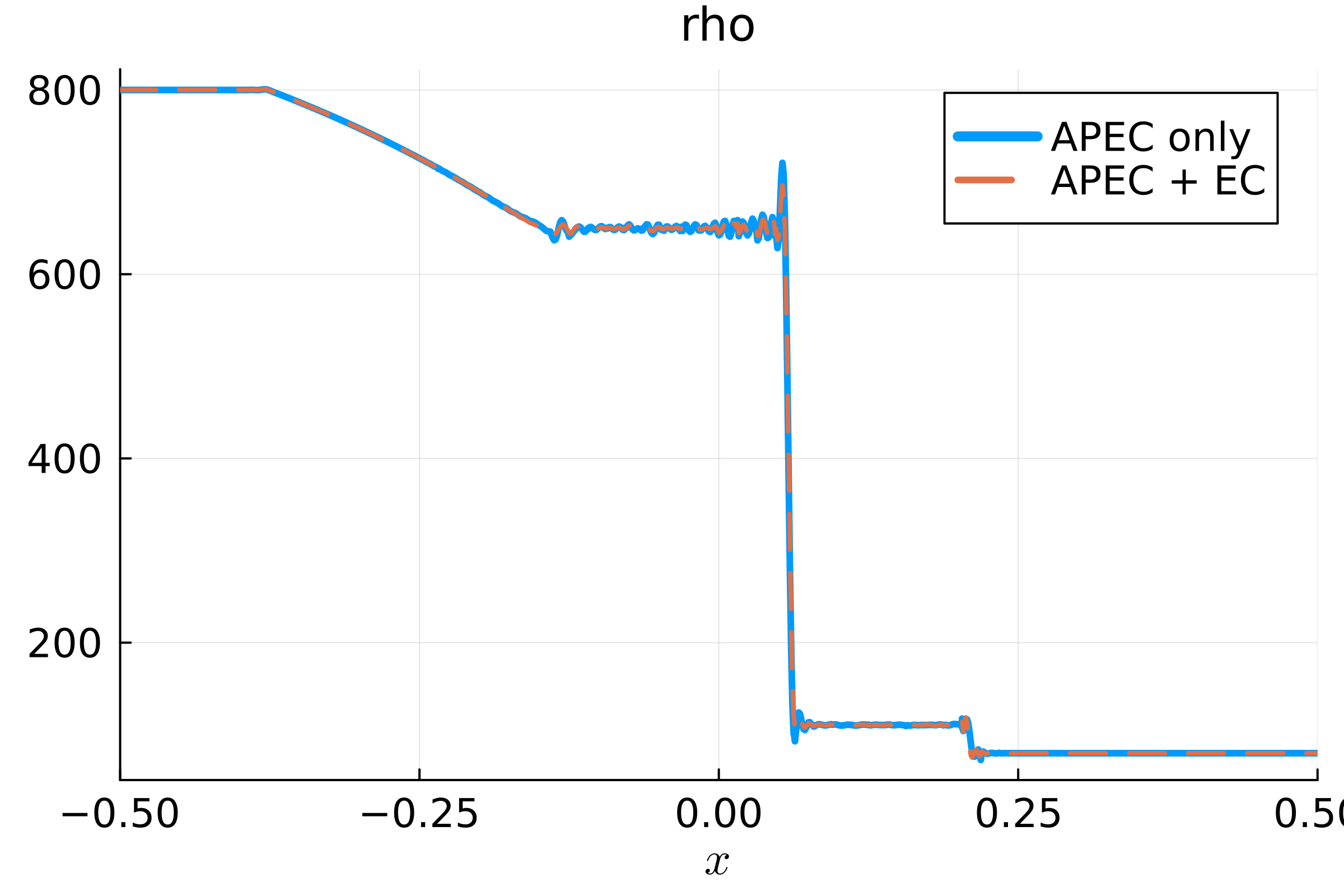}}
\subfloat[Pressure, $N=7$, $64$ elements]{\includegraphics[width=0.45\textwidth, trim={0 0cm 0 3cm}, clip]{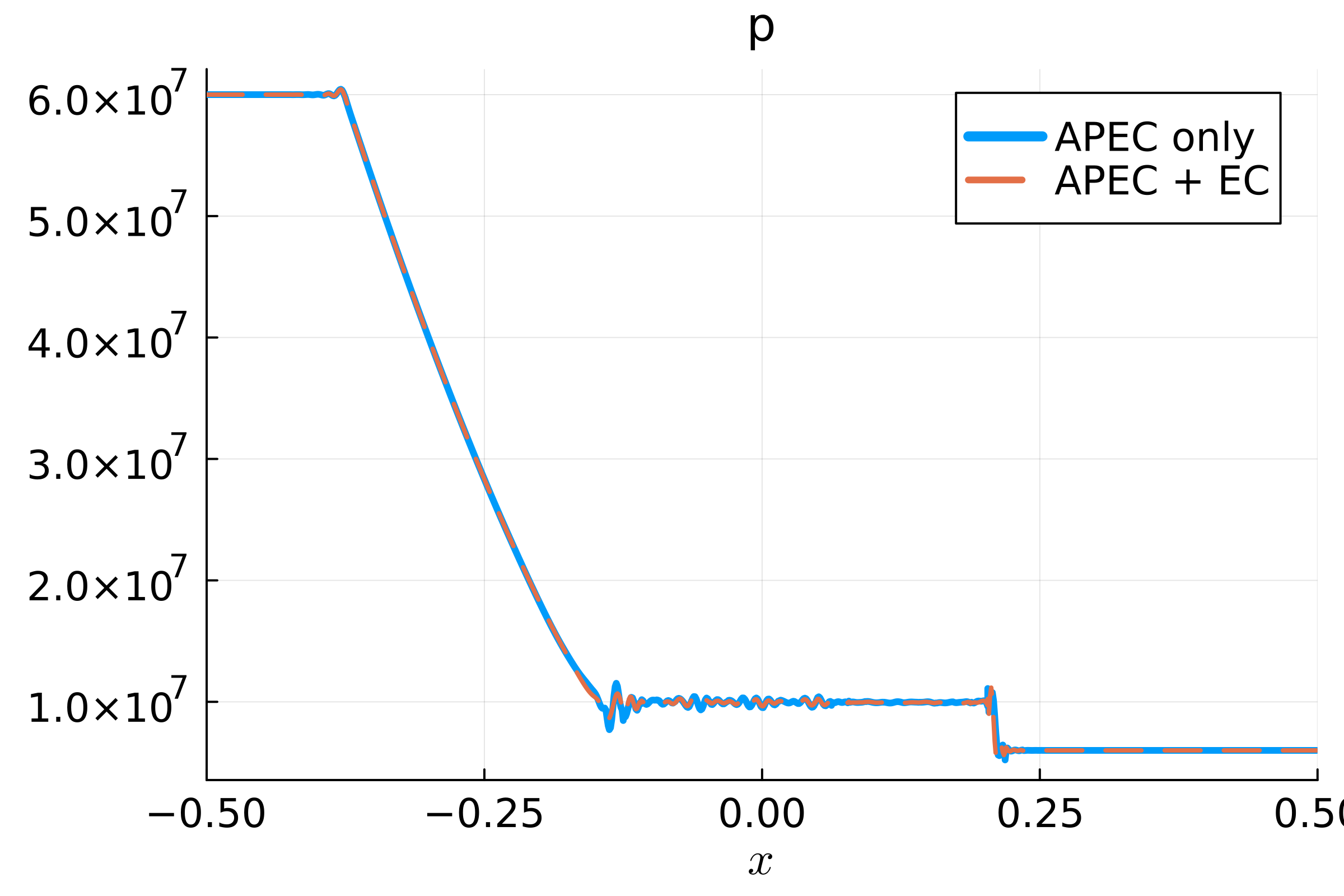}}\\
\caption{Density and pressure for the transcritical shock tube.} 
\label{fig:transcritical_shock}
\end{figure}

The APEC flux differencing DG formulation (APEC volume flux and APEC + LxF interface penalization) runs stably both with and without entropy correction, and Figure~\ref{fig:transcritical_shock} compares the density and pressure profiles for degree $N=3$ and $N=7$ approximations at the final time. We observe that entropy correction slightly reduces oscillations near the shock, contact discontinuity, and start of the rarefaction wave; however, the inclusion of entropy correction does not appear to make a significant difference for this formulation. 

We also note that the behavior of the entropy correction depends strongly on the choice of volume flux. Figure~\ref{fig:compare_blending_coeffs} shows the density overlaid with the entropy correction blending coefficient $\theta$ for different choices of volume flux. The fluxes considered are the central flux, the APEC flux \eqref{eq:APEC_KEEP}, and a modification of the APEC flux which uses only the new internal energy average of \eqref{eq:APEC_KEEP}:
\begin{equation}
\label{eq:APEC_central}
\bm{f}_{\rm APEC-C}(\bm{u}_L, \bm{u}_R) = \begin{bmatrix}
\avg{\rho v_1}\\
\avg{\rho v_1^2 + p}\\
\avg{\LRp{\frac{1}{2}\rho v_1^2 + p}v_1} + (\rho e)_{\rm avg}\avg{v_1}
\end{bmatrix}.
\end{equation}
This modification is introduced to emphasize that the choice of internal energy average $(\rho e)_{\rm avg}$ is primarily responsible for reducing pressure equilibrium errors. Because this is a combination of the central flux (e.g., arithmetic averages of the mass, momentum, and energy fluxes) and the APEC flux \eqref{eq:APEC_KEEP}, we refer to \eqref{eq:APEC_central} as APEC-C. 

\begin{figure}
\centering
\subfloat[Central]{\includegraphics[width=0.33\textwidth, trim={0 0cm 0 3cm}, clip]{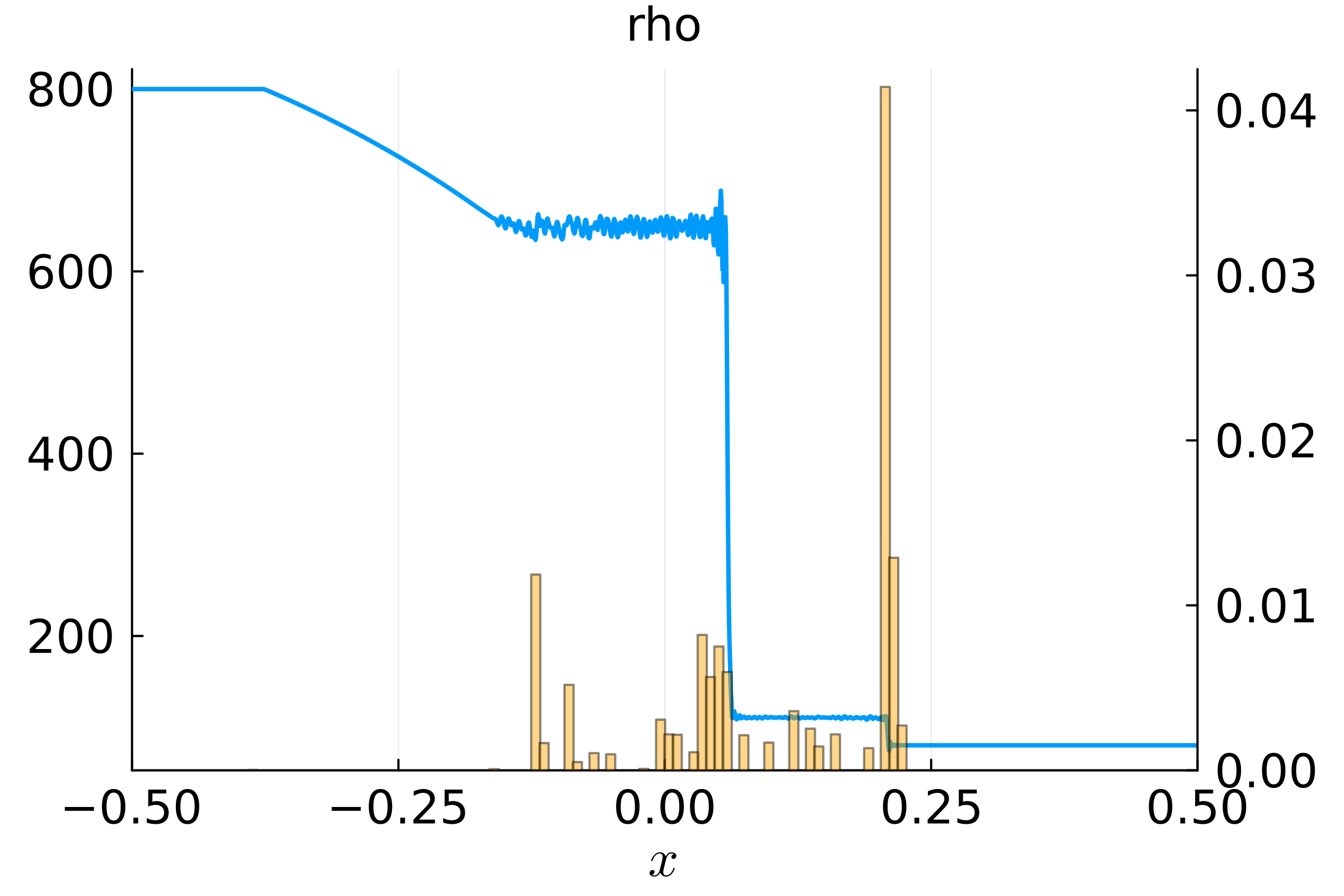}}
\subfloat[APEC-C]{\includegraphics[width=0.33\textwidth, trim={0 0cm 0 3cm}, clip]{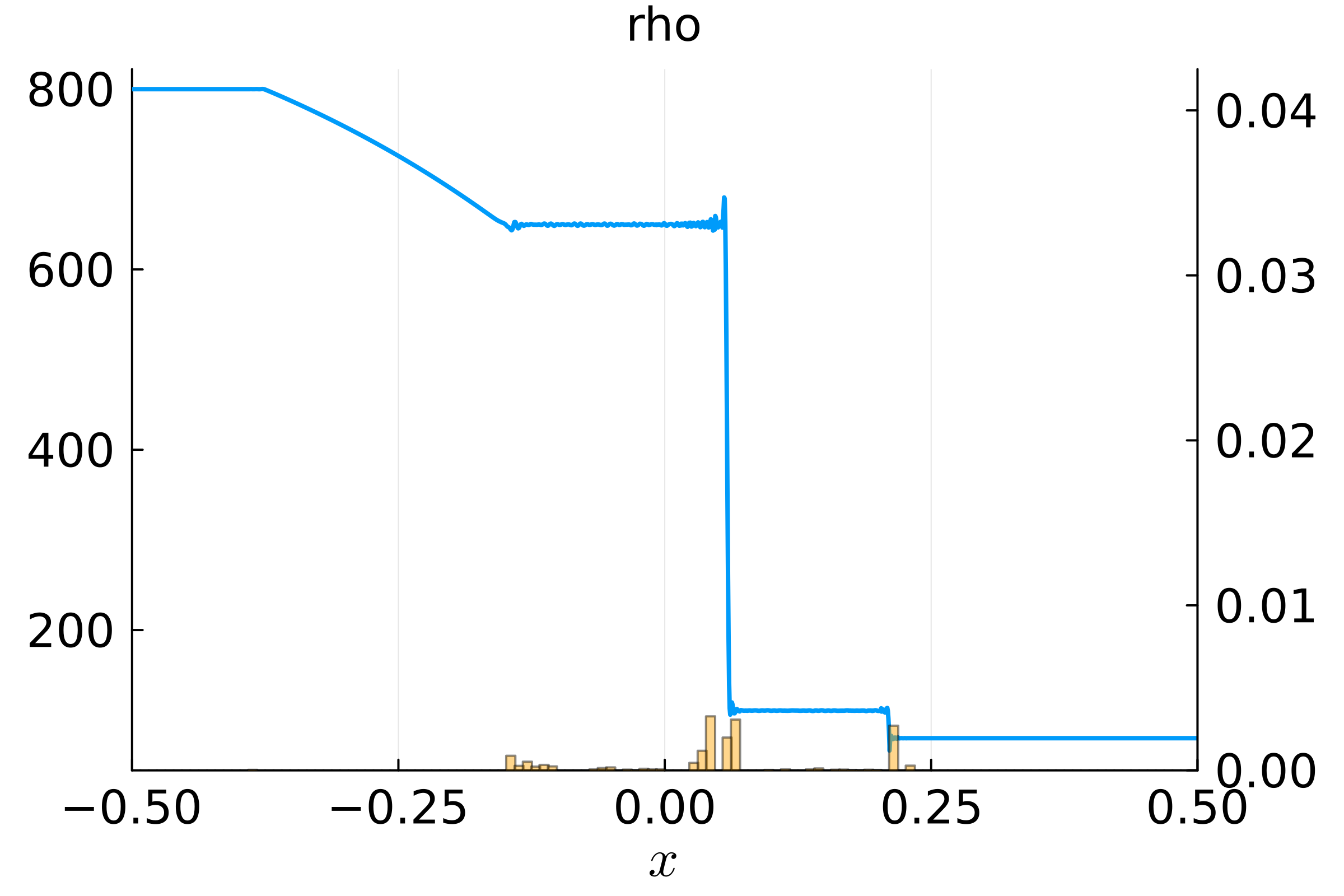}}
\subfloat[APEC]{\includegraphics[width=0.33\textwidth, trim={0 0cm 0 3cm}, clip]{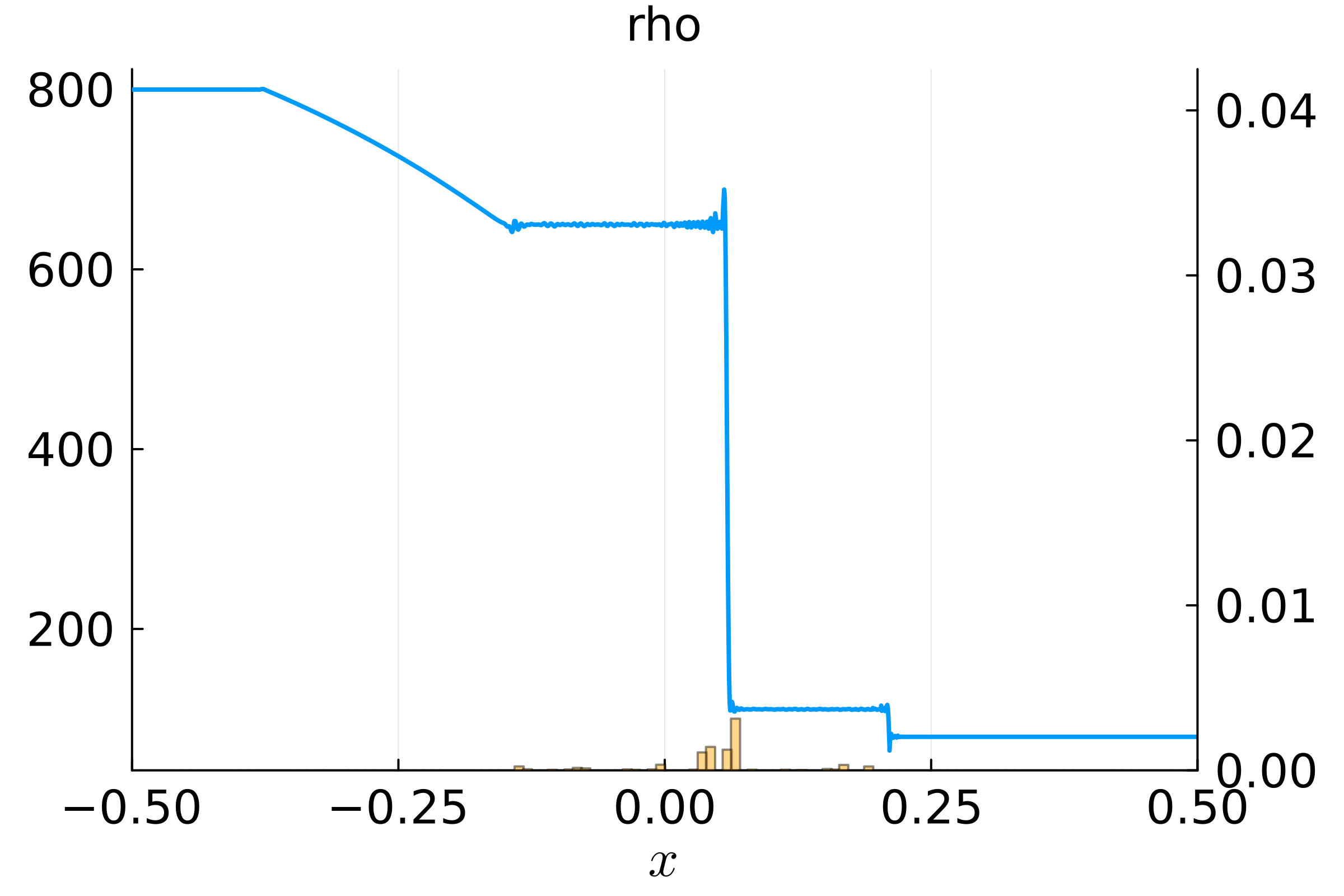}}
\caption{Density profile and entropy correction blending coefficient $\theta$ for degree $N=3$ and $128$ elements and various choices of volume flux in \eqref{eq:dgform}. The left $y$-axis corresponds to the density profile, while the right $y$-axis  corresponds to the magnitude of $\theta$.}
\label{fig:compare_blending_coeffs}
\end{figure}

First, we observe in Figure~\ref{fig:compare_blending_coeffs} that the use of a central volume flux leads to spurious oscillations in the region between the start of the rarefaction wave and the contact discontinuity; in contrast, these oscillations are significantly reduced using APEC or APEC-C volume fluxes. Second, we observe that the distribution and magnitude of the entropy correction blending coefficient $\theta$ depends strongly on the choice of volume flux. For a central volume flux, the entropy correction blending coefficient is most active near the shock and less active around the spurious oscillations behind the contact discontinuity. Moreover, we observe that, for a central volume flux, the transcritical shock tube becomes unstable and diverges without entropy correction. 

For an APEC-C volume flux, spurious pressure oscillations are not observed, and the blending near the shock and around the oscillations near the contact discontinuity are closer in magnitude. For the kinetic-energy preserving APEC volume flux, the blending coefficient becomes smaller at the shock compared with oscillations near the contact discontinuity, possibly due to the anti-aliasing and stabilizing effects of split-form approximations \cite{gassner2016split, winters2018comparative,rojas2021robustness,sjogreen2018high}.

\subsection{Two-dimensional transcritical mixing layer}
\label{sec:pr_transcritical_mixing}

We now consider a 2D transcritical mixing layer problem from \cite{bernades2023kinetic} using the Peng-Robinson EOS. The computational domain is $\LRs{-0.5,\,0.5} \times \LRs{-0.25,\,0.25}$, with periodic boundary conditions imposed in the $x$-direction and slip-wall boundary conditions imposed at $y = \pm 0.25$. The initial pressure is constant and given by $p = 2 p_c = 6.8 \times 10^6\,\mathrm{Pa}$. The initial temperature and velocity fields are given by
\begin{gather*}
T(x,y) = T_c \LRp{3A - A \tanh\LRp{\frac{y}{\delta}}}, \\
v_1(x,y) = v_{1,0} \LRp{1 + 0.2 \tanh\LRp{\frac{y}{\delta}}} + \Delta v(x,y),
\qquad
v_2(x,y) = \Delta v(x,y),
\end{gather*}
where $A = \frac{3}{8}$, $\delta = \frac{1}{20}$, $v_{1,0} = 25\,\mathrm{m\,s^{-1}}$, and $\Delta v(x, y)$ is a spanwise perturbation
\[
\Delta v(x,y)
= \epsilon \sin\LRp{\pi k x}
\frac{\tanh\LRp{100\LRp{y + 0.1}} - \tanh\LRp{100\LRp{y - 0.1}}}{2},
\qquad
\epsilon = 1, \qquad k = 6.
\]
The specific volume $V$ is calculated by solving $p\LRp{V,T} = 2p_c$, and the simulation is run until final time  $2 t_c$, where $t_c = 0.033$ \cite{bernades2023kinetic}. 

\begin{figure}
\centering
\subfloat[$\rho / \rho_c$]{\includegraphics[width=0.49\textwidth, trim={0 8cm 0 8cm}, clip]{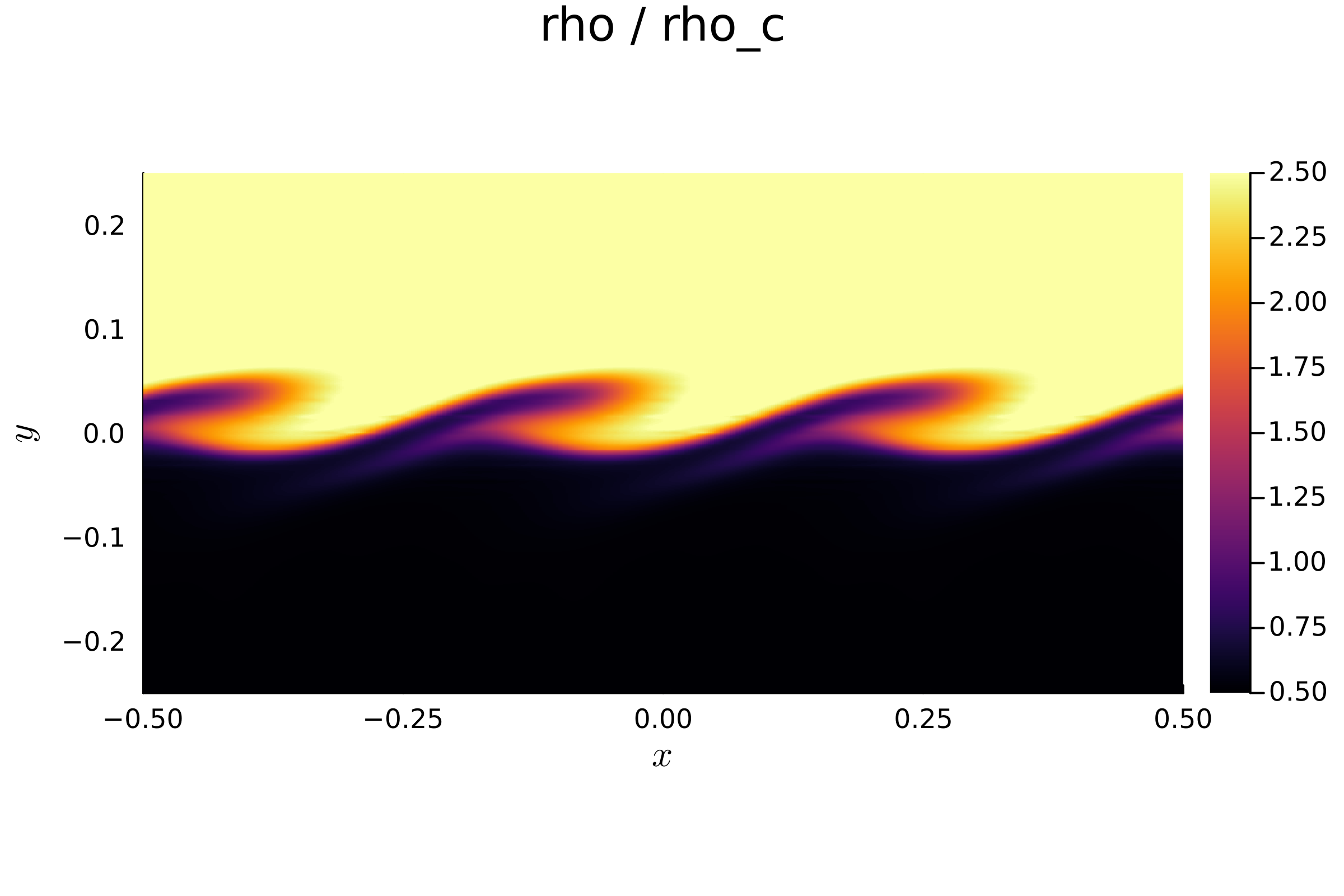}}
\subfloat[$\gamma$]{\includegraphics[width=0.49\textwidth, trim={0 8cm 0 8cm}, clip]{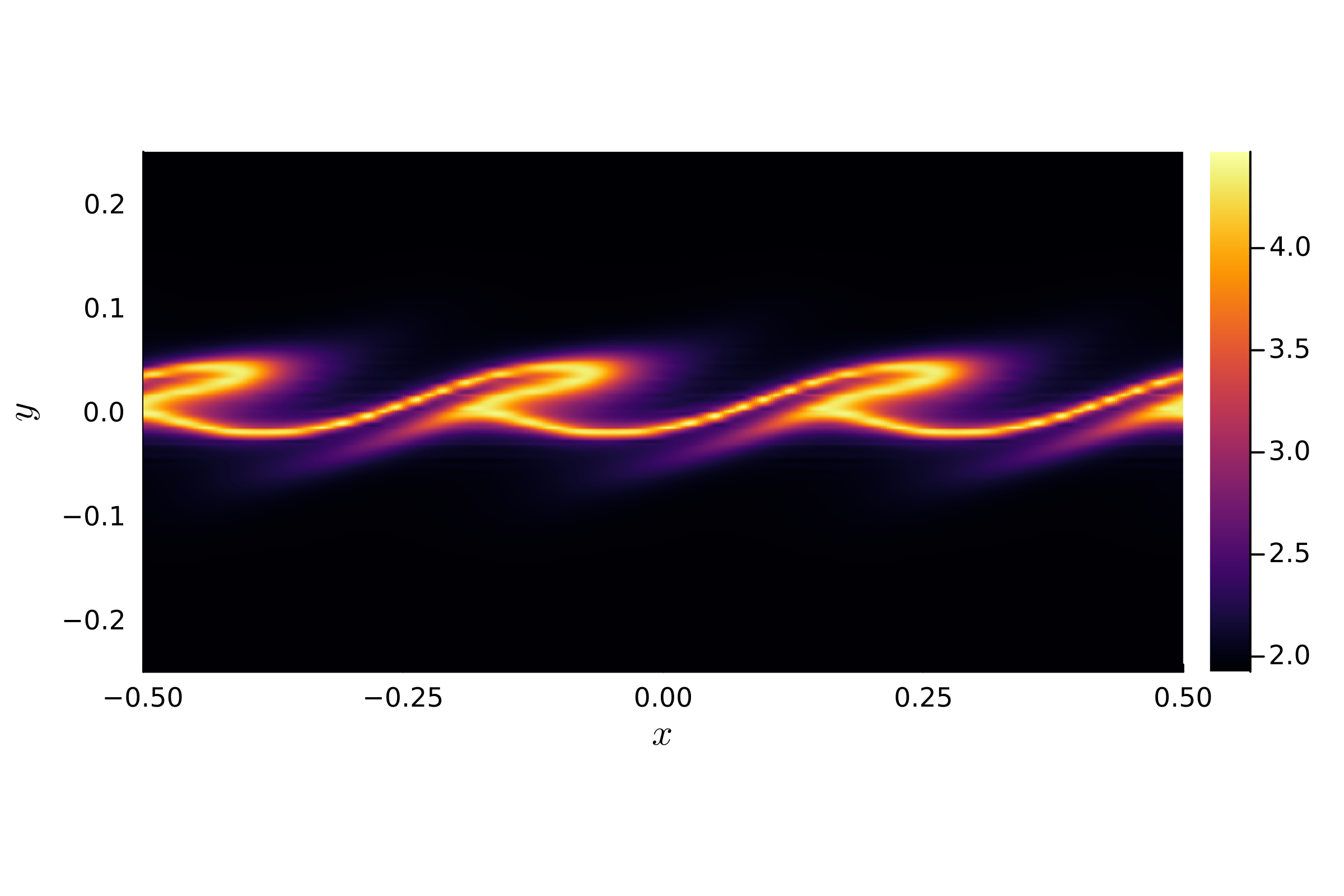}}
\caption{Relative density $\rho / \rho_c$ and adiabatic index $\gamma = c_p / c_v$ at final time $2t_c$ using degree $N=3$ polynomials, APEC flux differencing, and entropy correction. }
\label{fig:2t_c}
\end{figure}

We consider a structured mesh with $64 \times 32$ elements and degree $N=3$ polynomials, as well as a $32 \times 16$ mesh with degree $N=7$ polynomials, such that the total number of degrees of freedom is identical between both meshes. CFL values of $0.9$ and $0.6$ were used for $N=3$ and $N=7$, respectively. Figure~\ref{fig:2t_c} plots the density normalized by the critical specific volume $V_c = \rho_c^{-1} \approx 0.00338 \mathrm{\,m^{3}\, kg^{-1}}$. For these cases, APEC flux differencing with and without entropy correction are virtually identical. This is further supported by the small value of the entropy correction blending coefficient $\theta$. For $N=3$ on a $64\times 32$ mesh, the largest value of $\theta$ at final time $2t_c$ is $\theta \approx 1.194\times 10^{-4}$; for $N=7$ on a $32\times 16$ mesh, the largest value of $\theta$ is $\theta \approx 5.843\times 10^{-5}$. 

\begin{figure}
\centering
\subfloat[$\rho / \rho_c$]{\includegraphics[width=0.49\textwidth, trim={0 8cm 0 8cm}, clip]{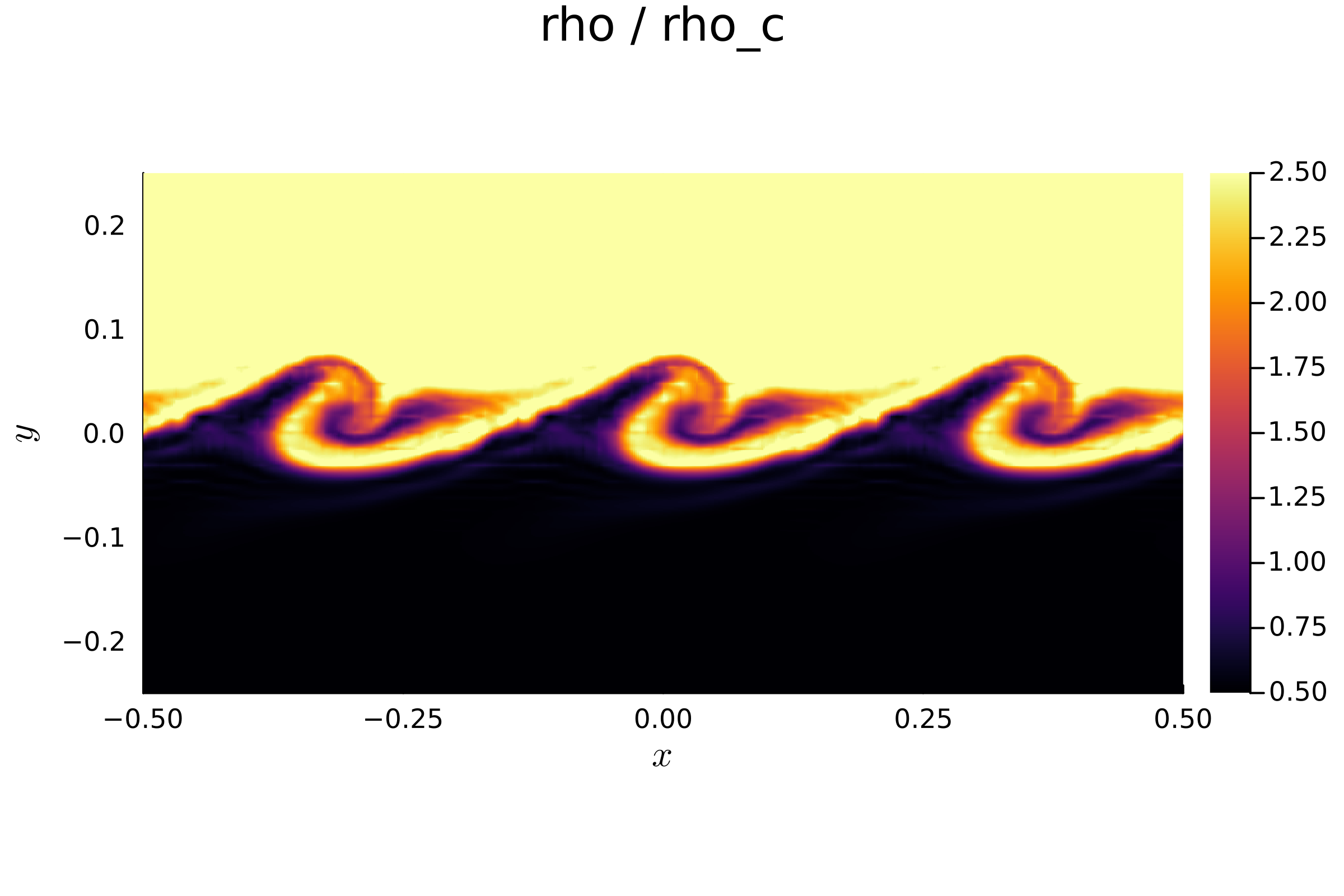}}
\subfloat[$\gamma$]{\includegraphics[width=0.49\textwidth, trim={0 8cm 0 8cm}, clip]{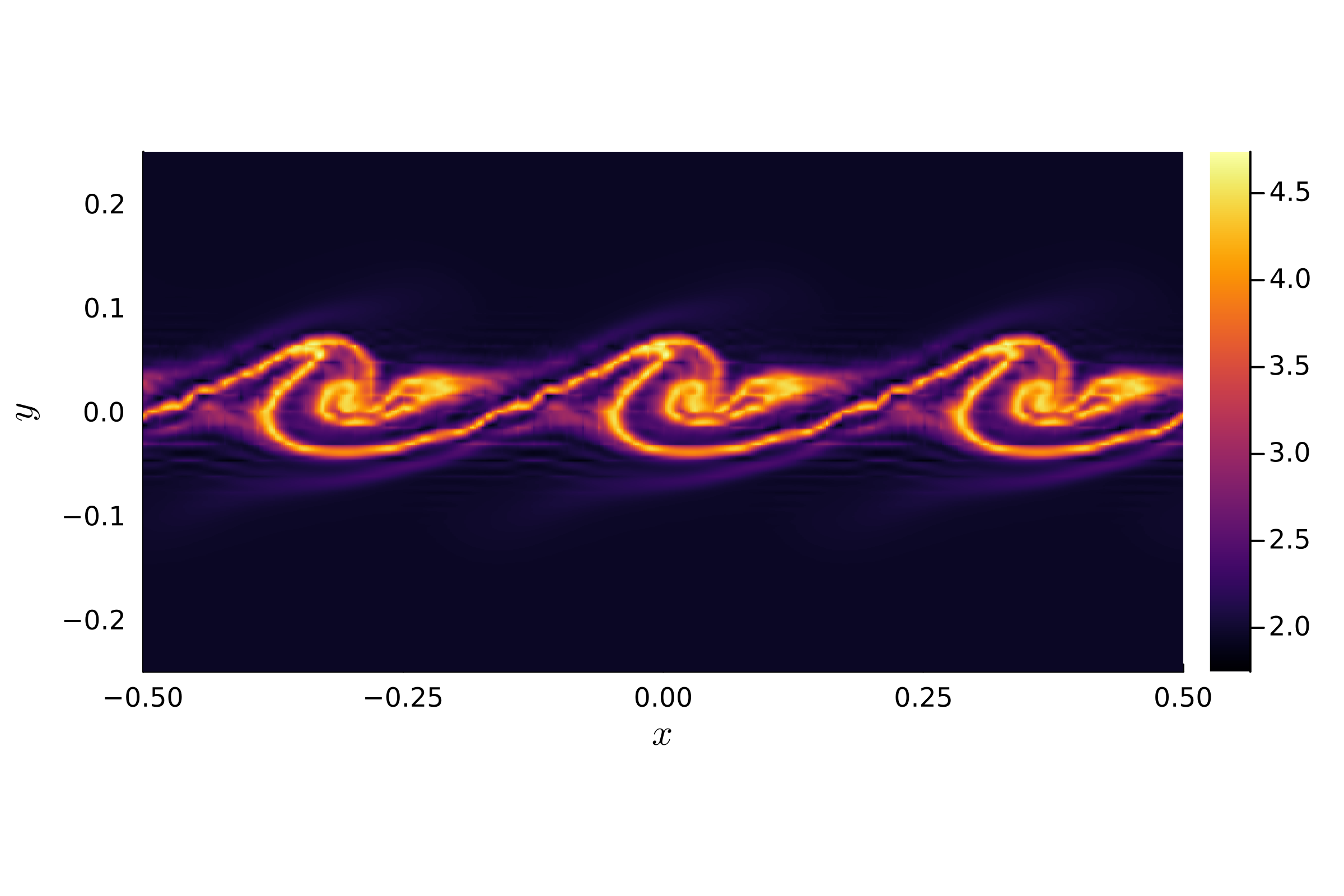}}
\caption{Relative density $\rho / \rho_c$ and adiabatic index $\gamma = c_p / c_v$ at final time $4t_c$ using degree $N=3$ polynomials, APEC flux differencing, and entropy correction. }
\label{fig:4t_c}
\end{figure}

Additionally, to highlight the effect of the entropy correction, we ran the transcritical mixing problem to a longer final time $4 t_c = 0.132$. Without entropy correction, APEC flux differencing with both $N=3$ or $N=7$ did not finish and diverged around time $t \approx 0.11$. Figure~\ref{fig:4t_c} shows the adiabatic index $\gamma = c_p / c_v$ at final time $4 t_c$ for APEC flux differencing with entropy correction. The solution remains stable despite the density showing visible artifacts due to under-resolution.

\subsection{Two-dimensional inviscid transcritical jet injection}
\label{sec:transcritical_jet}

We conclude with an inviscid transcritical jet injection example adapted from \cite{foll2019use} and \cite{ma2017entropy}. Let $h = 2.2 \times 10^{-3}\,\mathrm{m}$ denote the nozzle width. The domain is a rectangular channel given by
\[
\Omega = \LRs{0,\,32h} \times \LRs{-8h,\,8h}
= \LRs{0,\,7.04 \times 10^{-2}} \times \LRs{-1.76 \times 10^{-2},\,1.76 \times 10^{-2}}\,\mathrm{m}.
\]
Periodic boundary conditions are imposed in the $y$-direction, and slip-wall (mirror) boundary conditions are imposed on the right boundary. The initial condition is a quiescent gas 
\[
\rho = 45\,\mathrm{kg\,m^{-3}}, \qquad v_1 = 0, \qquad v_2 = 0, \qquad T = 290.2\,\mathrm{K}.
\]
At the left boundary $x=0$, the jet inflow is prescribed as
\[
\LRp{\rho, v_1, v_2, T} =
\begin{cases}
\LRp{500\,\mathrm{kg\,m^{-3}},\,100\,\mathrm{m\,s^{-1}},\, 10^{-4}\,\mathrm{m\,s^{-1}},\,124.6\,\mathrm{K}}
& \LRb{y} \leq h/2, \\[4pt]
\LRp{45\,\mathrm{kg\,m^{-3}},\,0,\,0,\,290.2\,\mathrm{K}}
& \text{otherwise}.
\end{cases}
\]
where the small $y$-velocity perturbation is arbitrarily chosen to break symmetry. 
The simulation is run to final time $1 \times 10^{-3}\,\mathrm{s}$.


\begin{figure}
\centering
\subfloat[Density $\rho$]{\includegraphics[width=.5\textwidth, trim={3cm 0 3cm 3cm}, clip]{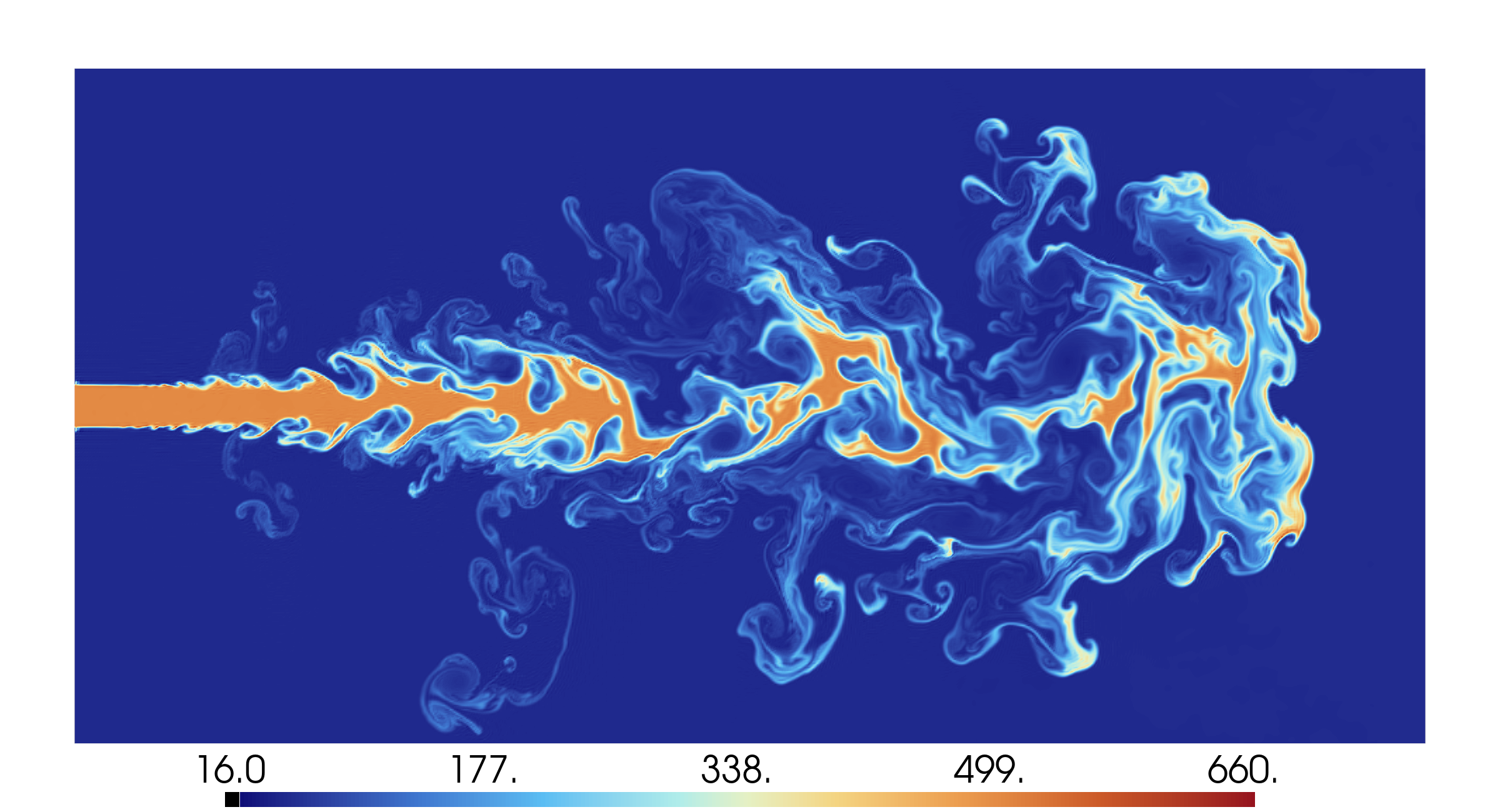}}
\subfloat[Pressure $p$]{\includegraphics[width=.5\textwidth, trim={3cm 0 3cm 3cm}, clip]{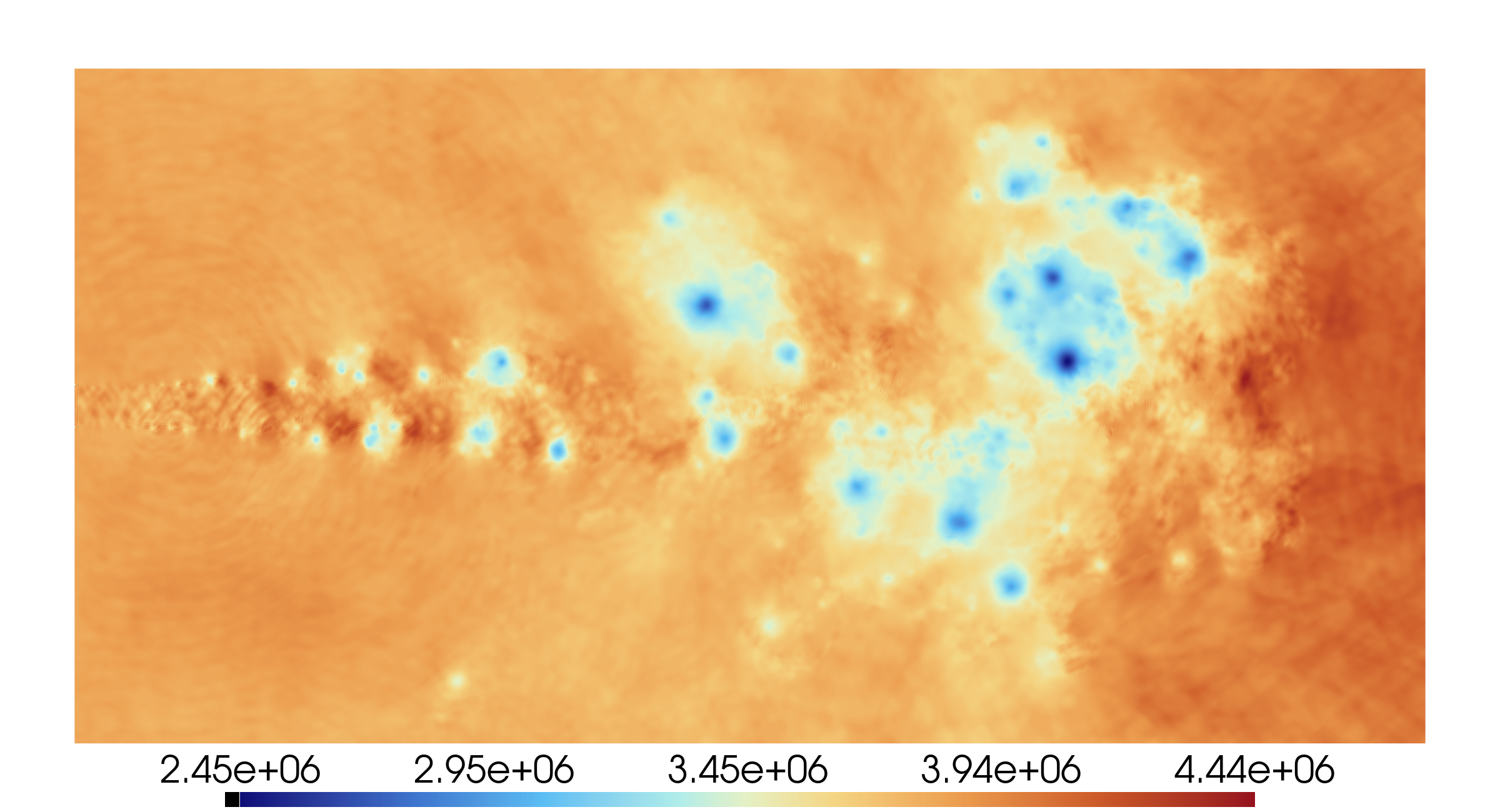}}\\
\subfloat[Adiabatic index $\gamma$]{\includegraphics[width=.5\textwidth, trim={3cm 0 3cm 3cm}, clip]{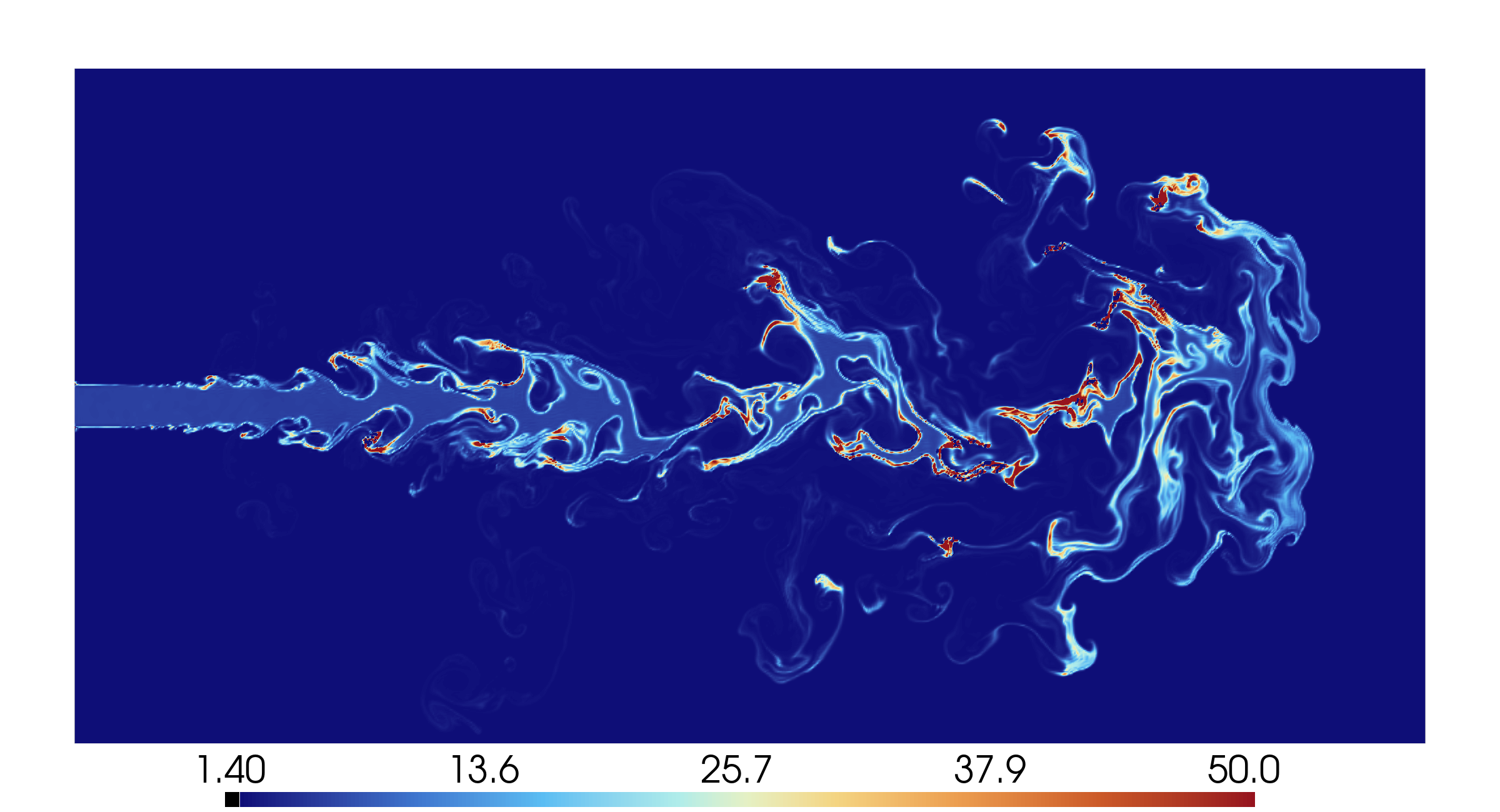}}
\subfloat[Blending coefficient $\theta$]{\includegraphics[width=.5\textwidth, trim={3cm 0 3cm 3cm}, clip]{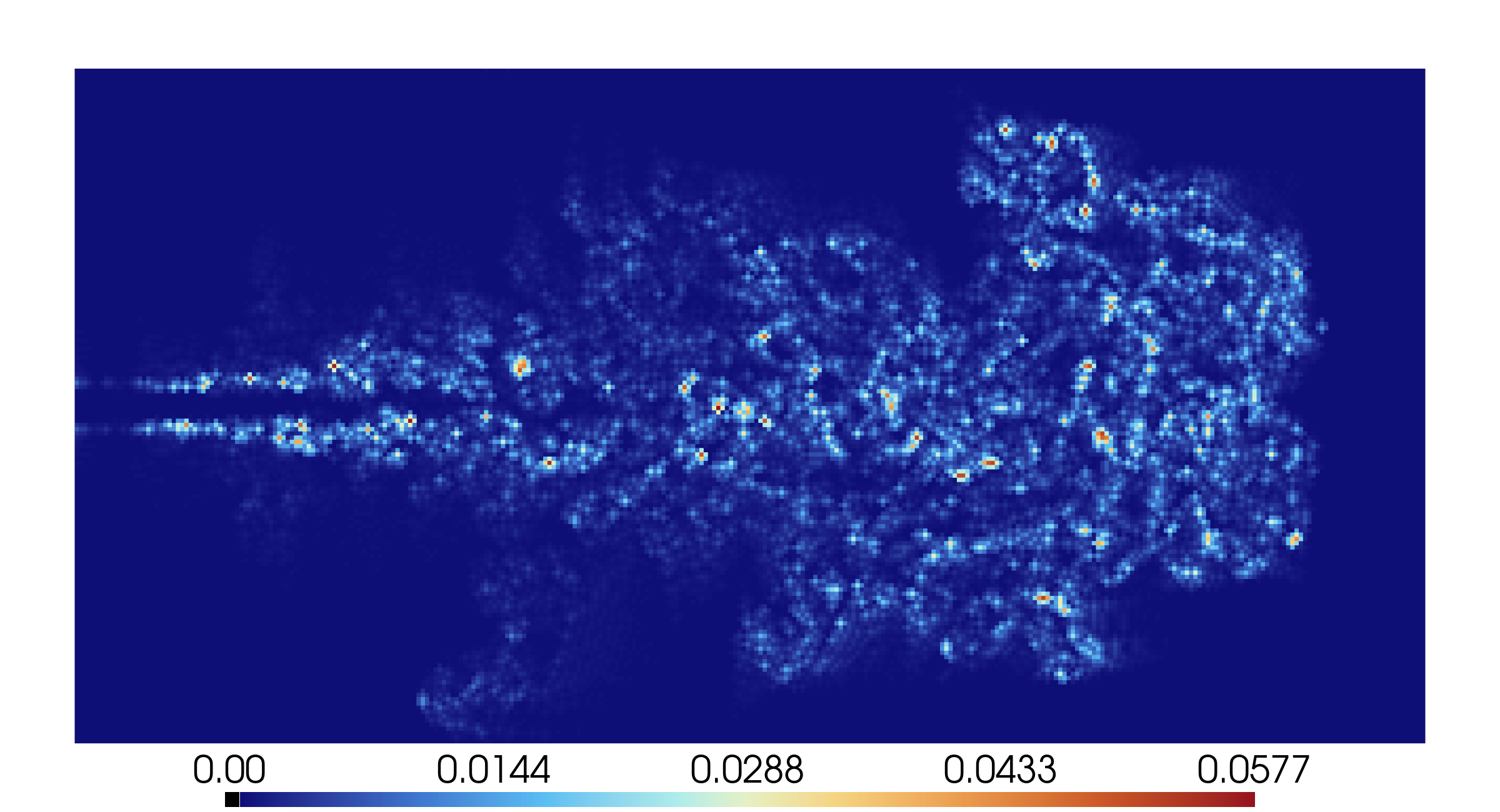}}
\caption{Density, pressure, ratio of specific heats $\gamma$, and entropy correction blending coefficient for the transcritical jet at final time $10^{-3}$ on a $320\times 160$ mesh with degree $N=3$ polynomials.}
\label{fig:jet}
\end{figure}

We note that for this problem, entropy correction alone diverged with negative density and temperature. However, we observe that this is remedied by ``smoothing'' the entropy correction factors as done in \cite{hennemann2021provably}. This smoothing step postprocesses the blending coefficient so that on each element $D^k$, $\theta$ is at least half of the value of the blending coefficient on all neighbors across element faces, which we denote by $N(k)$:
\[
\theta \longleftarrow \max_{j \in N(k)} \LRp{\theta, \frac{1}{2}\theta_j}.
\]
Finally, we arbitrarily scale the entropy correction factor $\theta$ by a factor of 1.5. While the simulation runs without this modification, we observe that scaling the correction factor $\theta \longleftarrow \min(1, 1.5 \theta)$ results in a $5-10\times$ larger time-step size. This scaling and smoothing likely provides a weak shock capturing effect which avoids near-inadmissible solution states; future work will investigate non-heuristic approaches for preserving admissibility. 

Figure~\ref{fig:jet} shows the density, pressure, ratio of specific heats $\gamma$, and blending coefficient $\theta$ at the final time under a CFL of $0.8$. A uniform quadrilateral mesh of $320 \times 160$ elements and degree $N=3$ polynomials is used, which corresponds to roughly $0.8$ million degrees of freedom and is comparable to the resolution used for the jet injection problem in \cite{foll2019use}. We observe that $\gamma$ deviates strongly from the ideal gas value of $1.4$, and that the entropy correction blending coefficient $\theta$ is both relatively small in magnitude and activated sparsely. 

\begin{figure}
\centering
\subfloat[Density $\rho$]{\includegraphics[width=.5\textwidth, trim={3cm 0 3cm 3cm}, clip]{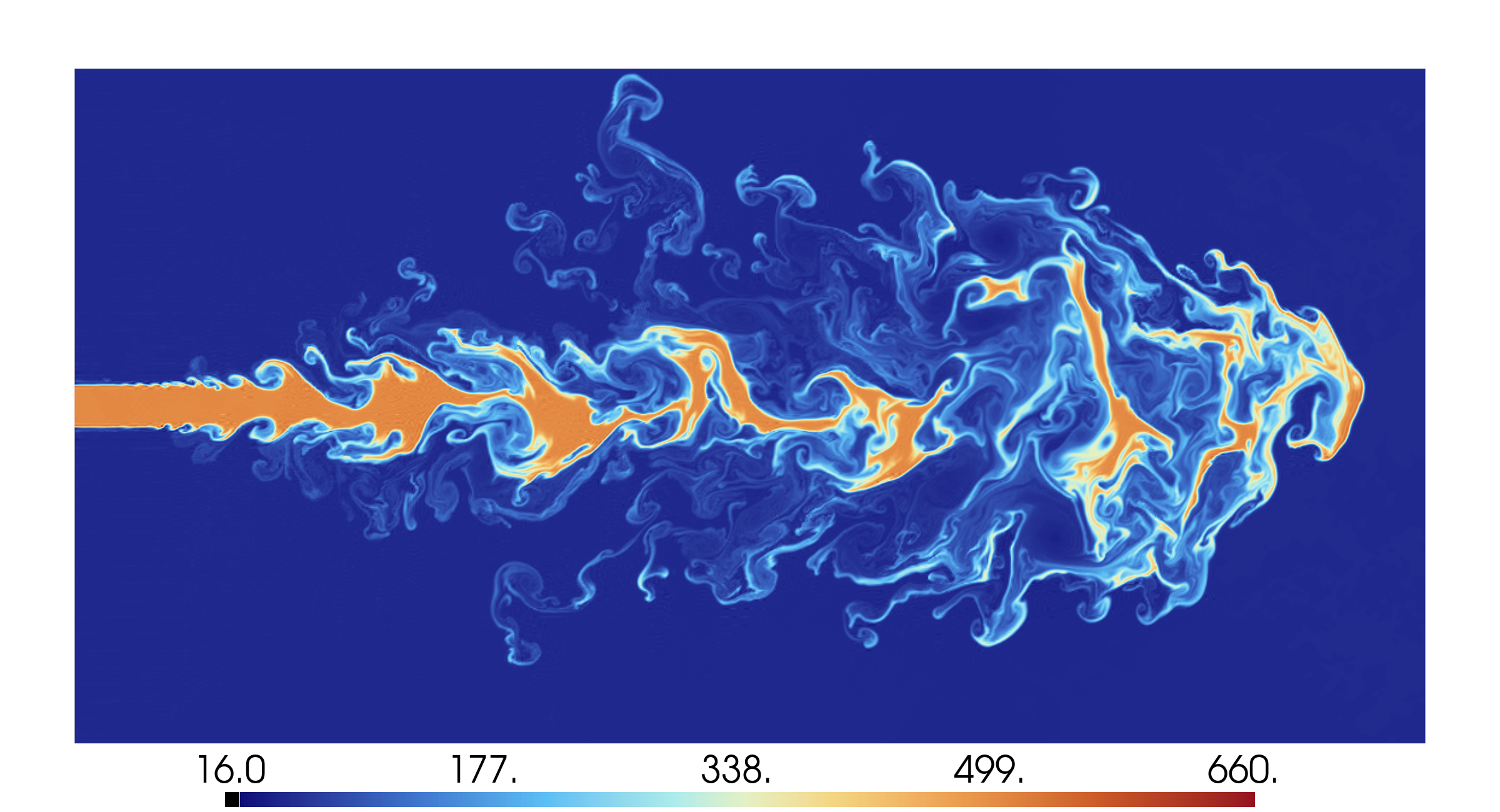}}
\subfloat[Pressure $p$]{\includegraphics[width=.5\textwidth, trim={3cm 0 3cm 3cm}, clip]{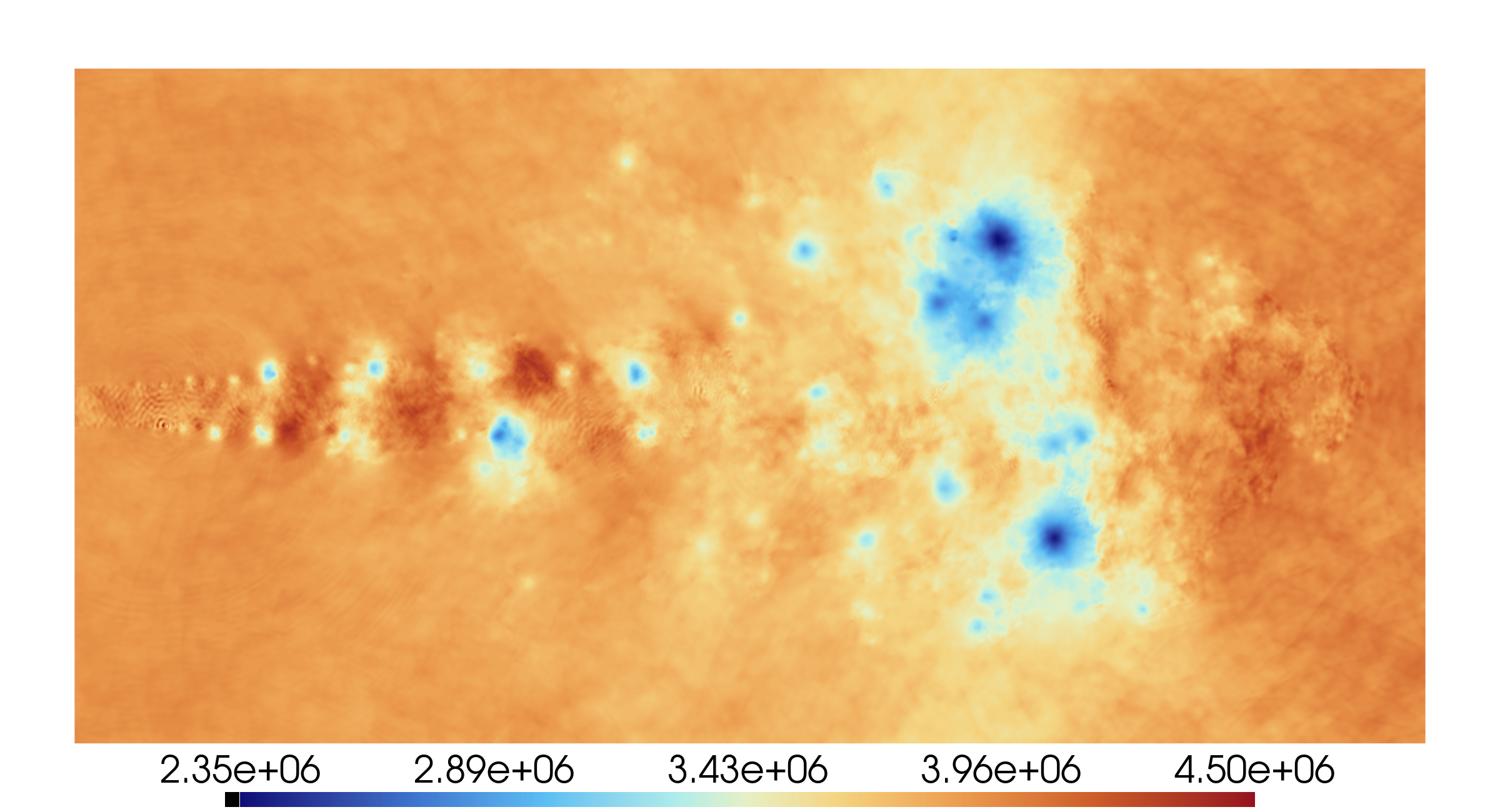}}\\
\subfloat[Adiabatic index $\gamma$]{\includegraphics[width=.5\textwidth, trim={3cm 0 3cm 3cm}, clip]{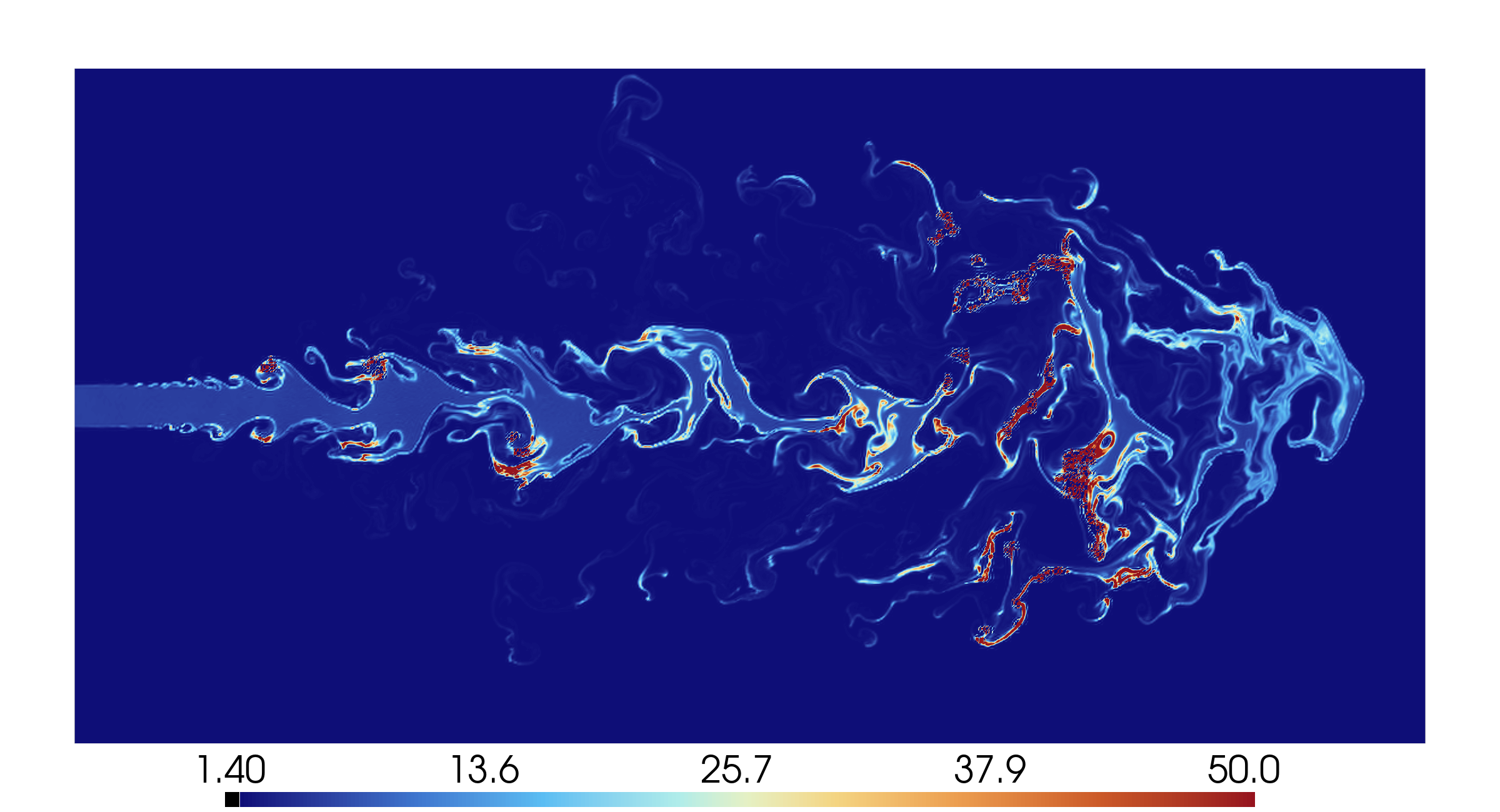}}
\subfloat[Blending coefficient $\theta$]{\includegraphics[width=.5\textwidth, trim={3cm 0 3cm 3cm}, clip]{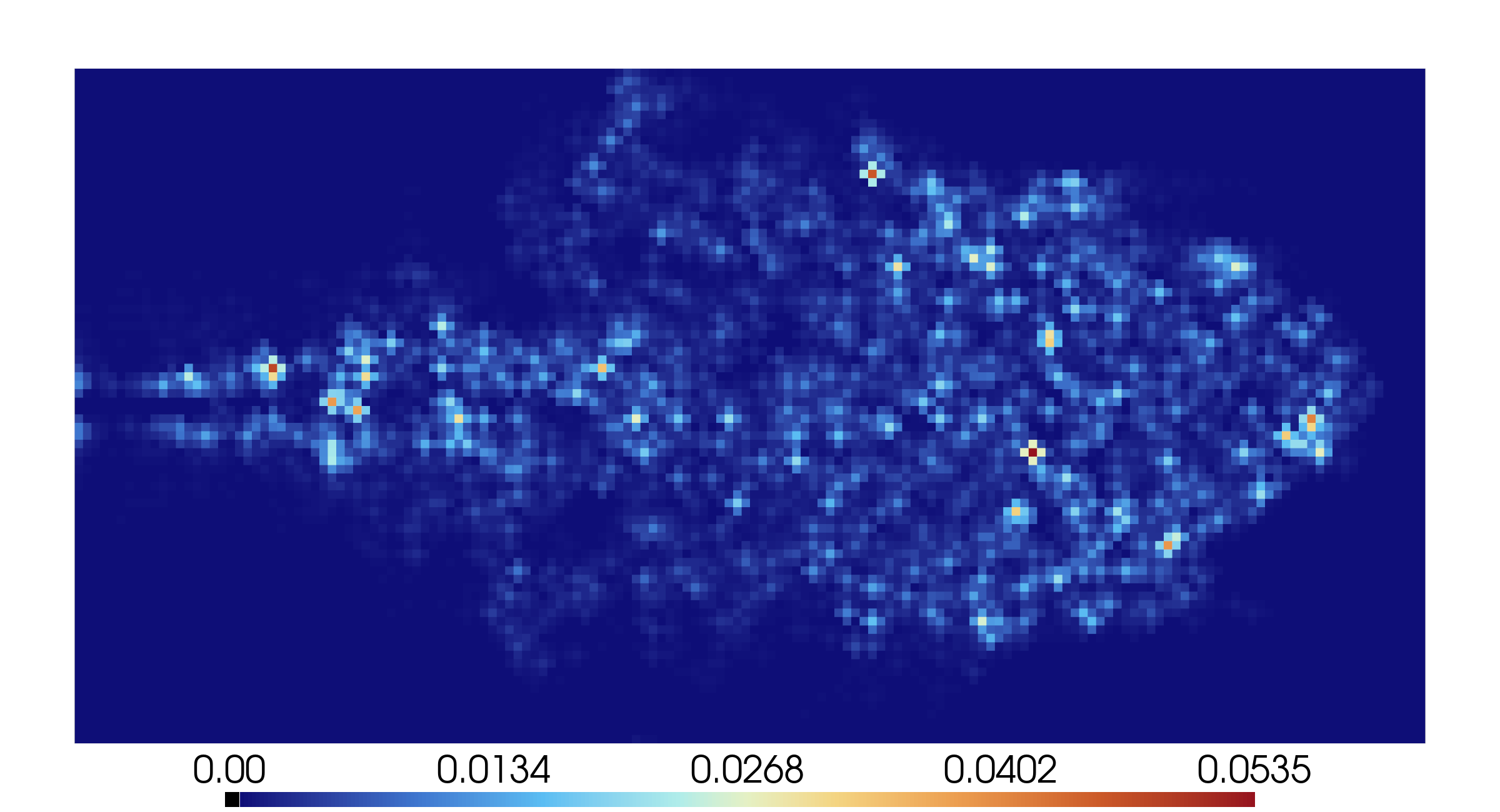}}
\caption{Density, pressure, ratio of specific heats $\gamma$, and entropy correction blending coefficient for the transcritical jet at final time $10^{-3}$ on a $160\times 80$ mesh with degree $N=7$ polynomials.}
\label{fig:jet_polydeg7}
\end{figure}

Finally, to check robustness of this strategy for higher polynomial degrees, Figure~\ref{fig:jet_polydeg7} shows the same quantities for degree $N=7$ polynomials on a coarser mesh of $160\times 80$ elements, such that the number of degrees of freedom is the same as the $N=3$ case.  For visualization purposes, the color scale for density is the same as in Figure~\ref{fig:jet}. The solutions differ due to sensitivity of the Euler equations \cite{fjordholm2017construction}, but similar behavior is observed for both $N=3$ and $N=7$.

\section{Conclusion}

In this work, we investigate structure-preserving techniques for real gas settings. These include exactly and approximately pressure equilibrium preserving discretizations, as well as FCT-like entropy correction techniques applicable to general non-ideal EOS. Pressure equilibrium conserving formulations reduce spurious oscillations and distortions near pressure equilibria, while the enforcement of a cell entropy inequality is especially helpful for improving robustness for long-time simulations, higher polynomial degrees, and under-resolved solution features. 

The combination of these techniques is uniquely well-suited to high order DG discretizations. First, pressure equilibrium preservation can be straightforwardly extended to the DG setting by combining APEC fluxes and flux differencing nodal DG formulations. Second, pressure equilibrium violations due to interface dissipation are small at higher orders of approximation. This enables the stabilization of under-resolved solutions by introducing DG interface penalty terms while preserving advantages of approximate pressure equilibrium conserving formulations. 

Because the entropy correction used in this work is minimally dissipative, it does not introduce large pressure equilibrium errors. However, small pressure equilibrium errors are still introduced through this correction, and future work will explore entropy dissipative corrections designed to preserve pressure equilibria. Finally, we note that to the authors' knowledge, the design of positivity preserving limiters which respect pressure equilibria for non-ideal EOS remains an open problem.

\section{Reproducibility}

A reproducibility repository containing Julia codes used to generate the results presented in this paper is available online \cite{chan2026nodalRepro}.

\section{Acknowledgements}

The authors thank Daniel Doehring for code reviews in Trixi.jl related to this work, as well as Gennaro Coppola, Alessandro Aiello, and Carlo De Michele for helpful discussions. 
JC gratefully acknowledges support from the National Science Foundation under award DMS-2607830. 
RP gratefully acknowledges support from the National Science Foundation under award NSF-GRFP-DGE-2137420. JL acknowledges the support by the Deutsche Forschungsgemeinschaft (DFG) within the Research Training Group GRK 2583 ``Modeling, Simulation and Optimization of Fluid Dynamic Applications." 
HR was supported by the Deutsche Forschungsgemeinschaft
(DFG, German Research Foundation, project number 528753982
as well as within the DFG priority program SPP~2410 with project number 526031774).
EC acknowledges the Office of Naval Research through the Naval Research Laboratory 6.1 Computational Physics Task Area.
AE acknowledges the Air Force Office of Scientific Research (program officers: Drs. Chiping Li, Fariba Farhoo, and Justin Koo) under contract No. 25RQCOR004, as well as Amentum under contract No. FA9300-20-F-9801.

The authors acknowledge the Texas Advanced Computing Center (TACC) at The University of Texas at Austin for providing computational resources that have contributed to the research results reported within this paper. 
Finally, the authors acknowledge the use of AI tools such as Cursor in the data analysis and preparation of figures for this manuscript. 

\appendix

\section{Thermodynamic identities}

The implementations in this work require the following thermodynamic quantities:
\begin{enumerate}
\item Pressure $p(V, T)$ and partial derivatives $\LRl{\pd{p}{V}}_T, \LRl{\pd{p}{T}}_V$
\item Internal energy $e(V, T)$ 
\item Specific entropy $s(V, T)$
\item Speed of sound $c(V, T)$ for a Lax-Friedrichs interface flux.
\end{enumerate}
These can then be used to calculate the Gibbs free energy $g = h - T s$ and enthalpy $h = e + p V$, as well as the mapping between conservative and entropy variables and the entropy Hessian inverse $\LRp{\pdn{S}{\bm{u}}{2}}^{-1}$ (or entropy variable Jacobian $\pd{\bm{u}}{\bm{v}}$).

\begin{remark}
Temperature $T(V, e)$ is evaluated in terms of specific volume and internal energy. This is implemented by inverting the internal energy relation $e(V, T)$ at fixed $V$ using a Newton solver, which is guaranteed to converge for admissible solution states such that $\LRp{\pd{e}{T}}_V = c_v > 0$. 
\end{remark}

%
%
%

\section{Entropy correction via artificial viscosity}

Another entropy correction approach is based on artificial viscosity (AV) \cite{chan2025artificial}, which adds a Laplacian artificial viscosity term to \eqref{eq:dgform} with an artificial viscosity coefficient calculated based on $\delta_k(\bm{u}_h)$ in \eqref{eq:vol_entropy_residual}. We note that ECAV is closely related to the local entropy correction terms proposed in \cite{abgrall2018general, abgrall2022reinterpretation} and extended in \cite{edoh2024conservative, picklo2025entropy, gaburro2023high, mantri2024fully} and other papers. ECAV is also related to recent versions of entropy viscosity which add regularization proportional to the violation of the chain rule \eqref{eq:entropychainrule} \cite{guermond2019invariant}. 

The implementation of entropy correction using AV utilizes the inverse Hessian of the scalar entropy to discretize the Laplacian $\Delta \bm{u}$ as $\nabla \cdot \LRp{\pd{\bm{u}}{\bm{v}} \nabla \bm{v}(\bm{u})}$, where 
\[
\pd{\bm{u}}{\bm{v}} = \LRp{\pd{\bm{v}}{\bm{u}}}^{-1} = \LRp{\pdn{S}{\bm{u}}{2}}^{-1}.
\]
It is possible to calculate this analytically for a general EOS by factoring this into the product of Jacobian matrices with respect to primitive variables $\bm{q} = \LRs{V, v_1, v_2, T}$
\[
\pd{\bm{v}}{\bm{u}} = \pd{\bm{v}}{\bm{q}} \pd{\bm{q}}{\bm{u}}.
\]
Because the last entropy variable is $-T^{-1}$, this choice of variables implies that $\pd{\bm{v}}{\bm{q}}, \pd{\bm{q}}{\bm{u}}$ are lower and upper triangular, respectively. The inverse of $\pd{\bm{v}}{\bm{u}}$ can now be calculated as
\[
\pd{\bm{u}}{\bm{v}}  = \LRp{\pd{\bm{v}}{\bm{u}}}^{-1} = \pd{\bm{q}}{\bm{u}}^{-1}\pd{\bm{v}}{\bm{q}}^{-1} = \pd{\bm{u}}{\bm{q}} \pd{\bm{q}}{\bm{v}}.
\]
Moreover, the entries of $\pd{\bm{u}}{\bm{q}}, \pd{\bm{q}}{\bm{v}}$ can be calculated explicitly using fundamental thermodynamic relations and derivatives of pressure with respect to $V, T$:
\[
\pd{\bm{u}}{\bm{q}}
=
\begin{bmatrix}
-\rho^2	&			&			& \\
-v_1 \rho^2	& \rho		&			& \\
-v_2 \rho^2	& 0			& \rho		& \\
\pd{\rho e}{V}	& \rho v_1	& \rho v_2	& \pd{\rho e}{T}
\end{bmatrix},
\qquad
\pd{\bm{q}}{\bm{v}}
=
\begin{bmatrix}
a		& a{v_1}	& a{v_2}	& ab \\
		& T			& 0			& v_1 T \\
		&			& T			& v_2 T \\
		&			&			& T^2
\end{bmatrix}.
\]
where $a = \LRp{\LRl{\pd{(g/T)}{V}}_T}^{-1}$ and $b = \frac{1}{2}\LRp{v_1^2 + v_2^2} - T^2\LRl{\pd{(g/T)}{T}}_V$.
Additional thermodynamic quantities used to calculate these matrices can be expressed in terms of expressions involving only derivatives of pressure and internal energy:
\begin{gather*}
\pd{\rho e}{V}  = -\rho^2 e -  \frac{1}{2}\rho^2\LRp{v_1^2 + v_2^2} + \rho\LRl{\pd{e}{V}}_T, \qquad \pd{\rho e}{T}  = \rho \LRl{\pd{e}{T}}_V, \\
\LRl{\pd{e}{V}}_T = T  \LRl{\pd{p}{T}}_V - p(V, T), \qquad \LRl{\pd{e}{T}}_V = c_v,
\\
\LRl{\pd{(g/T)}{V}}_T = \frac{V}{T} \LRl{\pd{p}{V}}_T, \qquad \LRl{\pd{(g/T)}{T}}_V = \frac{V  T  \LRl{\pd{p}{T}}_V - h}{T^2},
\end{gather*}
where $h = e + p V$ is the enthalpy and $c_v$ is the specific heat at constant volume.

\begin{remark}
Because entropy correction AV directly uses the entropy Hessian, it can fail if the entropy is non-convex and the Hessian is singular. In contrast, a discrete graph viscosity is entropy dissipative for any convex entropy, for example ones known to exist only under mild hyperbolicity assumptions \cite{harten1998convex}. This can be avoided using the FCT-based entropy correction, or an AV formulation based on a discrete ``graph'' viscosity formulation as was done in \cite{christner2025entropyfd}.
\end{remark}

\subsection{Stiffened gas}
\label{sec:SG}

We examine the stiffened gas EOS, which is commonly used to model nearly incompressible fluids (e.g. water), and is frequently represented by the following relations
\begin{gather*}
    e(p, \rho) = \frac{p + \gamma p_{\infty}}{(\gamma - 1)\rho} + q, \qquad
    \rho(p, T) = \frac{p + p_\infty}{(\gamma - 1)c_vT}, \\
    s(p, \rho) = c_v\log\bigg(\frac{p + p_\infty}{\rho^\gamma}\bigg), \qquad c^2(p, \rho) = \frac{\gamma(p + p_\infty)}{\rho},
\end{gather*}
where $p_\infty, q$ are species-dependent parameters. These variables can all be expressed as functions of $V, T$ as shown below.
\begin{gather*}
    p(V, T) = \frac{(\gamma - 1)c_vT}{V} - p_\infty, \qquad e(V, T) = c_vT + p_\infty V + q \\
    s(V, T) = c_v\log\Big((\gamma - 1)c_vTV^{\gamma - 1}\Big), \qquad c^2(V, T) = (\gamma - 1)\gamma c_vT
\end{gather*}
As noted in Remark~\ref{remark:PEC_condition}, under a stiffened gas EOS, the PEC conditions \eqref{eq:equiv_pe_condition} and \eqref{eq:internal_energy_avg} are automatically satisfied for a standard DG formulation (e.g., a flux differencing DG formulation with central fluxes) and a local Lax-Friedrichs surface flux. Thus, pressure equilibrium errors are not an issue for this example. 

We examine a 1D quasi-compressible Euler system of equations governed by the stiffened gas (SG) EOS. As noted in Lemma 1 of \cite{chan2024high}, an entropy for the quasi-1D compressible Euler equations is simply the entropy of the standard 1D compressible Euler equations scaled by the nozzle width $a(x)$. Additionally, under this choice of entropy, the entropy variables for the quasi-1D system are identical to the entropy variables for the original 1D Euler equations. We use this to construct the Hessian matrix for ECAV. 

\begin{remark}
Since the quasi-1D Euler equations are a non-conservative system, the construction of a subcell finite volume scheme for the FCT-based entropy correction is more involved \cite{rueda2022subcell}. In contrast, the AV-based entropy correction is agnostic to whether a system is conservative or non-conservative. 
\end{remark}

The problem setup, boundary conditions, and SG parameters are detailed in \cite{berry2010quasinozzle, delchini2015quasinozzle}, and are reproduced here for convenience. We consider two different types of fluid (liquid water and steam), while the rest of the problem settings remain the same. The SG EOS parameters are listed below in Table \ref{table:sg_params}. The domain is $[0, 1]$ and the nozzle area $A(x) = 1 + 0.5\cos(2\pi x)$ is allowed to vary spatially but is constant in time. The initial condition is given by 
\[
p = (p_{\text{outflow}} - p_{\text{inflow}})x + p_{\text{inflow}}, \qquad v_1 = 0, \qquad T= 453,
\]
where $p_{\text{outflow}} = 0.5\text{ MPa}$ and $p_{\text{inflow}} = \text{1 MPa}$. The stagnation inlet and subsonic static outlet pressure condition are imposed at the left and right boundaries, respectively. The numerical solution is allowed to reach steady state by simulating up to time $T = 0.3$. Note that these experiments are intended only to illustrate the performance of ECAV; efficient exact solvers exist for the quasi-1D Euler equations with stiffened gas EOS \cite{yeom2019efficient} (see, for example, Appendix C in \cite{berry2010quasinozzle}). 
\begin{table}[h]
\centering
\begin{tabular}{lcccc}
\hline
Water &
$\gamma$ &
$q\;(\mathrm{J\,kg^{-1}})$ &
$p_{\infty}\;(\mathrm{Pa})$ &
$c_{v}\;(\mathrm{J\,kg^{-1}\,K^{-1}})$ \\
\hline
Liquid &
2.35 &
$-1167\times 10^{3}$ &
$10^{9}$ &
1816 \\
Steam &
1.43 &
$2030\times 10^{3}$ &
0 &
1040 \\
\hline
\end{tabular}
\caption{SG EOS parameters for 1D quasi-compressible Euler equations.}
\label{table:sg_params}
\end{table}
Figure \ref{fig:nozzle_water} and \ref{fig:nozzle_steam} show numerical solutions to the 1D nozzle flow problem. Each flow exhibits different behaviors due to differences in the stiffness parameter $p_\infty$. For water, the initial pressure ratio is small, and the flow remains subsonic. In contrast, for steam, the ratio is large enough for the system to develop a shock as the flow accelerates through the nozzle. 

Due to the initial condition and subsonic nature of the flow, the solution for liquid water is smooth and behaves similarly with and without ECAV. However, the simulation using steam parameters crashes for both degree $N = 3, 7$ approximations without entropy correction stabilization due to the presence of strong shocks. 
\begin{figure}[h!]
    \centering
    \vspace{1em} 
    \begin{subfigure}[b]{0.45\textwidth}
        \centering
        \includegraphics[width=\textwidth]{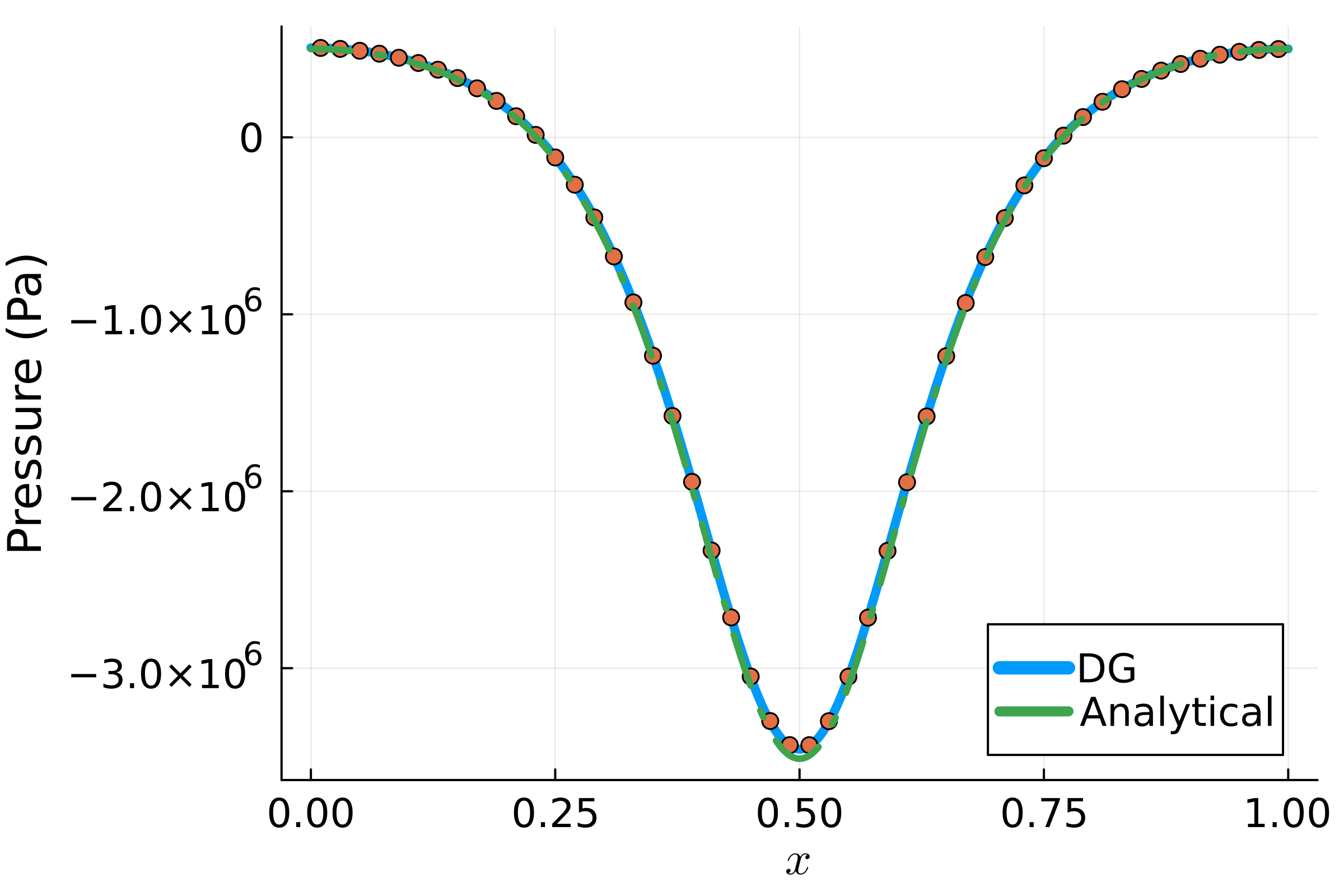}
        \caption{Pressure}
        \label{fig:5}
    \end{subfigure}
    \begin{subfigure}[b]{0.45\textwidth}
        \centering
        \includegraphics[width=\textwidth]{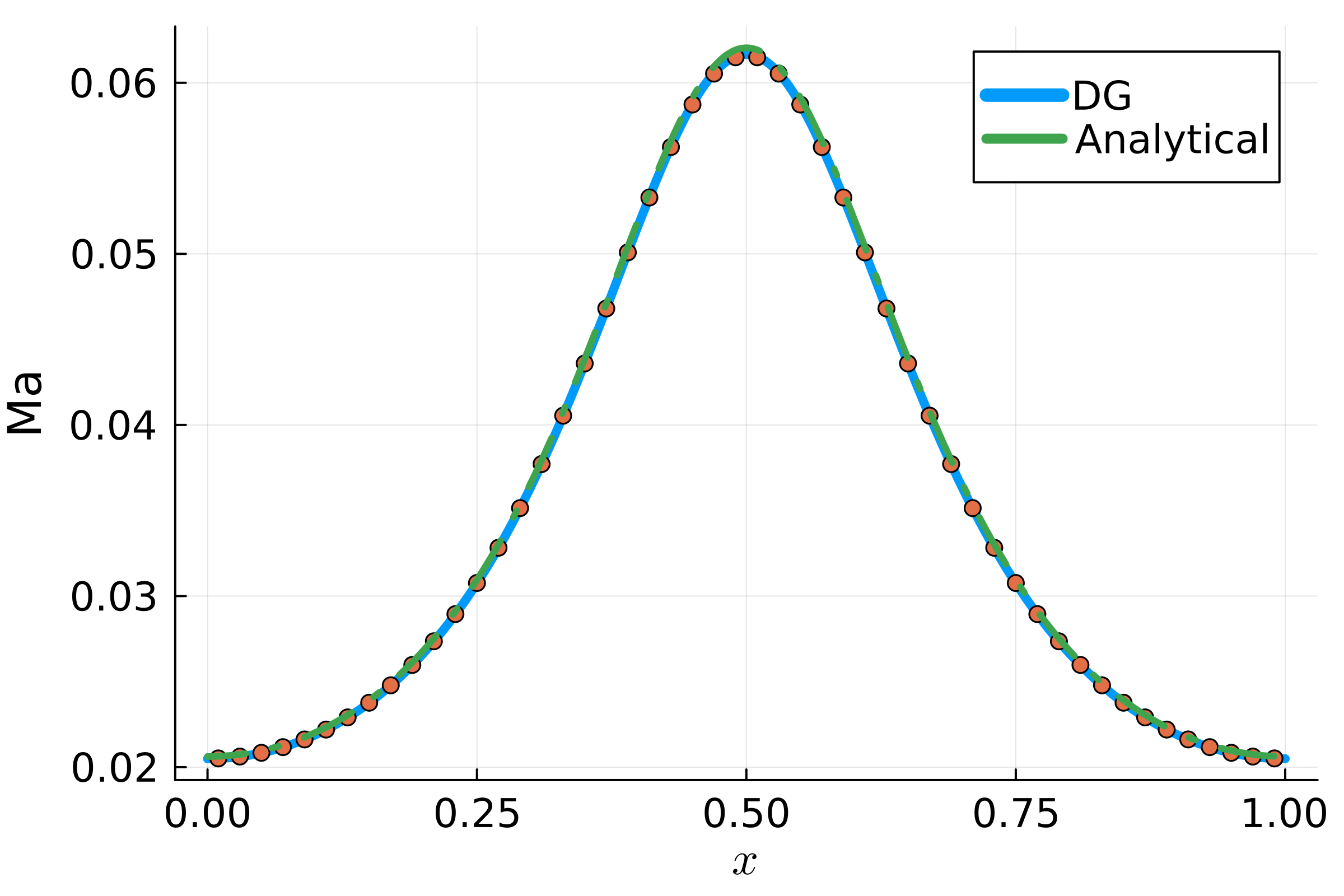}
        \caption{Mach number}
        \label{fig:6}
    \end{subfigure}

    \caption{Steady state solution for a liquid water flowing through 1D converging-diverging nozzle using degree $N = 7$ approximation and $K = 50$ elements. The $N=3$, $K=100$ element solution is omitted as it is visually indistinguishable from the $N=7$ solution. Circles denote DG cell averages. }
    \label{fig:nozzle_water}
\end{figure}

\begin{figure}[h!]
    \centering
    \begin{subfigure}[b]{0.45\textwidth}
        \centering
        \includegraphics[width=\textwidth]{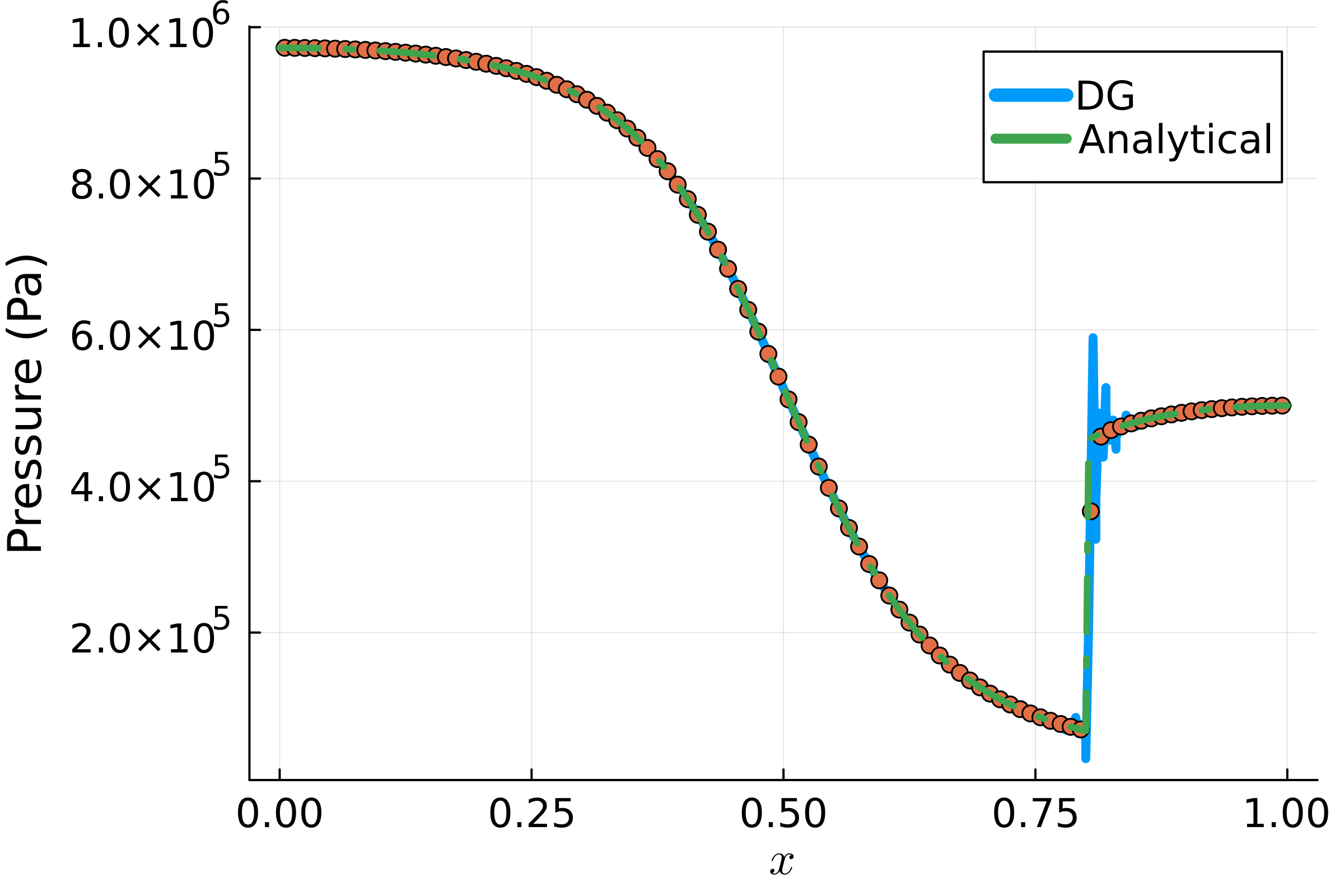}
        \caption{Pressure, $N = 3, K = 100$}
        \label{fig:2}
    \end{subfigure}
    \begin{subfigure}[b]{0.45\textwidth}
        \centering
        \includegraphics[width=\textwidth]{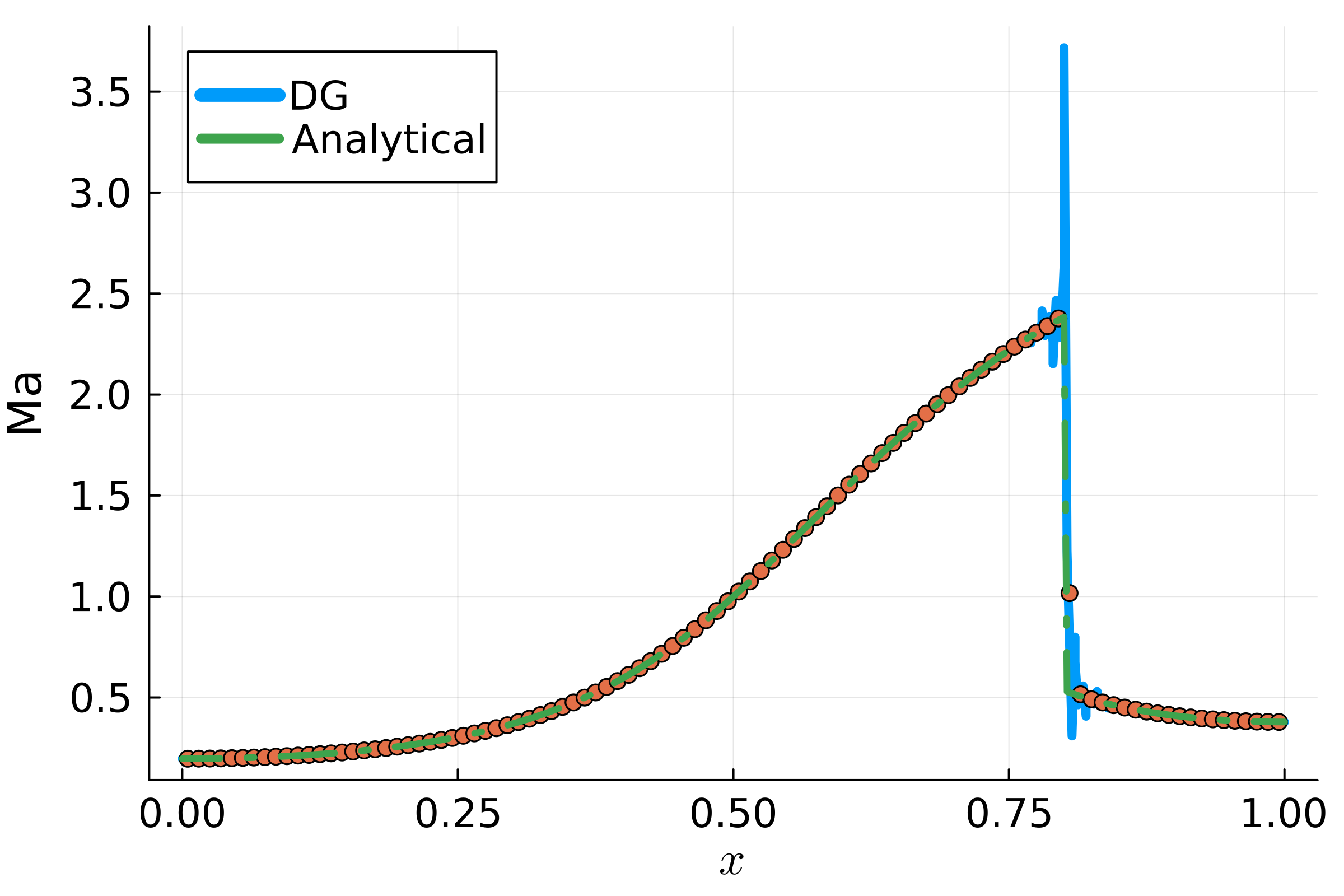}
        \caption{Mach number, $N = 3, K = 100$}
        \label{fig:3}
    \end{subfigure}
    \vspace{1em} 
    \begin{subfigure}[b]{0.45\textwidth}
        \centering
        \includegraphics[width=\textwidth]{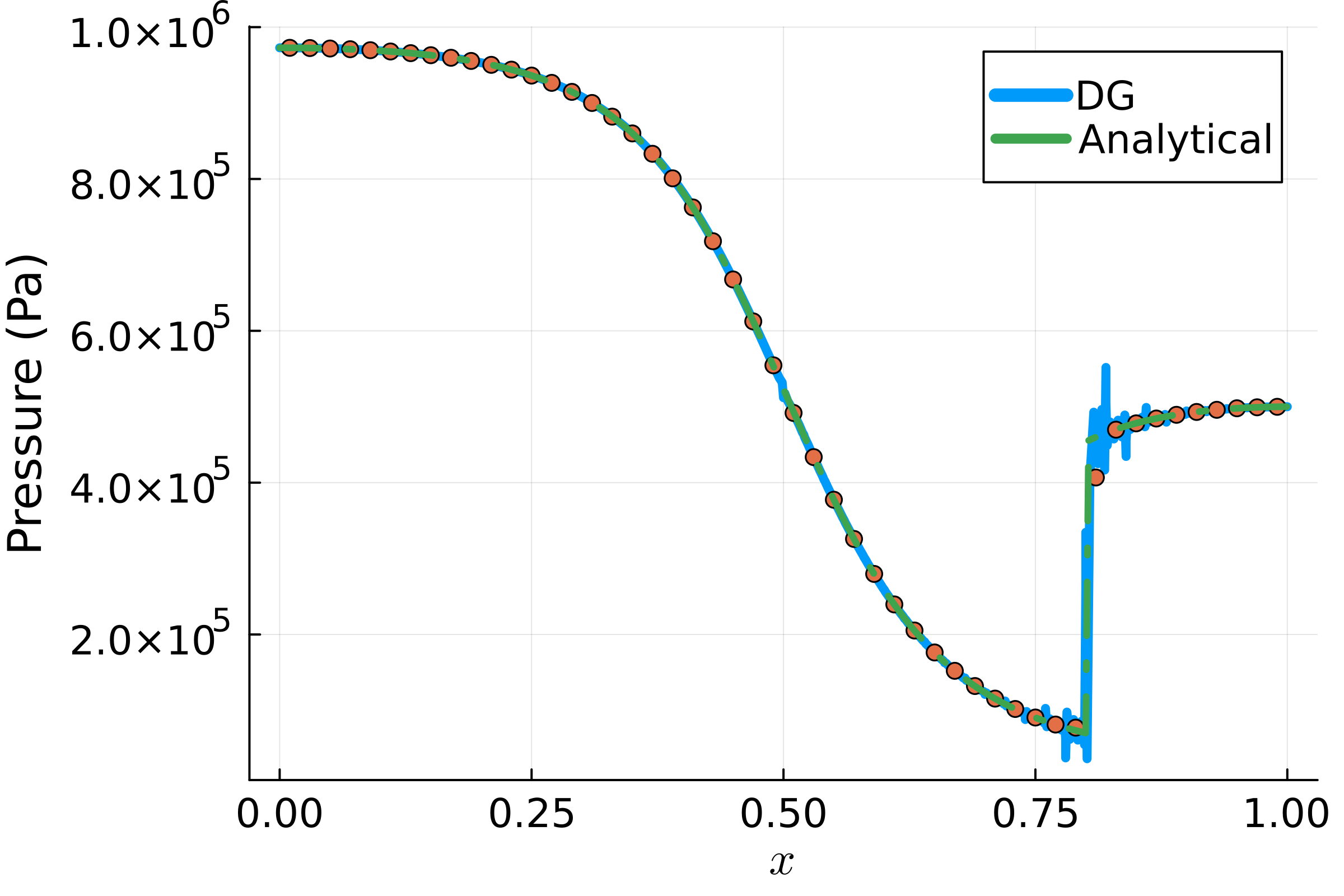}
        \caption{Pressure, $N = 7, K = 50$}
        \label{fig:5}
    \end{subfigure}
    \begin{subfigure}[b]{0.45\textwidth}
        \centering
        \includegraphics[width=\textwidth]{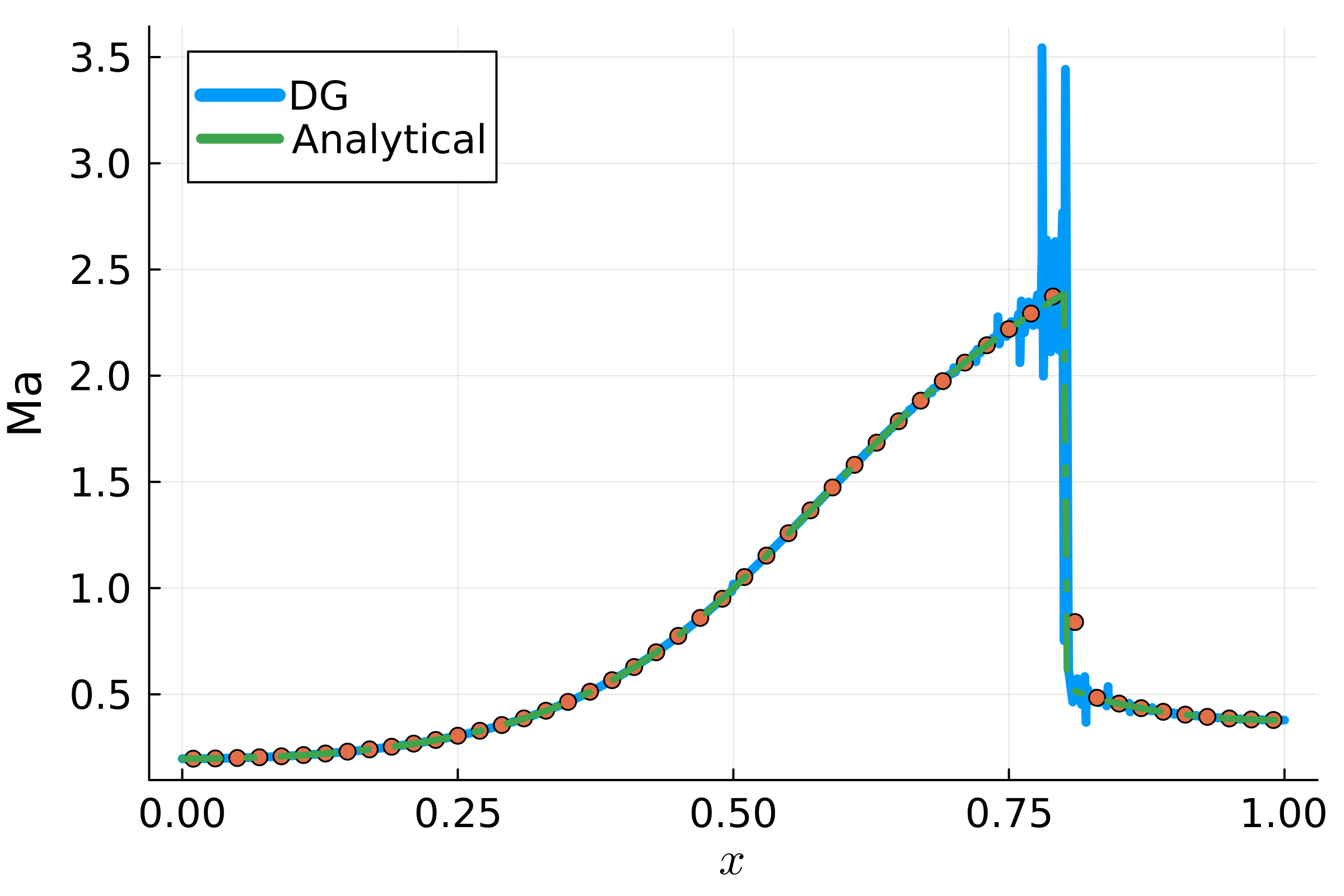}
        \caption{Mach number, $N = 7, K = 50$}
        \label{fig:6}
    \end{subfigure}

    \caption{Steady state solution for a vapor phase (steam) through a 1D converging-diverging nozzle using polynomial degree $N = 3$ and $100$ elements (first row) and degree $N = 7$ and $50$ elements (second row). Circles denote DG cell averages. }
    \label{fig:nozzle_steam}
\end{figure}

\bibliographystyle{plain}
\bibliography{reference.bib}

\end{document}